\pdfoutput=1		

\documentclass{article}  
\usepackage{commath}
\usepackage{graphicx}
\usepackage{enumitem}
\usepackage{tikz-cd}
\usepackage{subcaption}

\usepackage{amsmath}
\usepackage{amssymb}
\usepackage{amsthm}
\usepackage{bbm}
\usepackage{mathrsfs}

\let\emptyset\varnothing

\DeclareMathOperator{\Aut}{Aut}
\DeclareMathOperator{\CF}{CF}

\DeclareMathOperator{\CM}{CM}

\DeclareMathOperator{\Crit}{Crit}
\DeclareMathOperator{\CZ}{CZ}

\DeclareMathOperator{\Diff}{Diff}

\DeclareMathOperator{\ev}{ev}
\DeclareMathOperator{\Fix}{Fix}

\DeclareMathOperator{\grad}{grad}

\DeclareMathOperator{\Ham}{Ham}
\DeclareMathOperator{\Hess}{Hess}
\DeclareMathOperator{\HF}{HF}
\DeclareMathOperator{\HM}{HM}

\DeclareMathOperator{\id}{id}
\DeclareMathOperator{\im}{im}
\DeclareMathOperator{\ind}{ind}
\DeclareMathOperator{\Int}{Int}

\DeclareMathOperator{\loc}{loc}
\DeclareMathOperator{\Log}{Log}

\DeclareMathOperator{\M}{M}

\DeclareMathOperator{\ord}{ord}
\DeclareMathOperator{\pr}{pr}

\DeclareMathOperator{\RFC}{RFC}
\DeclareMathOperator{\RFH}{RFH}

\DeclareMathOperator{\SH}{SH}
\DeclareMathOperator{\sgn}{sgn}

\DeclareMathOperator{\Sp}{Sp}

\DeclareMathOperator{\Spec}{Spec}
\DeclareMathOperator{\supp}{supp}
\DeclareMathOperator{\Symp}{Symp}

\newcommand{\dR}{\mathrm{dR}}
\newcommand{\bld}[1]{\boldmath\textit{\textbf{#1}}\unboldmath}

\newtheoremstyle{main} 		             	 		
  	{}	                                     		
  	{}	                                    		
  	{\itshape}			                     		
  	{}        	                             		
  	{\boldmath\bfseries}   	                         		
  	{.}            	                        		
  	{ }           	                         		
  	{\thmname{#1}\thmnumber{ #2}\thmnote{ (#3)}}	
\theoremstyle{main}
\newtheorem{definition}{Definition}[section]
\newtheorem{proposition}{Proposition}[section]
\newtheorem{corollary}{Corollary}[section]
\newtheorem{theorem}{Theorem}[section]
\newtheorem{lemma}{Lemma}[section]
\newtheorem{conjecture}{Conjecture}[section]
\newtheoremstyle{nonit} 		             	 		
  	{}	                                     		
  	{}	                                    		
  	{}			                     		
  	{}        	                             		
	{\boldmath\bfseries}   	                         		
  	{.}            	                        		
  	{ }           	                         		
  	{\thmname{#1}\thmnumber{ #2}\thmnote{ (#3)}}	
\theoremstyle{nonit}
\newtheorem{remark}{Remark}[section]
\newtheorem{example}{Example}[section]
\newtheoremstyle{ex} 		             	 		
  	{}	                                     		
  	{}	                                    		
  	{\small}			                     		
  	{}        	                             		
  	{\bfseries\boldmath}   	                         		
  	{.}            	                        		
  	{ }           	                         		
  	{\thmname{#1}\thmnumber{ #2}\thmnote{ (#3)}}	
\theoremstyle{ex}

\usepackage[backend=bibtex, style=alphabetic]{biblatex}
\bibliography{bibliography}
\usepackage[bookmarksopen=true,bookmarksnumbered=true]{hyperref}
\hypersetup{
  colorlinks   = true, 
  urlcolor     = blue, 
  linkcolor    = blue, 
  citecolor    = blue 
}

\begin{document}

\title{Lectures on Twisted Rabinowitz--Floer Homology
}


\author{Yannis B\"ahni
}

\maketitle

\begin{abstract}
	Rabinowitz--Floer homology is the Morse--Bott homology in the sense of Floer associated with the Rabinowitz action functional introduced by Kai Cieliebak and Urs Frauenfelder in 2009. In this manuscript, we consider a generalisation of this theory to a Rabinowitz--Floer homology of a Liouville automorphism. As an application, we show the existence of noncontractible periodic Reeb orbits on quotients of symmetric star-shaped hypersurfaces. In particular, this theory applies to lens spaces. Moreover, we prove a forcing theorem, which guarantees the existence of a contractible twisted closed characteristic on a displaceable twisted stable hypersurface in a symplectically aspherical geometrically bounded symplectic manifold if there exists a contractible twisted closed characteristic belonging to a Morse-Bott component, with energy difference smaller or equal to the displacement energy of the displaceable hypersurface. 	
\end{abstract}

\tableofcontents

\newpage
\section{Introduction}
\label{sec:introduction}

The existence of closed Reeb orbits on lens spaces is important in the study of celestial mechanics. Indeed, by \cite[Corollary~5.7.5]{frauenfelderkoert:3bp:2018}, the Moser regularised energy hypersurface near the earth or the moon of the planar circular restricted three-body problem for energy values below the first critical value is diffeomorphic to the real projective space $\mathbb{RP}^3$. See also \cite[Introduction]{hryniewicz:cm:2016} for more details. Our main result will be the following.

\begin{theorem}[{\cite[Theorem~1.2]{baehni:rfh:2021}}]
	Let $\Sigma \subseteq \mathbb{C}^n$, $n \geq 2$, be a compact and connected star-shaped hypersurface invariant under the rotation	
	\begin{equation*}
		\varphi \colon \mathbb{C}^n \to \mathbb{C}^n, \quad \varphi(z_1,\dots,z_n) := \del[1]{e^{2\pi i k_1/m}z_1,\dots,e^{2\pi i k_n/m}z_n}
	\end{equation*}
	\noindent for some even $m \geq 2$ and $k_1,\dots,k_n \in \mathbb{Z}$ coprime to $m$. Then $\Sigma/\mathbb{Z}_m$ admits a noncontractible periodic Reeb orbit.
	\label{thm:my_result}
\end{theorem}

Theorem \ref{thm:my_result} has similarities with the following two recent results. 

\begin{theorem}[{\cite[Corollary~1.6~(iv)]{sandon:reeb:2020}}]
	\label{thm:lens_space}
	Any contact form on $\mathbb{S}^{2n - 1}/\mathbb{Z}_m$ defining the standard contact structure admits a closed Reeb orbit.	
\end{theorem}

Using the fact that there is a natural bijection between contact forms on the odd-dimensional sphere equipped with the standard contact structure and star-shaped hypersurfaces, Theorem \ref{thm:lens_space} is actually stronger than Theorem \ref{thm:my_result} in that it does not restrict the parity of the lens space. However, Theorem \ref{thm:lens_space} does not say anything about the topological nature of the Reeb orbit. The proof of this theorem uses a generalisation of Givental's nonlinear Maslov index to lens spaces. 

\begin{theorem}[{\cite[Theorem~1.2]{liuzhang:noncontractible:2021}}]
	\label{thm:multiplicity}
	Every dynamically convex star-shaped $C^3$-hypersurface $\Sigma \subseteq \mathbb{C}^n$, $n \geq 2$, satisfying $\Sigma = -\Sigma$ admits at least two symmetric geometrically distinct closed characteristics.  
\end{theorem}

Theorem \ref{thm:multiplicity} has the advantage of being a \emph{multiplicity result}, but in disadvantage requires the assumption that the hypersurface is dynamically convex and does only work for $\mathbb{Z}_2$-symmetry. To the authors knowledge, the first named author is working on extending Theorem \ref{thm:multiplicity} to lens spaces. As many multiplicity results, the proof of this theorem makes use of index theory and in particular Ekeland--Hofer theory. The proof of Theorem \ref{thm:my_result} relies on a generalisation of \emph{Rabinowitz--Floer homology}. Rabinowitz--Floer homology is the Morse--Bott homology in the sense of Floer associated with the Rabinowitz action functional introduced by Kai Cieliebak and Urs Frauenfelder in 2009. See the excellent survey article \cite{albersfrauenfelder:rfh:2012} for a brief introduction to Rabinowitz--Floer homology and \cite{schlenk:floer:2019} for an overview of common Floer theories. One important feature of this homology in our work is that it provides an affirmative answer to the \emph{Weinstein conjecture} in some instances. Specifically, we introduce an analogue of the twisted Floer homology \cite{uljarevic:liouville:2017} in the Rabinowitz--Floer setting. Following  \cite{cieliebakfrauenfelder:rfh:2009} and \cite{albersfrauenfelder:rfh:2010}, we construct a Morse--Bott homology for a suitable twisted version of the standard Rabinowitz action functional, that is, the Lagrange multiplier functional of the symplectic area functional.

\begin{theorem}[{\cite[Theorem~1.1]{baehni:rfh:2021}}]
	\label{thm:twisted_rfh}
	Let $(M,\lambda)$ be the completion of a Liouville domain $(W,\lambda)$ and let $\varphi \in \Diff(W)$ be of finite order and such that $\varphi^* \lambda - \lambda = df_\varphi$ for some smooth compactly supported function $f_\varphi \in C^\infty_c(\Int W)$.

	\begin{enumerate}[label=\textup{(\alph*)}]
		\item The semi-infinite dimensional Morse--Bott homology $\RFH^\varphi(\partial W,M)$ in the sense of Floer of the twisted Rabinowitz action functional exists and is well-defined. Moreover, twisted Rabinowitz--Floer homology is invariant under twisted homotopies of Liouville domains.
		\item If $\partial W$ is simply connected and does not admit any nonconstant twisted Reeb orbit, then $\RFH^\varphi_*(\partial W,M) \cong \operatorname{H}_*(\Fix(\varphi\vert_{\partial W});\mathbb{Z}_2)$.
		\item If $\partial W$ is displaceable by a compactly supported Hamiltonian symplectomorphism in $(M,\lambda)$, then $\RFH^\varphi(\partial W,M) \cong 0$.
	\end{enumerate}	
\end{theorem}

Twisted Rabinowitz--Floer homology does indeed generalise standard Rabinowitz--Floer homology as
\begin{equation*}
	\RFH^{\id_W}(\partial W,M) \cong \RFH(\partial W,M).
\end{equation*}
The proof of Theorem \ref{thm:my_result} is straightforward, once we have computed the $\mathbb{Z}_m$-equivariant twisted Rabinowitz--Floer homology of the spheres $\mathbb{S}^{2n - 1} \subseteq \mathbb{C}^n$. Indeed, by invariance we may assume that $\Sigma = \mathbb{S}^{2n - 1}$, as $\Sigma$ is star-shaped. Then we use the following elementary topological fact (see Lemma \ref{lem:noncontractible} below). Let $\Sigma$ be a simply connected topological manifold and let $\varphi \colon \Sigma \to \Sigma$ be a homeomorphism of finite order $m$ that is not equal to the identity. If the induced discrete action
\begin{equation*}
	\mathbb{Z}_m \times \Sigma \to \Sigma, \qquad [k] \cdot x := \varphi^k(x)
\end{equation*}
\noindent is free, then $\pi \colon \Sigma \to \Sigma/\mathbb{Z}_m$ is a normal covering map \cite[Theorem~12.26]{lee:tm:2011}. For a point $x \in \Sigma$ define the \bld{based twisted loop space of $\Sigma$ and $\varphi$} by
\begin{equation*}
	\mathscr{L}_\varphi(\Sigma,x) := \cbr[0]{\gamma \in C(I,\Sigma) : \gamma(0) = x \text{ and } \gamma(1) = \varphi(x)},
\end{equation*}
\noindent where $I := \intcc[0]{0,1}$. Then we have the following result. See Figure \ref{fig:twisted_lift}.

\begin{lemma}
	If $\gamma \in \mathscr{L}_\varphi(\Sigma,x)$, then $\pi \circ \gamma \in \mathscr{L}(\Sigma/\mathbb{Z}_m,\pi(x))$ is not contractible. Conversely, if $\gamma \in \mathscr{L}(\Sigma/\mathbb{Z}_m,\pi(x))$ is not contractible, then there exists $1 \leq k < m$ such that $\widetilde{\gamma}_x \in \mathscr{L}_{\varphi^k}(\Sigma,x)$ for the unique lift $\widetilde{\gamma}_x$ of $\gamma$ with $\widetilde{\gamma}_x(0) = x$.
	\label{lem:noncontractible}
\end{lemma}

For a more detailed study of twisted loop spaces of universal covering manifolds as well as a proof of Lemma \ref{lem:noncontractible} see Appendix \ref{sec:twisted_loops_on_universal_covering_manifolds}.

Another interesting application of Theorem \ref{thm:twisted_rfh} is the following \emph{forcing result}. Suppose that $\partial W$ is Hamiltonianly displaceable in the completion $(M,\lambda)$ and simply connected. If $\Fix(\varphi\vert_{\partial W}) \neq \emptyset$, then $\partial W$ does admit a twisted periodic Reeb orbit. Indeed, 
if there does not exist any twisted periodic Reeb orbit on $\partial W$, we compute
\begin{equation*}
	\RFH^\varphi(\partial W,M) \cong \operatorname{H}(\Fix(\varphi\vert_{\partial W});\mathbb{Z}_2) = \bigoplus_{k \geq 0} \operatorname{H}_k(\Fix(\varphi\vert_{\partial W});\mathbb{Z}_2) \neq 0
\end{equation*}
\noindent by part (b) of Theorem \ref{thm:twisted_rfh}, contradicting part (c) of Theorem \ref{thm:twisted_rfh}.

\begin{figure}[h!tb]
	\centering
	\includegraphics[width=.65\textwidth]{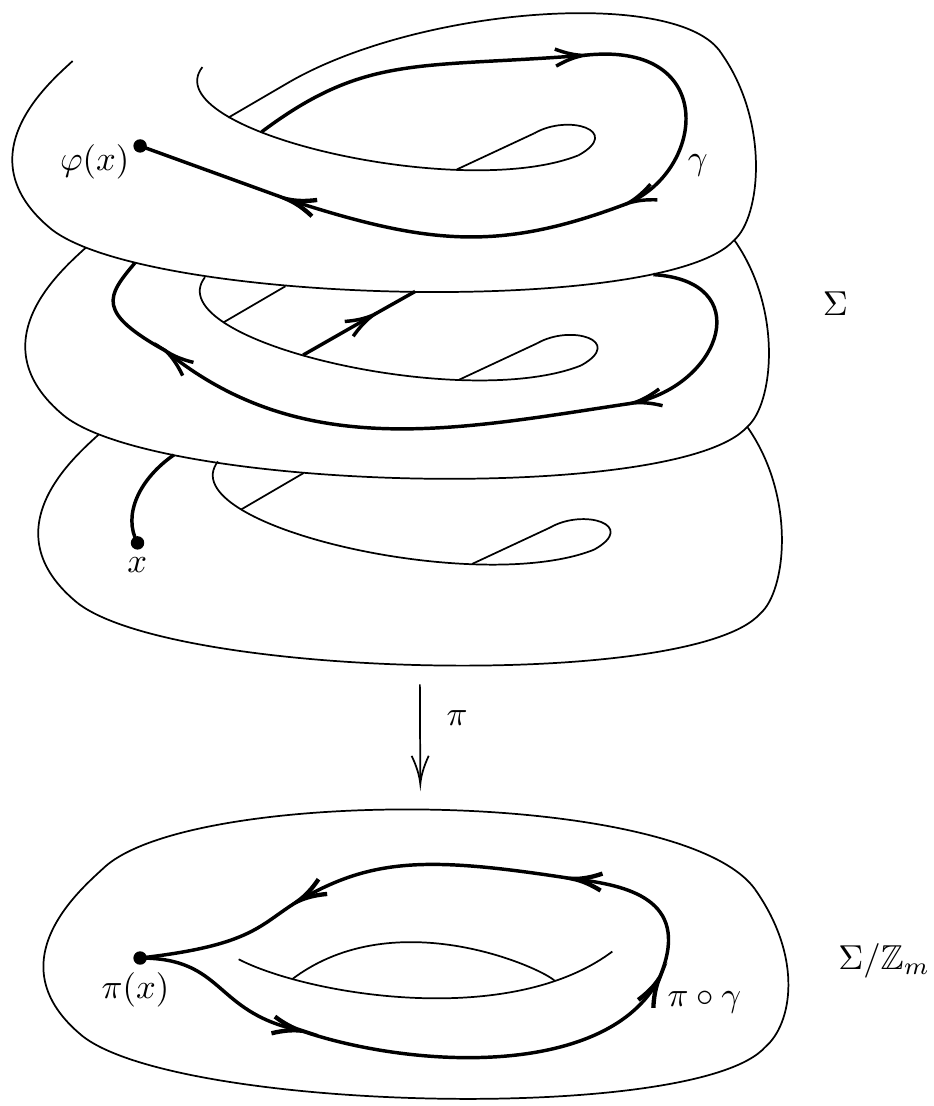}
	\caption{The projection $\pi \circ \gamma \in \mathscr{L}(\Sigma/\mathbb{Z}_m,\pi(x))$ of $\gamma \in \mathscr{L}_\varphi(\Sigma,x)$ is not contractible for the deck transformation $\varphi \neq \id_\Sigma$.}
	\label{fig:twisted_lift}
\end{figure}

\begin{theorem}[Forcing]
	\label{thm:forcing}
	Let $\Sigma$ be a twisted stable displaceable hypersurface in a symplectically aspherical, geometrically bounded, symplectic manifold $(M,\omega)$ for some $\varphi \in \Symp(M,\omega)$ and suppose that $v_0$ is a contractible twisted closed characteristic on $\Sigma$ belonging to a Morse--Bott component $C$. Then there exists a contractible twisted closed characteristic $v \notin C$ such that
	\begin{equation*}
		\int_{\mathbb{D}}\overline{v}^*\omega - \int_{\mathbb{D}} \overline{v}_0^*\omega \leq \ord(\varphi)e(\Sigma),
	\end{equation*}	
	\noindent where $e(\Sigma)$ denotes the displacement energy of $\Sigma$.
\end{theorem}

The proof of Theorem \ref{thm:forcing} is an adaptation of \cite[Theorem~4.9]{cieliebakfrauenfelderpaternain:mane:2010}. This result was initially shown by Felix Schlenk using quite different methods.

Finally, we put Theorem \ref{thm:my_result} into context. If $\Sigma^{2n - 1}/\mathbb{Z}_m$ satisfies the index condition
\begin{equation}
	\label{eq:index_condition}
	\mu_{\CZ}(\gamma) > 4 - n
\end{equation}
\noindent for all contractible Reeb orbits $\gamma$, the $\bigvee$-shaped symplectic homology $\check{\SH}(\Sigma)$ can be defined in the positive cylindrical end $\intco{0,+\infty} \times \Sigma$ by \cite[Corollary~3.7]{uebele:reeb:2019}. If $\Sigma/\mathbb{Z}_m$ admits a Liouville filling $W$, then we have 
\begin{equation*}
	\check{\SH}_*(\Sigma/\mathbb{Z}_m,M) \cong \RFH_*(\Sigma/\mathbb{Z}_m,M),
\end{equation*}
\noindent where $M$ denotes the completion of $W$. Note that even in the case of lens spaces this need not be the case, as for example $\mathbb{RP}^{2n - 1}$ is not Liouville fillable for any odd $n \geq 2$ by \cite[Theorem~1.1]{ghiggininiederkrueger:fillings:2020}. As the index condition \eqref{eq:index_condition} is only required for contractible Reeb orbits and they come from the universal covering manifold $\Sigma$, we can say something in the case where $\Sigma$ is strictly convex. Indeed, the Hofer--Wysocki--Zehnder Theorem \cite[Theorem~12.2.1]{frauenfelderkoert:3bp:2018} then implies that $\Sigma$ is dynamically convex, that is,
\begin{equation*}
	\mu_{\CZ}(\gamma) \geq n + 1
\end{equation*}
\noindent holds for all periodic Reeb orbits $\gamma$. Thus for $n \geq 2$, the index condition is satisfied and we can compute $\check{\SH}_*(\mathbb{S}^{2n - 1}/\mathbb{Z}_m)$ via the $\mathbb{Z}_m$-equivariant version of the symplectic homology $\check{\SH}_*(\mathbb{S}^{2n - 1})$.

In the case of hypertight contact manifolds, there is a similar construction without the index condition \eqref{eq:index_condition}. See for example \cite[Theorem~1.1]{meiwesnaef:hypertight:2018}. By \cite[Theorem~1.7]{meiwesnaef:hypertight:2018}, there do exist noncontractible periodic Reeb orbits on hypertight contact manifolds under suitable technical conditions. Moreover, one can show the existence of invariant Reeb orbits in this setting. See \cite[Corollary~1.6]{meiwesnaef:hypertight:2018} as well as \cite[Theorem~1.6]{merrynaef:invariant:2016} in the Liouville-fillable case.

The thesis is organised as follows. In Section \ref{sec:Floer_homology}, we review the basics of Hamiltonian Floer homology defined as the Morse--Bott homology associated with the symplectic action functional. We follow the cascade approach introduced by Urs Frauenfelder and define Hamiltonian Floer homology in the simplest case, that is, in the symplectically aspherical case. This is sufficient for our purposes. A detailed proof of the compactness of the relevant moduli spaces is given in Appendix \ref{ch:bubbling_analysis} and in order to deal with transversality, we use the polyfold approach which is sketched in Appendix \ref{ch:polyfolds}. We also review stable Hamiltonian manifolds, a generalisation of contact manifolds. Appendix \ref{ch:bubbling_analysis} is based on lecture notes written for a course on Hamiltonian Floer homology in the winter semester 2021/2022 taught by Urs Frauenfelder at the university of Augsburg.

In Section \ref{sec:twisted_rfh}, we introduce the main machinery for defining our new homology theory and prove Theorem \ref{thm:twisted_rfh}. This material is an extended version of \cite{baehni:rfh:2021}.

In Section \ref{sec:applications}, we prove Theorem \ref{thm:my_result} and the Forcing theorem \ref{thm:forcing}. Theorem \ref{thm:my_result} and its proof is also taken from \cite{baehni:rfh:2021}.

In the final Chapter \ref{sec:further_steps}, we indicate two further results that may be obtained using the theory developed in this thesis. 

\newpage
\section{Hamiltonian Floer Homology}
\label{sec:Floer_homology}

Floer homology was introduced by Andreas Floer around 1988 to tackle the \emph{homological Arnold conjecture}. Roughly speaking, the conjecture says in its simplest form that the number of nondegenerate solutions of a $1$-periodic Hamiltonian equation is bounded below by the dimension of the singular homology of the symplectic manifold with coefficients in $\mathbb{Z}_2$. That is, the number of such solutions is always bounded below by a topological invariant. This resembles the famous \emph{Morse inequalities}, and thus it is not suprising that the construction of Floer homology was largely influenced by Morse homology. See the excellent article \cite{schlenk:floer:2019}. However, a key technical ingredient for this semi-infinite dimensional version of Morse homology was Gromov's analysis of pseudoholomorphic curves introduced in 1985. Today there are many flavours of Floer theories, and we shall focus on \emph{Rabinowitz--Floer homology}. This homology was introduced in 2009 by Kai Cieliebak and Urs Frauenfelder. Crucial is the observation, that Floer homology can also be constructed in a more general way, namely in the \emph{Morse--Bott} case. In contrast to standard Hamiltonian Floer homology, Rabinowitz--Floer homology considers a fixed energy but arbitrary period problem. This leads to particular instances of the \emph{Weinstein conjecture} formulated in 1979, including the result of Rabinowitz in 1978. Weinstein conjectured that on every compact manifold admitting a contact form, there must exist a closed Reeb orbit. For an extensive historical treatment see \cite{mcduffsalamon:st:2017}. 

The aim of this introductory chapter is to explain the fundamental concepts required later on. In the first section we discuss the finite-dimensional version of Morse--Bott homology, a generalisation of Morse homology.

The second section discusses some basic facts coming from Hamiltonian dynamics, focusing on theory we need in subsequent chapters.

The third section introduces the archetypical version of Floer homology, \emph{Hamiltonian Floer homology}, on compact symplectic manifolds. We discuss the Morse--Bott approach, as this one will be useful in the discussion of Rabinowitz--Floer homology.

In the last section we discuss suitable structures on regular hypersurfaces in symplectic manifolds, including hypersurfaces of restricted contact type.

\subsection{Morse--Bott Homology}
\label{sec:mb}

Morse--Bott homology is a generalisation of Morse homology to functions with degenerate critical points. See \cite[Appendix~A]{cieliebakfrauenfelder:rfh:2009} for a short introduction via the cascade approach and \cite[Appendix~A]{frauenfelder:arnold-givental:2004} for a more extensive treatment. Morse--Bott functions often occur in the presence of symmetries. Indeed, let $G$ be a Lie group acting on a manifold $M$. If $f \in C^\infty(M)$ is $G$-invariant, that is, $f(gx) = f(x)$ holds for all $g \in G$ and $x \in M$, then $\Crit f$ is also $G$-invariant. In particular, $f$ is not a Morse function in general. However, the presence of symmetry usually simplifies the explicit computation of the Morse homology and symmetries occur throughout Physics.

\begin{definition}[{Morse--Bott Function, \cite[p.~232]{mcduffsalamon:st:2017}}]
	Let $M$ be a smooth manifold. A \bld{Morse--Bott function on $M$}\index{Morse--Bott!function} is defined to be a function $f \in C^\infty(M)$ such that
	\begin{enumerate}[label=\textup{(\roman*)}]
		\item $\Crit f \subseteq M$ is an embedded submanifold.
		\item $T_x \Crit f = \ker \Hess f\vert_x$ for all $x \in \Crit f$.
	\end{enumerate}
\end{definition}

\begin{remark}
	\label{rem:MB-lemma}
	Assumption (ii) is crucial for proving the Morse--Bott Lemma \cite[Lemma~3.51]{banyagahurtubise:morse:2004}, an analogue of the Morse Lemma. The Morse--Bott Lemma is a key technical ingredient for proving exponential decay of gradient flow lines.
\end{remark}

\begin{example}[$\mathbb{S}^{2n - 1}$]
	\label{ex:sphere}
	Let $f \colon \mathbb{S}^{2n - 1} \to \mathbb{R}$ on the odd-dimensional sphere
	\begin{equation*}
		\mathbb{S}^{2n - 1} := \cbr[4]{(z_1,\dots,z_n) \in \mathbb{C}^n : \sum_{j = 1}^n \abs[0]{z_j}^2 = 1}
	\end{equation*}
	\noindent be defined by
	\begin{equation*}
		f(z_1,\dots,z_n) := \sum_{j = 1}^n j\abs[0]{z_j}^2.
	\end{equation*}
	Then $f$ is a Morse--Bott function with critical manifold being a disjoint union of $n$ copies of $\mathbb{S}^1$.
\end{example}

Let $(M,g)$ be a compact Riemannian manifold and $f \in C^\infty(M)$ a Morse--Bott function. Choose an additional Morse function $h \in C^\infty(\Crit f)$ and a Riemannian metric $g_0$ on $\Crit f$ such that $(h,g_0)$ is a Morse--Smale pair, that is, the stable and unstable manifolds intersect transversally. Using Theorem \ref{thm:MB-polyfold}, one can define a nonnegative $\mathbb{Z}$-graded chain complex $(\CM_\ast(f),\partial_\ast)$ of $\mathbb{Z}_2$-vector spaces by
\begin{equation*}
	\CM_k(f) := \Crit_k h \otimes \mathbb{Z}_2, \quad \Crit_k h := \cbr[0]{x \in \Crit h : \ind_f(x) + \ind_h(x) = k},
\end{equation*}
\noindent for all $k \in \mathbb{Z}$, where $\ind$ denotes the ordinary Morse index, that is, the number of negative eigenvalues of the Hessian at that point, with boundary operator
\begin{equation*}
    \partial_k \colon \CM_k(f) \to \CM_{k - 1}(f), \qquad \partial_k x^- := \sum_{x^+ \in \Crit_{k - 1}h} n(x^-,x^+) x^+, 
\end{equation*}
\noindent where
\begin{equation*}
    n(x^-,x^+) := \#_2 \mathscr{M}(x^-,x^+) \in \mathbb{Z}_2
\end{equation*}
\noindent denotes the $\mathbb{Z}_2$-count of the abstractly perturbed unparametrised negative gradient flow lines with cascades from $x^-$ to $x^+$. Then $\partial_\ast$ is indeed a boundary operator by 
\begin{align*}
    (\partial_k \circ \partial_{k + 1})x^- &= \sum_{x^+ \in \Crit_{k - 1} h} \sum_{x \in \Crit_k h} \#_2\mathscr{M}(x^-,x) \#_2\mathscr{M}(x,x^+)x^+\\
    &= \sum_{x^+ \in \Crit_{k - 1} h} \#_2 \partial \mathscr{M}(x^-,x^+)x^+\\
    &= 0
\end{align*}
\noindent for all $x^- \in \Crit_{k + 1}h$ as 
\begin{equation*}
	\partial \mathscr{M}(x^-,x^+) \cong \coprod_{x \in \Crit h} \mathscr{M}(x^-,x) \times \mathscr{M}(x,x^+).
\end{equation*}
Thus we can define the \bld{Morse--Bott homology of $f$} by
\begin{equation*}
    \HM_k(f) := \frac{\ker \partial_k}{\im \partial_{k + 1}}, \qquad \forall k \in \mathbb{Z}.
\end{equation*}
As our notation suggests, $\HM(f)$ is independent of any auxiliary data up to natural isomorphisms. In particular, as every Morse function is a Morse--Bott function, Morse--Bott homology is canonically isomorphic to the ordinary Morse homology of $M$, and thus to the singular homology of $M$ with coefficients in $\mathbb{Z}_2$, that is, 
\begin{equation*}
	\HM_*(f) \cong \operatorname{H}_*(M;\mathbb{Z}_2).
\end{equation*}

\begin{example}[$\mathbb{S}^{2n - 1}$]
	Consider the odd-dimensional sphere $\mathbb{S}^{2n - 1}$ with Morse--Bott function $f \in C^\infty(\mathbb{S}^{2n - 1})$ defined in Example \ref{ex:sphere}. Choose the standard height function $h \in C^\infty(\Crit f)$ on each critical component $\mathbb{S}^1$. Then the associated chain complex is given by
	\begin{equation*}
		\CM_k(f) \cong \begin{cases}
			\mathbb{Z}_2 & 0 \leq k \leq 2n - 1,\\
			0 & \text{else}.
		\end{cases}
	\end{equation*}
	\noindent with boundary operator
	\begin{equation*}
		\partial_k = \begin{cases}
			1 & k \text{ even and } 1 \leq k \leq 2n - 1,\\
			0 & \text{else}.
		\end{cases}
	\end{equation*}
	See Figure \ref{fig:mb}. Thus the resulting homology is
	\begin{equation*}
		\HM_k(f) \cong \begin{cases}
			\mathbb{Z}_2 & k = 0,2n - 1,\\
			0 & \text{else}.
		\end{cases}
	\end{equation*}
\end{example}

\begin{figure}[h!tb]
	\centering
	\includegraphics[width=.75\textwidth]{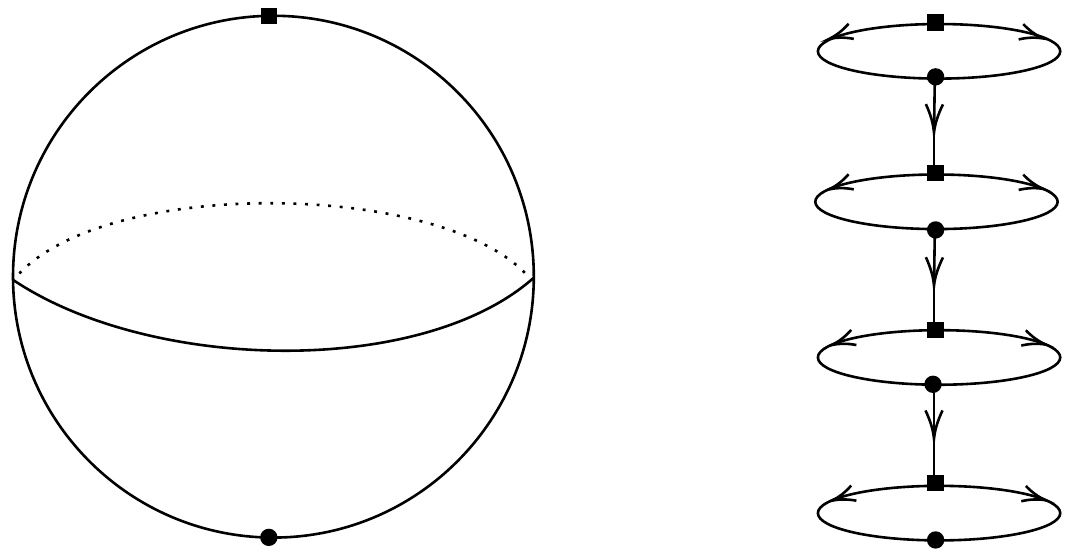}
	\caption{The sphere $\mathbb{S}^{2n - 1}$ with standard height function is depicted on the left and on the right we see the critical submanifold $\Crit f$ with standard height function on each critical component.}
	\label{fig:mb}
\end{figure}

\begin{example}[The Teapot]
    Consider the deformed sphere $\mathbb{S}^2$ as in Figure \ref{fig:teapot}. Then the standard height function is a Morse--Bott function on the teapot with critical manifold being the disjoint union of a circle $\mathbb{S}^1$ and four nondegenerate critical points. Choose also the standard height function on $\mathbb{S}^1$. Then the resulting chain complex is given by
    \begin{equation*}
		\begin{tikzcd}
			0 \arrow[r] & \mathbb{Z}_2 \oplus \mathbb{Z}_2 \arrow[r,"\begin{pmatrix}1 \>\> 1\\0 \>\> 0 \end{pmatrix}"] & \mathbb{Z}_2 \oplus \mathbb{Z}_2 \arrow[r,"\begin{pmatrix}0 \>\> 1\\0 \>\> 1 \end{pmatrix}"] & \mathbb{Z}_2 \oplus \mathbb{Z}_2 \arrow[r] & 0
		\end{tikzcd}
	\end{equation*}
	Thus the homology coincides again with $\operatorname{H}_*(\mathbb{S}^2;\mathbb{Z}_2)$.
\end{example}

\begin{figure}[h!tb]
	\centering
	\includegraphics[width=.75\textwidth]{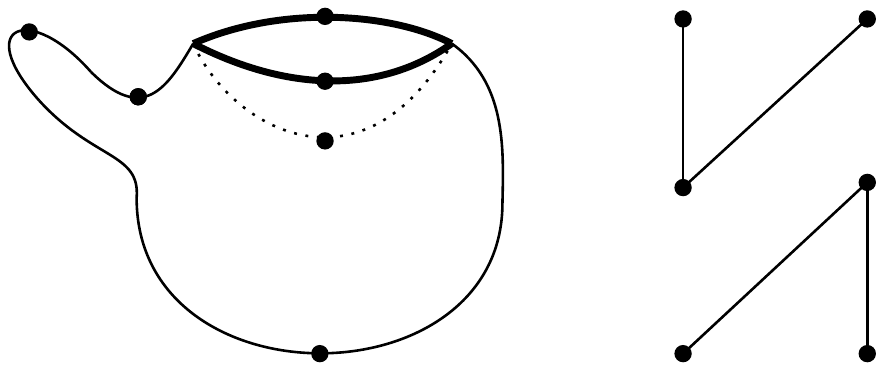}
	\caption{The deformed sphere $\mathbb{S}^2$ resembling a teapot with the standard height function.}
	\label{fig:teapot}
\end{figure}

\subsection{Hamiltonian Dynamics}
\label{sec:Hamiltonian_dynamics}

The modern language of classical mechanics is provided by symplectic geometry. For an introduction to symplectic geometry see \cite{silva:sg:2008} and for a more sophisticated treatment \cite{mcduffsalamon:st:2017}. For an introduction to Hamiltonian dynamics see \cite{abrahammarsden:cm:1978} as well as \cite{hoferzehnder:hd:1994} for a view towards symplectic invariants. Here we just briefly review the basics needed later on and to fix our conventions.

\begin{definition}[Hamiltonian System]
	A \bld{Hamiltonian system}\index{Hamiltonian!system} is a symplectic manifold $(M,\omega)$, called the \bld{phase space} together with a smooth function $H \in C^\infty(M)$, called a \bld{Hamiltonian function}. We write $(M,\omega,H)$ for a Hamiltonian system.
\end{definition}

\begin{definition}[Hamiltonian Vector Field]
	Let $(M,\omega,H)$ be a Hamiltonian system. The \bld{Hamiltonian vector field}\index{Hamiltonian!vector field} is defined to be the vector field $X_H \in \mathfrak{X}(M)$ given implicitly by
	\begin{equation*}
		i_{X_H}\omega = -dH.
	\end{equation*}	
\end{definition}

\begin{lemma}[{Jacobi, \cite[Theorem~3.3.19]{abrahammarsden:cm:1978}}]
	\label{lem:hamiltonian_vector_field_pullback}
	Let $(M,\omega,H)$ be a Hamiltonian system and let $\varphi \in \Symp(M,\omega)$ be a symplectomorphism. Then
	\begin{equation*}
		\varphi^*X_H = X_{\varphi^*H}.
	\end{equation*}
\end{lemma}

\begin{proof}
	We compute
	\begin{align*}
		i_{X_{\varphi^*H}}\omega = -d\varphi^*H = -\varphi^*dH = \varphi^*(i_{X_H}\omega) = i_{\varphi^*X_H}(\varphi^*\omega) = i_{\varphi^*X_H}\omega.
	\end{align*}
	Thus we conclude by the uniqueness of the Hamiltonian vector field.	
\end{proof}

\begin{lemma}
	\label{lem:conjugated_flow}
	Let $(M,\omega,H)$ be a Hamiltonian system and let $\varphi \in \Symp(M,\omega)$ be a symplectomorphism. Then 
	\begin{equation*}
		\phi^{X_{\varphi^*H}}_t = \varphi^{-1} \circ \phi_t^{X_H} \circ \varphi, 
	\end{equation*}
	\noindent whenever either side is defined, where $\phi$ denotes the smooth flow of a vector field.
\end{lemma}

\begin{proof}
	Using Lemma \ref{lem:hamiltonian_vector_field_pullback} we compute
	\begin{align*}
		\frac{d}{dt}\varphi^{-1} \circ \phi_t^{X_H} \circ \varphi &= D\varphi^{-1} \circ \frac{d}{dt}\phi^{X_H}_t \circ \varphi\\
		&= D\varphi^{-1} \circ X_H \circ \phi_t^{X_H} \circ \varphi\\
		&= D\varphi^{-1} \circ X_H \circ \varphi \circ \varphi^{-1} \circ \phi_t^{X_H} \circ \varphi\\
		&= \varphi^*X_H \circ \varphi^{-1} \circ \phi_t^{X_H} \circ \varphi\\
		&= X_{\varphi^*H} \circ \varphi^{-1} \circ \phi_t^{X_H} \circ \varphi,
	\end{align*}
	\noindent and the result follows by the uniqueness of integral curves.
\end{proof}

\begin{definition}[{Algebra of Classical Observables, \cite[p.~46]{takhtajan:qm:2008}}]
	Let $(M,\omega)$ be a symplectic manifold. The commutative real algebra $C^\infty(M)$ of smooth functions on $M$ is called the \bld{algebra of classical observables}.
\end{definition}

\begin{definition}[{Poisson Bracket, \cite[p.~578]{lee:dt:2012}}]
	Let $(M,\omega)$ be a symplectic manifold. Define a mapping, called the \bld{Poisson bracket on the algebra of classical observables}\index{Poisson!bracket},
	\begin{equation*}
		\cbr{\cdot,\cdot} : C^\infty(M) \times C^\infty(M) \to C^\infty(M)
	\end{equation*}
	\noindent by
	\begin{equation*}
		\cbr{f,g} := \omega(X_f,X_g).
	\end{equation*}
\end{definition}

\begin{remark}[{\cite[Corollary~22.20]{lee:dt:2012}}]
	Let $(M,\omega)$ be a symplectic manifold. Then
	\begin{equation*}
		(C^\infty(M),\{\cdot,\cdot\}) \to (\mathfrak{X}(M),[\cdot,\cdot]), \qquad f \mapsto X_f,
	\end{equation*}
	\noindent is a Lie algebra homomorphism, where $[\cdot,\cdot]$ denotes the Lie bracket given by
	\begin{equation*}
		[X,Y] = L_X Y = \frac{d}{dt}\bigg\vert_{t = 0} \del[1]{\phi^X_t}^* Y
	\end{equation*}
	\noindent for all $X,Y \in \mathfrak{X}(M)$.
\end{remark}

\begin{lemma}[{Evolution Equation, \cite[Corollary~3.3.15]{abrahammarsden:cm:1978}}]
	\label{lem:evolution_equation}
	Let $(M,\omega,H)$ be a Hamiltonian system. Then
	\begin{equation*}
		\frac{d}{dt}f \circ \phi^{X_H}_t = \{H,f\} \circ \phi^{X_H}_t \qquad \forall f \in C^\infty(M),	
	\end{equation*}
	\noindent whenever either side is defined.
\end{lemma}

\begin{proof}
	Using Fisherman's formula \cite[Proposition~22.14]{lee:dt:2012} we compute
	\begin{align*}
		\frac{d}{dt}f \circ \phi^{X_H}_t &= \frac{d}{dt} \del[1]{\phi_t^{X_H}}^*f\\
		&= \del[1]{\phi_t^{X_H}}^* L_{X_H} f\\
		&= \del[1]{\phi_t^{X_H}}^*\{H,f\}\\
		&= \{H,f\} \circ \phi^{X_H}_t
	\end{align*}
	\noindent for all $f \in C^\infty(M)$.
\end{proof}

\begin{corollary}[{Preservation of Energy, \cite[Theorem~2.2.2]{frauenfelderkoert:3bp:2018}}]
	\label{cor:preservation_of_energy}
	Let $(M,\omega,H)$ be a Hamiltonian system. Then
	\begin{equation*}
		H\del[1]{\phi^{X_H}_t(x)} = H(x) \qquad \forall x \in M, 
	\end{equation*}
	\noindent whenever the left side is defined.
\end{corollary}

\begin{proof}
	Using Lemma \ref{lem:evolution_equation} we compute
	\begin{equation*}
		\frac{d}{dt}H \circ \phi^{X_H}_t = \{H,H\} \circ \phi^{X_H}_t = 0
	\end{equation*}
	\noindent by antisymmetry of the Poisson bracket.
\end{proof}

We describe a particularly interesting class of Hamiltonian systems. Let $(M^n,g)$ be a compact Riemannian manifold and denote by $\pi \colon T^*M \to M$ its cotangent bundle. For a smooth potential function $V \in C^\infty(M)$ define $H \in C^\infty(T^*M)$ by 
\begin{equation}
	\label{eq:mechanical_Hamiltonian}
	H(q,p) := \frac{1}{2}\norm[0]{p}^2_{g^*} + V(q).
\end{equation}
For $\sigma \in \Omega^2(M)$ closed, the form $\omega_\sigma := dp \wedge dq + \pi^*\sigma$ is a symplectic form on $T^*M$ where $(q,p)$ denote the standard coordinates on the cotangent bundle. The symplectic manifold $(T^*M,\omega_\sigma)$ is called a \bld{magnetic cotangent bundle} and the Hamiltonian system $(T^*M,\omega_\sigma,H)$ is called a \bld{magnetic Hamiltonian system}. If $\sigma = 0$, the system is called a \bld{mechanical Hamiltonian system}. The dynamics of a magnetic Hamiltonian system are given by the flow of the associated Hamiltonian vector field
\begin{equation}
	\label{eq:magnetic_Hamiltonian_vf}
	X_H(q,p) = \sum_{i = 1}^n\frac{\partial H}{\partial p_i}\frac{\partial}{\partial q_i} + \sum_{i = 1}^n\del[4]{\sum_{j = 1}^n\sigma_{ij}(q)\frac{\partial H}{\partial p_j} - \frac{\partial H}{\partial q_i}}\frac{\partial}{\partial p_i},
\end{equation}
\noindent where $\sigma$ is locally given by
\begin{equation*}
	\sigma = \frac{1}{2} \sum_{i,j = 1}^n\sigma_{ij}(q) dq_i \wedge dq_j, \qquad \sigma_{ji} = -\sigma_{ij}.
\end{equation*}
Assume that $\sigma$ is exact, that is, there exists $\lambda \in \Omega^1(M)$ with $\sigma = d\lambda$. We claim that
	\begin{equation*}
		\varphi_\sigma \colon (T^*M, \lambda_{T^*M} + \pi^*\lambda) \to (T^*M, \lambda_{T^*M}), \quad \varphi_\sigma(q,p) := (q, p + \lambda_q)
	\end{equation*}
	\noindent is an exact symplectomorphism, where $\lambda_{T^*M}$ denotes the canonical Liouville form on $T^*M$. For $(q,p) \in T^*M$ and $v \in TT^*_{(q,p)}M$ we compute
	\begin{align*}
		(\varphi^*_\sigma\lambda_{T^*M})_{(q,p)}(v) &= \varphi_\sigma(q,p)(D\pi \circ D\varphi_\sigma(v))\\
		&= (p + \lambda_q)(D(\pi \circ \varphi_\sigma)(v))\\
		&= p(D\pi(v)) + \lambda_q(D\pi(v))\\
		&= \lambda_{T^*M}\vert_{(q,p)}(v) + (\pi^*\lambda)_{(q,p)}(v).
	\end{align*}
	The mechanical Hamiltonian \eqref{eq:mechanical_Hamiltonian} is transformed to the \bld{magnetic Hamiltonian}
	\begin{equation*}
		H \circ \varphi^{-1}_\sigma(q,p) = \frac{1}{2}\norm[0]{p - \lambda_q}^2_{g^*} + V(q). 
	\end{equation*} 

\begin{definition}[Cotangent Lift]	
	Let $\varphi \in \Diff(M)$ be a diffeomorphism of a smooth manifold $M$. Define a map $D\varphi^\dagger \colon T^*M \to T^*M$, called the \bld{cotangent lift of the diffeomorphism $\varphi$}, by
	\begin{equation*}
		\label{eq:cotangent_lift}
		D\varphi^\dagger(q,p)(v) := p\del[1]{D\varphi^{-1}(v)}, \qquad \forall v \in T_{\varphi(q)}M.
	\end{equation*}
\end{definition}

\begin{proposition}[{Physical Transformation, \cite[p.~10]{frauenfelderkoert:3bp:2018}}]
	\label{prop:physical_transformation}
	Let $\varphi \in \Diff(M)$ be a diffeomorphism and denote by $\lambda_{T^*M}$ the Liouville form on $T^*M$. Then 
	\begin{equation*}
		(D\varphi^\dagger)^*\lambda_{T^*M} = \lambda_{T^*M}.
	\end{equation*}
\end{proposition}

\begin{proof}
	Let $(q,p) \in T^*M$. Then we compute
	\begin{align*}
		(D\varphi^\dagger)^*\lambda_{T^*M}\vert_{(q,p)}(v) &= \lambda_{T^*M}\vert_{D\varphi^\dagger(q,p)} \del[1]{D(D\varphi^\dagger)(v)}\\
		&= \lambda_{T^*M}\vert_{(\varphi(q), p \circ D\varphi^{-1})} \del[1]{D(D\varphi^\dagger)(v)}\\
		&= (p \circ D\varphi^{-1})\del[1]{D\pi_{(\varphi(q), p \circ D\varphi^{-1})}\del[1]{D(D\varphi^\dagger)(v)}}\\
		&= p \del[1]{D\varphi^{-1} \circ D(\pi \circ D\varphi^\dagger)(v)}\\
		&= p \del[1]{D\varphi^{-1} \circ D(\varphi \circ \pi)(v)}\\
		&= \lambda_{T^*M}\vert_{(q,p)}(v) 
	\end{align*}
	\noindent for all $v \in T_{(q,p)}T^*M$.	
\end{proof}

\begin{example}[Holomorphic Function]
	\label{ex:holomorphic}
	Let $U \subseteq \mathbb{C}$ be an open subset and suppose that $\varphi \in C^\infty(U,\mathbb{C})$ is holomorphic with $\varphi' \neq 0$ on $U$ for the complex derivative $\varphi'$ of $\varphi$. Then the cotangent lift $D\varphi^\dagger$ of $\varphi$ is given by
	\begin{equation*}
		D\varphi^\dagger \colon T^*U \to T^*\mathbb{C}, \qquad D\varphi^\dagger(z,w) = \del[3]{\varphi(z),\frac{w}{\overline{\varphi'(z)}}}.
	\end{equation*}
\end{example}

\subsection{Morse--Bott Homology for the Symplectic Action Functional}
In this section we briefly describe how to construct a Morse--Bott homology in a semi-infinite dimensional setting following \cite{frauenfeldernicholls:morse:2020}.

\begin{definition}[{The Symplectic Action Functional, \cite[p.~446]{mcduffsalamon:st:2017}}]
	Let $(M,\omega)$ be a connected symplectic manifold such that $[\omega]\vert_{\pi_2(M)} = 0$ and denote by $\Lambda M$ the connected component of contractible loops in $C^\infty(\mathbb{T},M)$. For $H \in C^\infty(M \times \mathbb{T})$, define the \bld{symplectic action functional}\index{Action functional!symplectic}
	\begin{equation}
		\label{eq:symplectic_action_functional}
		\mathscr{A}_H \colon \Lambda M \to \mathbb{R}, \qquad \mathscr{A}_H(\gamma) := \int_{\mathbb{D}} \overline{\gamma}^* \omega - \int_0^1 H_t(\gamma(t))dt,
		\end{equation}
	\noindent where $\overline{\gamma} \in C^\infty(\mathbb{D},M)$ is a filling of $\gamma$, that is, $\overline{\gamma}(e^{2\pi i t}) = \gamma(t)$ for all $t \in \mathbb{T}$.
\end{definition}

The gradient $\grad_J \mathscr{A}_H$ of the symplectic action functional is given by
\begin{equation*}
	\grad_J \mathscr{A}_H\vert_\gamma(t) = J(\dot{\gamma}(t) - X_{H_t}(\gamma(t))) \qquad \forall t \in \mathbb{T},
\end{equation*}
\noindent for all $\gamma \in \Lambda M$ and for some $\omega$-compatible almost complex structure $J$ on $(M,\omega)$ with respect to the $L^2$-metric
\begin{equation*}
	\langle X,Y \rangle_J := \int_0^1 \omega(JX(t),Y(t)) dt
\end{equation*}
\noindent for all $X,Y \in \Gamma(\gamma^*TM)$. Thus a negative gradient flow line of the symplectic action functional $\mathscr{A}_H$ is a map $u \in C^\infty(\mathbb{R} \times \mathbb{T},M)$ that is a solution of the \bld{Floer equation}
\begin{equation}
	\label{eq:Floer_equation}
	\partial_s u(s,t) + J(\partial_tu(s,t) - X_{H_t}(u(s,t))) = 0 \qquad \forall (s,t) \in \mathbb{R} \times \mathbb{T}.
\end{equation}
Assume that $(M,\omega)$ is compact and we are given a sequence $(u_k)$ of negative gradient flow lines of the symplectic action functional such that the derivatives $Du_k$ are uniformly bounded. Then by \cite[Theorem~4.1.1]{mcduffsalamon:J-holomorphic_curves:2012} there exists a negative gradient flow line $u$ of the symplectic action functional such that
\begin{equation*}
	u_k \xrightarrow{C^\infty_{\loc}} u, \qquad k \to \infty,
\end{equation*}
\noindent up to a subsequence. Thus if the derivatives of the sequence of negative gradient flow lines are uniformly bounded, $(u_k)$ converges to a broken negative gradient flow line in the Floer--Gromov sense. In contrast to finite-dimensional Morse--Bott homology, a new phenomenon occurs if the derivatives explode. Indeed, if the derivatives explode, by Theorem \ref{thm:bubbling} there exists a nonconstant $J$-holomorphic sphere $v \in C^\infty(\mathbb{S}^2,M)$, see Figure \ref{fig:bubbling}. This cannot happen in our setting as we assumed $[\omega]\vert_{\pi_2(M)} = 0$, but by \cite[Lemma~2.2.1]{mcduffsalamon:J-holomorphic_curves:2012}
\begin{equation*}
	\int_{\mathbb{S}^2} v^*\omega > 0.
\end{equation*}

\begin{figure}[h!tb]
    \centering
    \includegraphics[width=\textwidth]{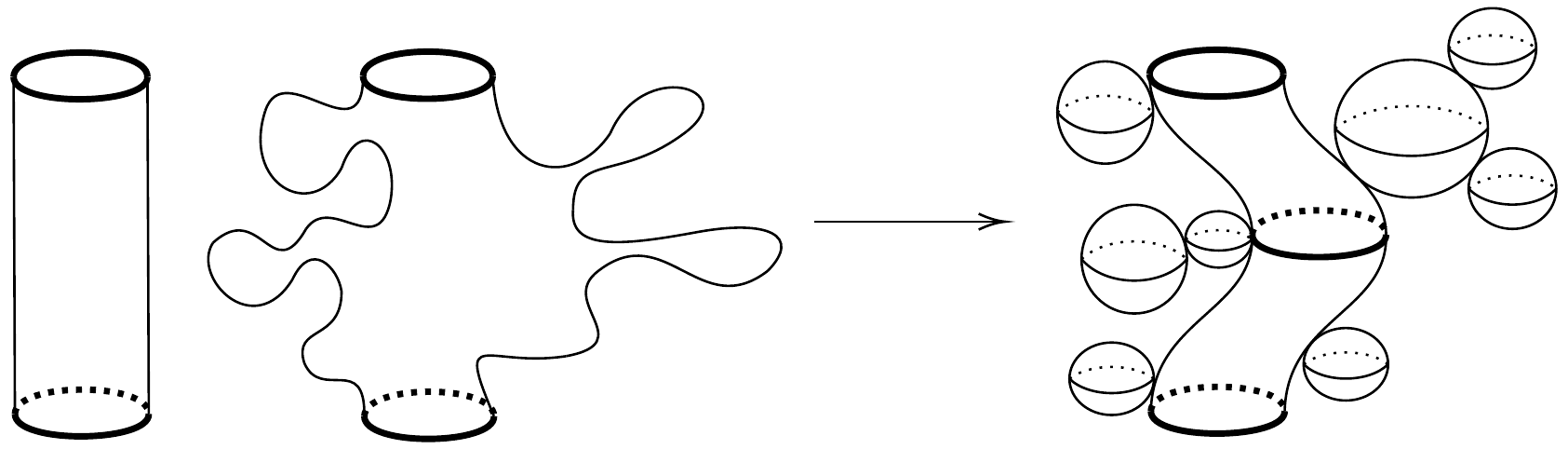}
    \caption{The bubbling phenomenon of a sequence of solutions of the Floer equation.}
    \label{fig:bubbling}
\end{figure}

By adapting Theorem \ref{thm:MB-polyfold} to the semi-infinite dimensional case as sketched in \cite[Corollary~8.13]{cieliebak:ga:2018}, we can define the \bld{Floer Homology of $H$} by
\begin{equation*}
	\HF(H) := \HM(\mathscr{A}_H)
\end{equation*}
\noindent if the symplectic action functional $\mathscr{A}_H$ is a Morse--Bott function. More precisely, choose an additonal Morse function $h \in C^\infty(\Crit \mathscr{A}_H)$ on $\Crit \mathscr{A}_H \subseteq M$ via the obvious identification $\gamma \mapsto \gamma(0)$. Define a $\mathbb{Z}_2$-vector space $\CF(H) := \Crit h \otimes \mathbb{Z}_2$ and a boundary operator 
\begin{equation*}
	\partial \colon \CF(\mathscr{A}_H) \to \CF(\mathscr{A}_H), \qquad \partial \gamma^- := \sum_{\gamma^+ \in \Crit h} n(\gamma^-,\gamma^+)\gamma^+,
\end{equation*}
\noindent where 
\begin{equation*}
	n(\gamma^-,\gamma^+) := \#_2 \mathscr{M}^0(\gamma^-,\gamma^+) \in \mathbb{Z}_2
\end{equation*}
\noindent denotes the $\mathbb{Z}_2$-count of the zero dimensional component of the moduli space of all abstractly perturbed unpartametrised negative gradient flow lines with cascades. Then the ungraded Floer homology of $H \in C^\infty(M \times \mathbb{T})$ is given by
\begin{equation*}
	\HF(H) = \frac{\ker \partial}{\im \partial}.
\end{equation*}
Again, one can show that the definition of Hamiltonian Floer homology does not depend on any auxiliary choices. In fact, Hamiltonian Floer homology is also independent of the choice of time-dependent Hamiltonian function $H$. Consequently, we have a chain of natural isomorphisms
\begin{equation}
	\label{eq:chain}
	\HF(H) = \HM(\mathscr{A}_H)\cong\HM(\mathscr{A}_0) = \bigoplus_{k\geq 0}\HM_k(h) \cong \bigoplus_{k\geq 0}\operatorname{H}_k(M;\mathbb{Z}_2).
\end{equation}
Consequently, we have that
\begin{equation*}
	\#\Crit \mathscr{A}_H = \dim_{\mathbb{Z}_2}\CF(H) \geq \dim_{\mathbb{Z}_2} \HF(H) = \sum_{k = 0}^{\dim M}\dim_{\mathbb{Z}_2}\operatorname{H}_k(M;\mathbb{Z}_2)
\end{equation*}
This resembles the famous Morse inequalities and yields a proof of a special case for the following conjecture.

\begin{conjecture}[Homological Arnold Conjecture]
	Let $(M,\omega)$ be a compact symplectic manifold and $H \in C^\infty(M \times \mathbb{T})$ such that $\mathscr{A}_H$ is a Morse function. Then the number of contractible periodic orbits $\mathscr{P}(H)$ of $H$ satisfies the inequality
	\begin{equation*}
		\# \mathscr{P}(H) \geq \sum_{k = 0}^{\dim M} \dim_{\mathbb{Z}_2} \operatorname{H}_k(M;\mathbb{Z}_2).
	\end{equation*}
\end{conjecture}
For a discussion of the general homological Arnold conjecture, see \cite[p.~29]{schlenk:floer:2019}.

We finally discuss a $\mathbb{Z}$-grading for Hamiltonian Floer homology. There is an obvious $\mathbb{Z}_2$-grading, but the $\mathbb{Z}$-grading requires an additional assumption. First, we observe that the ordinary Morse index and coindex are both infinite for the symplectic action functional. Indeed, as $\mathscr{A}_H$ is a zero-order perturbation of the symplectic area functional $\mathscr{A}_0$, it is enough to consider that case. Using a Darboux chart, it is actually enough to consider a loop $z \in C^\infty(\mathbb{T},\mathbb{C}^n)$. This loop can be represented by its Fourier series
\begin{equation*}
	z = \sum_{k = -\infty}^{+\infty}z_k e^{2\pi i kt}, \qquad z_k \in \mathbb{C}^n.
\end{equation*}
Then
\begin{equation}
	\label{eq:Fourier}
	\mathscr{A}_0(z) = -\pi \sum_{k = -\infty}^{+\infty} k \norm[0]{z_k}^2 =: \mathscr{A}(z).
\end{equation}
Indeed, we have that $\mathscr{A}_0(0) = 0 = \mathscr{A}(0)$ and as $C^\infty(\mathbb{T},\mathbb{C}^n)$ is connected, it suffices to show that the differential on both sides of \eqref{eq:Fourier} coincide. Represent the tangent vector $v \in T_z C^\infty(\mathbb{T},\mathbb{C}^n) \cong C^\infty(\mathbb{T},\mathbb{C}^n)$ by its Fourier series as
\begin{equation*}
	v = \sum_{k = -\infty}^{+\infty}v_k e^{2\pi i kt}, \qquad v_k \in \mathbb{C}^n.
\end{equation*}
Then we compute
\allowdisplaybreaks
\begin{align*}
	d\mathscr{A}_0\vert_z(v) &= \int_0^1 \omega(v(t),\dot{z}(t))dt\\
	&= \int_0^1 \operatorname{Re}(i \overline{v}(t)\dot{z}(t))dt\\
	&= -2\pi \operatorname{Re}\del[4]{\sum_{k = -\infty}^{+\infty}k\overline{v}_kz_k}\\
	&= -\pi \sum_{k = -\infty}^{+\infty} k(\overline{v}_k z_k + v_k \overline{z}_k)\\
	&= d\mathscr{A}\vert_z(v). 
\end{align*}
Thus the gradient $\grad \mathscr{A}_0$ with respect to the standard real inner product on $\mathbb{C}^n$ is given by 
\begin{equation*}
	\grad \mathscr{A}_0\vert_z^k = -2\pi k z_k \qquad \forall k \in \mathbb{Z}.
\end{equation*}
Consequently, the spectrum of the Hessian of $\mathscr{A}_0$ is $2\pi\mathbb{Z}$ and we see that $\mathscr{A}_0$ has infinite Morse index and coindex. Let $\gamma \in \Crit \mathscr{A}_H$ be a critical point of the symplectic action functional and choose a filling disk $\overline{\gamma} \in C^\infty(\mathbb{D},M^{2n})$ of $\gamma$. Fix a symplectic trivialisation of the pullback tangent bundle $\Phi \colon \mathbb{D} \times \mathbb{R}^{2n} \to \overline{\gamma}^*TM$. This is possible by \cite[Proposition~2.6.7]{mcduffsalamon:st:2017}. Then we can associate to the pair $(\gamma,\overline{\gamma})$ an integer by
\begin{equation*}
	\mu(\gamma,\overline{\gamma}) := \mu_{\CZ}(\Psi) \in \mathbb{Z},
\end{equation*}
\noindent where $\mu_{\CZ}$ denotes the Conley--Zehnder index \cite[Definition~10.4.1]{frauenfelderkoert:3bp:2018} of the path of symplectic matrices $\Psi$ defined by
\begin{equation*}
	\Psi \colon \intcc[0]{0,1} \to \Sp(n), \qquad \Psi(t) := \Phi^{-1}_{e^{2\pi i t}} \circ D\phi^{X_{H_t}}_t \circ \Phi_1.
\end{equation*}
If we choose another filling disk $\overline{\gamma}' \in C^\infty(\mathbb{T},M)$ of $\gamma$, then we have the formula
\begin{equation*}
	\mu(\gamma,\overline{\gamma}) - \mu(\gamma,\overline{\gamma}') = 2c_1([\overline{\gamma}'\# \overline{\gamma}^-]),
\end{equation*}
\noindent where $[\overline{\gamma}'\# \overline{\gamma}^-] \in \pi_2(M)$ and $c_1$ denotes the first Chern class \cite[p.~85]{mcduffsalamon:st:2017}. So there is always a well-defined $\mathbb{Z}_2$-grading of $\CF(H)$. If $c_1(TM)\vert_{\pi_2(M)} = 0$, there is also a well-defined $\mathbb{Z}$-grading. More precisely, if $c_1(TM)\vert_{\pi_2(M)} = 0$, then the Floer chain group $\CF(H)$ is graded by the signature index
\begin{equation*}
	\mu_{\CZ}(\gamma) - \frac{1}{2}\sgn \Hess h (\gamma) \qquad \forall \gamma \in \Crit h.
\end{equation*}
See \cite[p.~297]{cieliebakfrauenfelder:rfh:2009}. Note that the chain of canonical isomorphisms \eqref{eq:chain} gives rise to a natural isomorphism $\HF_*(M) \cong \operatorname{H}_{\ast + n}(M;\mathbb{Z}_2)$.

\subsection{Regular Energy Surfaces}
\label{sec:regular_energy_surfaces}
In contrast to Floer homology, in Rabinowitz--Floer homology, we study an arbitrary period but fixed energy problem. Thus we need to consider hypersurfaces in Hamiltonian systems. 

\begin{definition}[Regular Energy Surface]
	Let $(M,\omega,H)$ be a Hamiltonian system. The level set $\Sigma := H^{-1}(0)$ is a \bld{regular energy surface}\index{Surface!regular energy}, if we have that $\Crit H \cap \Sigma = \emptyset$.
\end{definition}

\begin{definition}[{Hamiltonian Manifold, \cite[Definition~2.4.1]{frauenfelderkoert:3bp:2018}}]
	A \bld{Hamiltonian manifold}\index{Manifold!Hamiltonian} is defined to be a pair $(\Sigma,\omega)$, where $\Sigma$ is an odd-dimensional smooth manifold and $\omega \in \Omega^2(\Sigma)$ is closed such that $\ker \omega$ is a line distribution. The foliation inducing the line distribution $\ker \omega$ is called the \bld{characteristic foliation}.
\end{definition}

\begin{example}[Regular Energy Surface]
	Let $\Sigma$ be a regular energy surface in a Hamiltonian system $(M,\omega,H)$. Then $(\Sigma,\omega\vert_\Sigma)$ is a Hamiltonian manifold. Moreover, the line distribution $\ker \omega \vert_\Sigma$ is spanned by the Hamiltonian vector field $X_H\vert_\Sigma$.
\end{example}

\begin{definition}[{Stable Hamiltonian Manifold, \cite[p.~1773]{cieliebakfrauenfelderpaternain:mane:2010}}]
	A Hamiltonian manifold $(\Sigma,\omega)$ is called \bld{stable}\index{Manifold!Hamiltonian!stable}, if there exists $\lambda \in \Omega^1(\Sigma)$ which is nowhere-vanishing on $\ker \omega$ and such that $\ker \omega \subseteq \ker d\lambda$. We write $(\Sigma,\omega,\lambda)$ for a stable Hamiltonian manifold.
\end{definition}

\begin{remark}
	Equivalently, a Hamiltonian manifold $(\Sigma^{2n - 1},\omega)$ is stable, if and only if there exists $\lambda \in \Omega^1(\Sigma)$ with $\ker \omega \subseteq \ker d\lambda$ and  such that $\lambda \wedge \omega^{n - 1}$ is a volume form on $\Sigma$. 
\end{remark}

\begin{example}[Regular Energy Surface]
	Let $\Sigma$ be a regular energy surface in a Hamiltonian system $(M,\omega,H)$. Suppose that there exists a vector field $X$ in a neighbourhood of $\Sigma$ with $X$ nowhere-tangent to $\Sigma$ and $\ker \omega\vert_\Sigma \subseteq \ker L_X\omega\vert_\Sigma$. Then $(\Sigma,\omega\vert_\Sigma,i_X\omega\vert_\Sigma)$ is a stable Hamiltonian manifold. 
\end{example}

\begin{example}[{\cite[Section~6.1]{cieliebakfrauenfelderpaternain:mane:2010}}]
	\label{ex:twisted_torus}
	Let $\mathbb{T}^n$ be the standard flat torus for $n \geq 2$ and let $J \colon \mathbb{R}^n \to \mathbb{R}^n$ be  an antisymmetric nonzero linear map. Define $\sigma \in \Omega^2(\mathbb{T}^n)$ by setting $\sigma(\cdot,\cdot) := \langle \cdot, J \cdot \rangle$ and denote by $\omega_\sigma = dp \wedge dq + \pi^*\sigma$ the magnetic symplectic form on $T^*\mathbb{T}^n \cong \mathbb{T}^n \times \mathbb{R}^n$. For an energy value $c \in \mathbb{R}$ set $\Sigma_c := H^{-1}(c)$ for the mechanical Hamiltonian function 
	\begin{equation*}
		H(q,p) := \frac{1}{2} \norm[0]{p}^2 \qquad \forall (q,p) \in \mathbb{T}^n \times \mathbb{R}^n.
	\end{equation*}
	Define $A := (J\vert_{\im J})^{-1}$ and $\alpha \in \Omega^1(\im J)$ by
	\begin{equation*}
		\alpha_x(v) := \frac{1}{2}\langle x,Av\rangle.
	\end{equation*}
	Then $\Sigma_c$ is a stable Hamiltonian manifold for every $c > 0$ by \cite[Proposition~6.3]{cieliebakfrauenfelderpaternain:mane:2010}. The stabilising form $\lambda$ on $\Sigma_c$ is given by
	\begin{equation}
		\label{eq:stabilising_form}
		\lambda := f^*(pdq) + (\pr_\parallel \circ \pr)^*\alpha,
	\end{equation}
	\noindent where 
	\begin{equation*}
		\pr_\perp \colon \mathbb{R}^n \to \ker J, \qquad \pr_\parallel \colon \mathbb{R}^n \to \im J, \qquad \pr \colon \mathbb{T}^n \times \mathbb{R}^n \to \mathbb{R}^n
	\end{equation*}
	\noindent denote the projections with respect to the orthogonal splitting
	\begin{equation*}
		\mathbb{R}^n = \ker J \oplus \im J,
	\end{equation*}
	\noindent and
	\begin{equation*}
		f \colon \mathbb{T}^n \times \mathbb{R}^n \to \mathbb{T}^n \times \mathbb{R}^n, \qquad f(q,p) := \del[1]{q,\pr_\perp(p)}.
	\end{equation*}
\end{example}

\begin{definition}[{Reeb Vector Field, \cite[p.~1773]{cieliebakfrauenfelderpaternain:mane:2010}}]
	\label{def:Reeb}
	Let $(\Sigma,\omega,\lambda)$ be a stable Hamiltonian manifold. The unique vector field $R \in \mathfrak{X}(\Sigma)$ implicitly defined by 
	\begin{equation*}
		i_R\omega = 0 \qquad \text{and} \qquad i_R \lambda = 1
	\end{equation*}
	\noindent is called the \bld{Reeb vector field}\index{Vector field!Reeb}.
\end{definition}

\begin{example}
	\label{ex:magnetic_flow}
	The flow $\phi_t$ of the magnetic Hamiltonian system in Example \ref{ex:twisted_torus} is given by 
	\begin{equation*}
		\phi_t(q,p) =\del[3]{\int_0^t e^{sJ}p ds + q, e^{tJ}p},
	\end{equation*}
	\noindent as one can explicitly compute this flow using \eqref{eq:magnetic_Hamiltonian_vf}, and $(q,p) \in \Sigma_c$ gives rise to a contractible closed orbit of period $\tau$ if and only if
	\begin{equation}
		\label{eq:magnetic_flow}
		\pr_\perp(p) = 0, \qquad e^{\tau J} \pr_\parallel(p) = \pr_\parallel(p), \qquad \text{and} \qquad \norm[1]{\pr_\parallel(p)}^2 = 2c.
	\end{equation}
	It is illustrative to consider the special case $n = 2$ and $\sigma = dq_1 \wedge dq_2$. Then
	\begin{equation*}
		\lambda = -\frac{1}{2}(p_1dp_2 - p_2dp_1),
	\end{equation*}
	\noindent and the projection of a contractible periodic orbit to $\mathbb{T}^2$ is depicted in Figure \ref{fig:torus}.
\end{example}

\begin{figure}[h!tb]
	\centering
	\includegraphics[width=.7\textwidth]{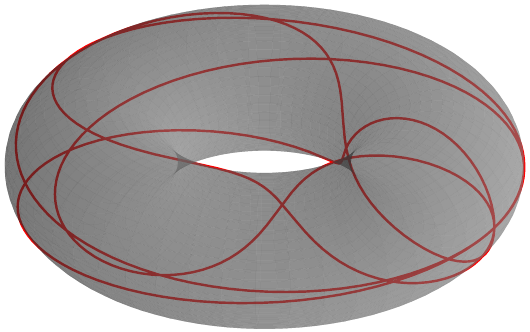}
	\caption{A contractible periodic Reeb orbit on $\mathbb{T}^2$.}
	\label{fig:torus}
\end{figure}

The following provides a large class of examples of stable Hamiltonian manifolds.

\begin{definition}[{Contact Manifold, \cite[Definition~2.5.1]{frauenfelderkoert:3bp:2018}}]
	A \bld{contact manifold}\index{Manifold!contact} is defined to be a stable Hamiltonian manifold $(\Sigma,d\lambda,\lambda)$. We simply write $(\Sigma,\lambda)$ for a contact manifold.
\end{definition}

\begin{example}[{Regular Energy Surface, \cite[Theorem~1.2.2]{abbashofer:cg:2019}}]
	Let $\Sigma$ be a compact regular energy surface in a mechanical Hamiltonian system $(T^*M,\omega_0,H)$. Then there exists $\lambda \in \Omega^1(\Sigma)$ such that $d\lambda = \omega\vert_\Sigma$ and $(\Sigma,\lambda)$ is a contact manifold.
\end{example}

\begin{definition}[{Liouville Domain, \cite[Definition~2.6.2]{frauenfelderkoert:3bp:2018}}]
	A \bld{Liouville domain}\index{Liouville!domain} is a compact connected exact symplectic manifold $(W,\lambda)$ with connected boundary such that the Liouville vector field $X \in \mathfrak{X}(W)$, implicitly defined by $i_Xd\lambda = \lambda$, is outward-pointing along the boundary $\partial W$. 
\end{definition}

\begin{remark}[{\cite[p.~24]{frauenfelderkoert:3bp:2018}}]
	For a Liouville domain $(W^{2n},\lambda)$ we have that $\partial W \neq \emptyset$. Indeed, if $\partial W = \emptyset$, then using Stokes Theorem we compute
	\begin{equation*}
		0 < \int_W d\lambda^n = \int_W d(\lambda \wedge d\lambda^{n - 1}) = \int_{\partial W} \lambda \wedge d\lambda^{n - 1} = 0.
	\end{equation*}
\end{remark}

\begin{remark}[{\cite[Lemma~2.63]{frauenfelderkoert:3bp:2018}}]
	Let $(W,\lambda)$ be a Liouville domain. Then the boundary $(\partial W,\lambda\vert_{\partial W})$ is a contact manifold.
\end{remark}

\begin{remark}
	In our definition of a Liouville domain $(W,\lambda)$ it would actually suffice to assume that the Liouville vector field $X \in \mathfrak{X}(W)$ is nowhere-tangent to the boundary $\Sigma := \partial W$. Indeed, if the Liouville vector field is inward-pointing at the boundary, we get an exact symplectic embedding
	\begin{equation*}
		\psi \colon \del[1]{\intco[0]{0,+\infty} \times \Sigma, e^r\lambda\vert_\Sigma} \hookrightarrow (W,\lambda)
	\end{equation*}
	\noindent defined by
	\begin{equation*}
		\psi(r,x) := \phi^X_t(x),
	\end{equation*}
	\noindent where $\phi^X$ denotes the smooth flow of $X$. But $\psi$ expands volume as $\psi^*_r \lambda = e^r \lambda\vert_\Sigma$.
\end{remark}

\begin{example}[{Star-Shaped Domain, \cite[Example~2.6.6]{frauenfelderkoert:3bp:2018}}]
	\label{ex:star-shaped}
	Denote by $(\mathbb{C}^n,\lambda)$ the standard exact linear symplectic manifold where
	\begin{equation}
		\label{eq:Liouville_form}
		\lambda := \frac{1}{2}\sum_{j = 1}^n \del[1]{y_j dx_j - x_j dy_j} = \frac{i}{4} \sum_{j = 1}^n \del[1]{\overline{z}_j dz_j - z_j d\overline{z}_j}, 
	\end{equation}
	\noindent with coordinates $z_j := x_j + iy_j$. Then the Liouville vector field $X \in \mathfrak{X}(\mathbb{C}^n)$ is given by 
	\begin{equation}
		\label{eq:Liouville_vector_field}
		X = \frac{1}{2}\sum_{j = 1}^n\del[3]{x_j\frac{\partial}{\partial x_j} + y_j \frac{\partial}{\partial y_j}} = \frac{1}{2}\sum_{j = 1}^n\del[3]{z_j\frac{\partial}{\partial z_j} + \overline{z}_j \frac{\partial}{\partial \overline{z}_j}}.
	\end{equation}
	If $U \subseteq \mathbb{C}^n$ is a bounded connected open subset which is star-shaped with respect to the origin and with a smooth connected boundary $\partial U$, then $(\overline{U},\lambda\vert_{\overline{U}})$ is a Liouville domain\index{Domain!star-shaped}. All star-shaped Liouvile domains are diffeomorphic to $(B^n,\lambda\vert_{B^n})$ via the radial projection, where $B^n := \{z \in \mathbb{C}^n : \norm[0]{z} \leq 1\}$ denotes the closed unit disc. However, the Reeb flow of these star-shaped contact type boundaries can be very different as in \cite[Example~5.1]{schlenk:floer:2019}. 
\end{example}

\begin{example}[{Cotangent Bundle, \cite[Example~5.2]{schlenk:floer:2019}}]
	\label{ex:cotangent_bundle}
	Let $(M,g)$ be a compact connected Riemannian manifold and consider the cotangent bundle $(T^*M,pdq)$. Then the Liouville vector field is locally given by $p \frac{\partial}{\partial p}$. Suppose $U \subseteq T^*M$ is a bounded connected open set with smooth boundary containing the zero section and such that the fibrewise intersection $U \cap T^*_qM$ is star-shaped with respect to the origin for all $q \in M$. Then $(\overline{U},pdq\vert_{\overline{U}})$ is a Liouville domain. Any such star-shaped Liouville domain in the cotangent bundle $T^*M$ is diffeomorphic to the \bld{unit cotangent bundle}
	\begin{equation*}
		D^*M := \{(q,p) \in T^*M : \norm[0]{p}_{g^*} \leq 1\},
	\end{equation*}
	\noindent with contact type boundary the spherisation $(S^*M,pdq\vert_{S^*M})$. Once more, the Reeb flow of such star-shaped hypersurfaces can be very different.
\end{example}

\begin{example}[{Magnetic Hamiltonian System, \cite[Example~5.2]{schlenk:floer:2019}}]
	Consider an exact magnetic Hamiltonian system $(T^*M,pdq + \pi^*\lambda,H)$ for $\lambda \in \Omega^1(M)$. 
	For
	\begin{equation*}
		c > \max_{q \in M}\del[3]{\frac{1}{2}\norm[0]{\lambda_q}^2_{m^*} + V(q)}
	\end{equation*}
	\noindent the level set $H^{-1}(c)$ is fibrewise star-shaped. 
\end{example}

\begin{definition}[{Liouville Automorphism, \cite[p.~237]{cieliebak:stein:2012}}]
	Let $(W, \lambda)$ be a Liouville domain with boundary $\Sigma$. A diffeomorphism $\varphi \in \Diff(W)$ is said to be a \bld{Liouville automorphism}\index{Liouville!automorphism}, if $\varphi^*\lambda - \lambda$ is exact and compactly supported in the interior $\Int W$, and $\ord \varphi < \infty$. We denote by $\Aut(W,\lambda)$ the set of all Liouville automorphisms on a given Liouville domain $(W,\lambda)$.
\end{definition}

\begin{remark}
	Let $\varphi \in \Aut(W,\lambda)$ be a Liouville automorphism. Then there exists a unique function $f_\varphi \in C^\infty_c(\Int W)$ such that $\varphi^* \lambda - \lambda = df_\varphi$.
\end{remark}

\begin{remark}
	For a Liouville domain $(W,\lambda)$, the set $\Aut(W,\lambda)$ of Liouville automorphisms is in general not a group. Indeed, for $\varphi,\psi \in \Aut(W,\lambda)$ it is not necessarily true that $\varphi \circ \psi$ is of finite order unless $\varphi$ and $\psi$ commute.
\end{remark}

\begin{remark}
	Any $\varphi \in \Aut(W,\lambda)$ induces a strict contactomorphism $\varphi\vert_{\partial W}$.
\end{remark}

\begin{example}[Rotation]
	\label{ex:rotation}
	For $m \geq 1$ consider the rotation
	\begin{equation*}
		\varphi \colon \mathbb{C}^n \to \mathbb{C}^n, \qquad \varphi(z_1,\dots,z_n) := \del[1]{e^{2\pi i k_1/m}z_1,\dots,e^{2\pi ik_n/m}z_n},
	\end{equation*}
	\noindent where $k_1,\dots,k_n \in \mathbb{Z}$ are coprime to $m$. Let $(W,\lambda)$ be a star-shaped Liouville domain in $\mathbb{C}^n$ as in Example \ref{ex:star-shaped} invariant under the rotation $\varphi$, that is, $\varphi(\partial W) = \partial W$. Then $\varphi$ is a Liouville automorphism as $\varphi^*\lambda = \lambda$ by \eqref{eq:Liouville_form} and $\ord \varphi = m$.
\end{example}

\begin{figure}[h!tb]
	\centering
	\includegraphics[width=.45\textwidth]{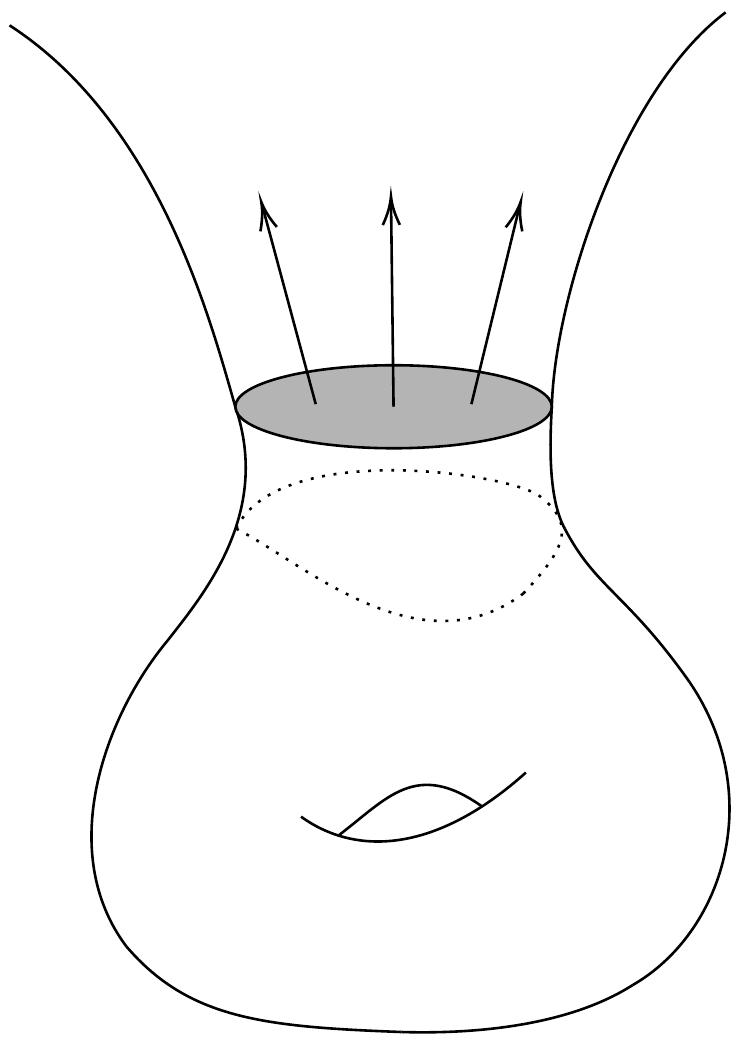}
	\caption{The completion of a Liouville domain.}
	\label{fig:completion}
\end{figure}

One can complete a Liouville domain $(W,\lambda)$ to a noncompact exact symplectic manifold without boundary by attaching the positive cylindrical end $\intco[0]{0,+\infty} \times \partial W$ to its boundary $\partial W$. See Figure \ref{fig:completion}.

\begin{definition}[{Completion of a Liouville Domain, \cite[p.~148]{mcduffsalamon:st:2017}}]
	Let $(W,\lambda)$ be a Liouville domain with boundary $\Sigma$. The \bld{completion of $(W,\lambda)$}\index{Liouville!domain!completion of} is defined to be the exact symplectic manifold $(M,\lambda)$, where 
\begin{equation*}
	M := W \cup_\Sigma \intco[0]{0,+\infty} \times \Sigma \qquad \text{and} \qquad \lambda\vert_{\intco[0]{0,+\infty} \times \Sigma} := e^r\lambda\vert_\Sigma.
	\end{equation*}
\end{definition}

\begin{example}[Star-Shaped Domain]
	Let $(W,\lambda)$ be a star-shaped Liouville domain as in Example \ref{ex:star-shaped}. Then the completion $(M,\lambda)$ of $(W,\lambda)$ is symplectomorphic to the exact linear symplectic manifold $(\mathbb{C}^n,\lambda)$ via the flow of the Liouville vector field. See Figure \ref{fig:star-shaped}. The completion of a star-shaped Liouville domain in a cotangent bundle $T^*M$ as in Example \ref{ex:cotangent_bundle} is $(T^*M,pdq)$.
\end{example}

Finally, we consider more general hypersurfaces in symplectic manifolds.

\begin{definition}[{Stable Hypersurface, \cite[p.~1774]{cieliebakfrauenfelderpaternain:mane:2010}}]
	Let $(M,\omega)$ be a connected symplectic manifold. A \bld{stable hypersurface in $(M,\omega)$} is a compact connected hypersurface $\Sigma \subseteq M$ such that the following conditions are satisfied.
	\begin{enumerate}[label=\textup{(\roman*)}]
		\item $\Sigma$ is separating, that is, $M \setminus \Sigma$ consists of two connected components $M^\pm$, where $M^-$ is bounded and $M^+$ is unbounded.
		\item There exists a vector field $X$ in a neighbourhood of $\Sigma$ such that $X$ is outward-pointing to $\Sigma \cup M^-$ and such that $\ker \omega\vert_\Sigma \subseteq \ker L_X\omega\vert_\Sigma$.
	\end{enumerate}
\end{definition}

\begin{definition}[{Displaceability, \cite[p.~411]{mcduffsalamon:st:2017}}]
	A subset $A \subseteq M$ of a symplectic manifold $(M,\omega)$ is said to be \bld{Hamiltonianly displaceable}, if there exists a compactly supported Hamiltonian symplectomorphism $\varphi_F \in \Ham_c(M,\omega)$, such that
	\begin{equation*}
		\varphi_F(A) \cap A = \emptyset.
	\end{equation*}
\end{definition}

\begin{example}[{\cite[p.~4]{frauenfelderschlenk:hd:2007}}]
	Every compact subset of $(M \times \mathbb{C}, \omega \oplus \omega_0)$ is Hamiltonianly displaceable, where $(M,\omega)$ is any symplectic manifold.
\end{example}

\begin{figure}[h!tb]
	\centering
	\includegraphics[width=.8\textwidth]{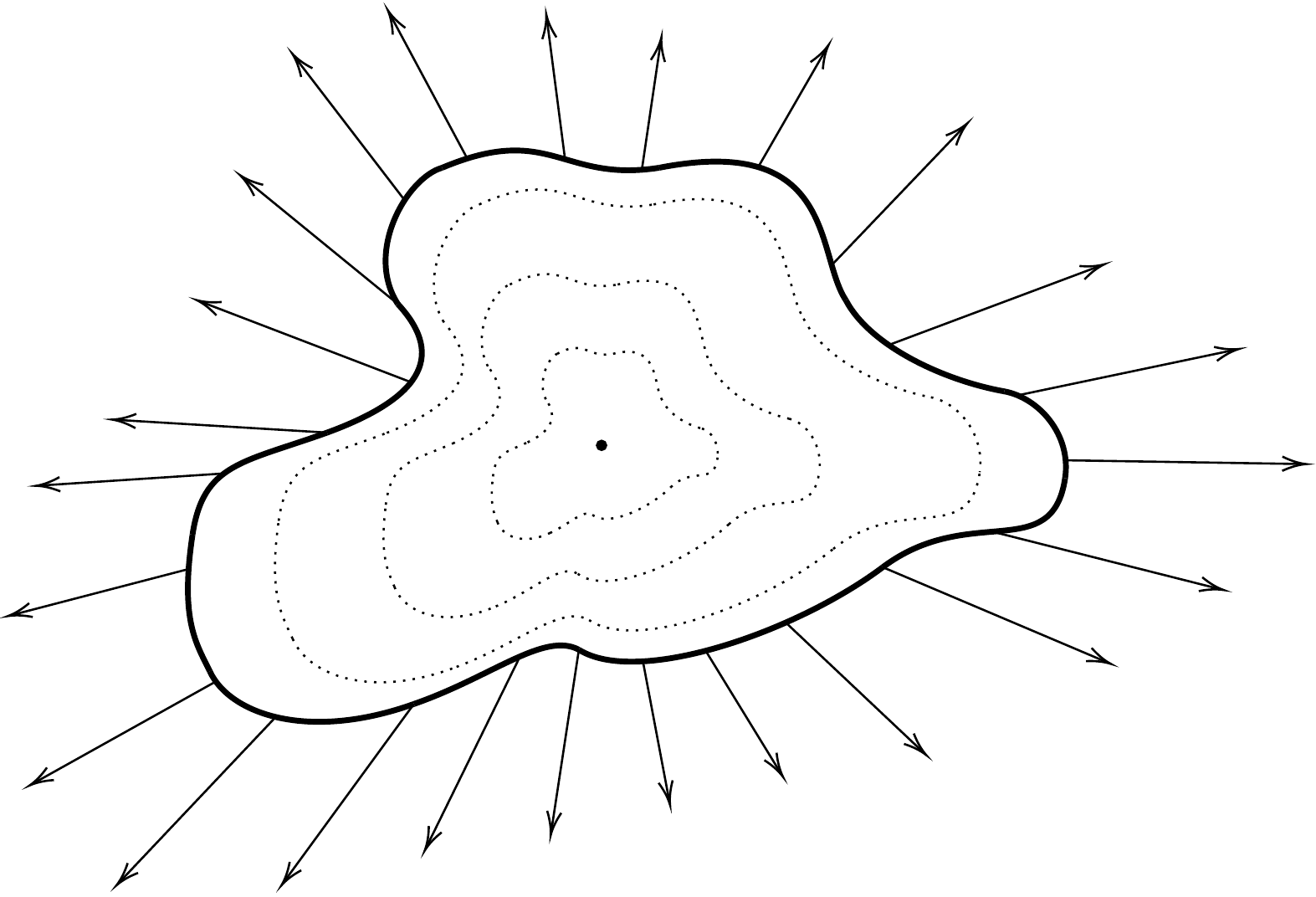}
	\caption{The completion $(\mathbb{C}^n,\lambda)$ of a star-shaped Liouville domain.}
	\label{fig:star-shaped}
\end{figure}

\begin{definition}[{Hofer Norm, \cite[p.~466]{mcduffsalamon:st:2017}}]
	Let $(M,\omega)$ be a symplectic manifold and $F \in C^\infty_c(M \times \intcc[0]{0,1})$. Define the \bld{Hofer norm of $F$} by
	\begin{equation*}
		\norm{F} := \norm{F}_+ + \norm{F}_-,
	\end{equation*}
	\noindent where
	\begin{equation*}
		\norm{F}_+ := \int_0^1 \max_{x \in M} F_t(x)dt \qquad \text{and} \qquad \norm{F}_- := -\int_0^1 \min_{x \in M} F_t(x)dt. 
	\end{equation*}
\end{definition}

\begin{definition}[{Displacement Energy, \cite[p.~469]{mcduffsalamon:st:2017}}]
	Let $(M,\omega)$ be a symplectic manifold and $A \subseteq M$ a compact subset. The \bld{displacement energy of $A$} is
	\begin{equation*}
		e(A) := \inf_{\substack{F \in C^\infty_c(M \times \intcc[0]{0,1})\\\varphi_F(A) \cap A = \emptyset}}\norm{F}.
	\end{equation*}
\end{definition}

\begin{example}
	\label{ex:displacement_energy}
	Consider the displaceable hypersurface $\Sigma_c \subseteq (T^*\mathbb{T}^n, \omega_\sigma,H)$ as in Example \ref{ex:twisted_torus}. Then by \cite[Theorem~12.3.4]{mcduffsalamon:st:2017}, the displacement energy of $\Sigma_c$ is given by
	\begin{equation*}
		e(\Sigma_c) = e\del[1]{\overline{B}^n_{\sqrt{2c}}(0)} = 2\pi c \qquad \forall c > 0,
	\end{equation*}
	\noindent where $\overline{B}^n_{\sqrt{2c}}(0) \subseteq \mathbb{R}^n$ denotes the closed ball around the origin with radius $\sqrt{2c}$.
\end{example}

\newpage
\section{Twisted Rabinowitz--Floer Homology}
\label{sec:twisted_rfh}
In this chapter we construct the generalisation of Rabinowitz--Floer homology and prove Theorem \ref{thm:twisted_rfh}.

To begin, we define the twisted Rabinowitz action functional for an exact symplectic manifold and compute its first and second variation. 

In the second section we prove a compactness result for the moduli space of twisted negative gradient flow lines in a restricted geometrical setting. This follows a standard procedure, but one has to carefully adapt the constructions and proofs to this more general case.

In the third section we define ungraded and graded twisted Rabinowitz--Floer homology and prove part (b) of Theorem \ref{thm:twisted_rfh} in Proposition \ref{prop:fixed_points}.

In the fourth section we briefly illustrate how to prove part (a) of Theorem \ref{thm:twisted_rfh} (see Theorem \ref{thm:invariance}) and define the notion of twisted homotopies of Liouville domains.

In the last section we prove an important vanishing result for twisted Rabinowitz--Floer homology, that is, part (c) of Theorem \ref{thm:twisted_rfh} (see Theorem \ref{thm:displaceable}).

\subsection{The Twisted Rabinowitz Action Functional}
\label{sec:the_twisted_Rabinowitz_action_functional}

\begin{definition}[Free Twisted Loop Space]
	Let $\varphi \in \Diff(M)$ be a diffeomorphism of a smooth manifold $M$. Define the \bld{free twisted loop space of $M$ and $\varphi$} by
	\begin{equation*}
		\mathscr{L}_\varphi M := \cbr[0]{\gamma \in C^\infty(\mathbb{R},M) : \gamma(t + 1) = \varphi(\gamma(t)) \> \forall t \in \mathbb{R}}.
	\end{equation*}
\end{definition}

Let $(M,\omega)$ be a symplectic manifold and let $\varphi \in \Symp(M,\omega)$. Given a twisted loop $\gamma \in \mathscr{L}_\varphi M$ and $\varepsilon_0 > 0$, we say that a curve
\begin{equation*}
	\intoo[0]{-\varepsilon_0,\varepsilon_0} \to \mathscr{L}_\varphi M, \qquad \varepsilon \mapsto \gamma_\varepsilon
\end{equation*}
\noindent starting at $\gamma$ is \bld{smooth}, if the induced variation
\begin{equation*}
	\mathbb{R} \times \intoo[0]{-\varepsilon_0,\varepsilon_0} \to M, \qquad (t,\varepsilon) \mapsto \gamma_\varepsilon(t)
\end{equation*}
\noindent is smooth. Since $\gamma_\varepsilon(t + 1) = \varphi(\gamma_\varepsilon(t))$ holds for all $\varepsilon \in \intoo[0]{-\varepsilon_0,\varepsilon_0}$ and $t \in \mathbb{R}$, we call such a variation a \bld{twisted variation}. Then the infinitesimal variation
\begin{equation*}
	\delta\gamma := \frac{\partial \gamma_\varepsilon}{\partial \varepsilon}\bigg\vert_{\varepsilon = 0} \in \mathfrak{X}(\gamma),
\end{equation*}
\noindent satisfies
\begin{equation*}
	\delta\gamma(t + 1) = D\varphi(\delta\gamma(t)) \qquad \forall t \in \mathbb{R}.
\end{equation*}

\begin{lemma}
	\label{lem:twisted_variation}
	Let $(M,\omega)$ be a symplectic manifold and let $\varphi \in \Symp(M,\omega)$ be of finite order. Let $\gamma \in \mathscr{L}_\varphi M$ and let $X \in \mathfrak{X}(\gamma)$ be such that
	\begin{equation*}
		X(t + 1) = D\varphi(X(t)) \qquad \forall t \in \mathbb{R}.
	\end{equation*}
	Then there exists a twisted variation of $\gamma$ such that $\delta \gamma = X$.	
\end{lemma}

\begin{proof}
	As $\varphi$ is assumed to be of finite order, there exists a $\varphi$-invariant $\omega$-compatible almost complex structure $J$ on $M$ by \cite[Lemma~5.5.6]{mcduffsalamon:st:2017}. With respect to the induced Riemannian metric
	\begin{equation*}
		m_J := \omega(J\cdot,\cdot),
	\end{equation*}
	\noindent the symplectomorphism $\varphi$ is an isometry. Define the exponential variation
	\begin{equation*}
		\mathbb{R} \times \intoo[0]{-\varepsilon_0,\varepsilon_0} \to M, \qquad \gamma_\varepsilon(t) := \exp^{\nabla_J}_{\gamma(t)}(\varepsilon X(t)),
	\end{equation*}
	\noindent for $\varepsilon_0 > 0$ sufficiently small and $\nabla_J$ denoting the Levi--Civita connection associated with $m_J$. Such an $\varepsilon_0 > 0$ exists by the naturality of geodesics \cite[Corollary~5.14]{lee:dg:2018}. Then we compute
	\begin{align*}
		\gamma_\varepsilon(t + 1) &= \exp^{\nabla_J}_{\gamma(t + 1)}(\varepsilon X(t + 1))\\
		&= \exp^{\nabla_J}_{\varphi(\gamma(t))}(D\varphi(\varepsilon X(t)))\\
		&= \varphi \del[1]{\exp^{\nabla_J}_{\gamma(t)}(\varepsilon X(t))}\\
		&= \varphi(\gamma_\varepsilon(t))
	\end{align*}
	\noindent by naturality of the exponential map \cite[Proposition~5.20]{lee:dg:2018}.
\end{proof}

\begin{remark}
	The statement of Lemma \ref{lem:twisted_variation} remains true if $\ord \varphi = \infty$.
\end{remark}

This discussion together with Lemma \ref{lem:free_twisted_loop_space} motivates the following definition of the tangent space to the free twisted loop space.

\begin{definition}[Tangent Space to the Free Twisted Loop Space]
	Let $(M,\omega)$ be a symplectic manifold and $\varphi \in \Symp(M,\omega)$. For $\gamma \in \mathscr{L}_\varphi M$ define the \bld{tangent space to the free twisted loop space at $\gamma$} by
	\begin{equation*}
		T_\gamma\mathscr{L}_\varphi M := \cbr[0]{X \in \Gamma(\gamma^*TM) : X(t + 1) = D\varphi(X(t)) \> \forall t \in \mathbb{R}}.
	\end{equation*}
\end{definition}

\begin{definition}[Twisted Hamiltonian Function]
	Let $(M,\omega)$ be a symplectic manifold and $\varphi \in \Symp(M,\omega)$. A function $H \in C^\infty(M \times \mathbb{R})$ is said to be a \bld{twisted Hamiltonian function}, if
	\begin{equation*}
		\varphi^*H_{t + 1} = H_t \qquad \forall t \in \mathbb{R}.
	\end{equation*}
	We denote the space of all twisted Hamiltonian functions by $C^\infty_\varphi(M \times \mathbb{R})$ and the subspace of all autonomous twisted Hamiltonian functions by $C^\infty_\varphi(M)$.
\end{definition}

Recall, that an exact symplectic manifold is a pair $(M,\lambda)$ such that $(M,d\lambda)$ is a symplectic manifold. Moreover, an exact symplectomorphism of an exact symplectic manifold $(M,\lambda)$ is a diffeomorphism $\varphi \in \Diff(M)$ such that $\varphi^*\lambda - \lambda$ is exact.

\begin{definition}[Perturbed Twisted Rabinowitz Action Functional]
	Let $(M,\lambda)$ be an exact symplectic manifold and $\varphi \in \Diff(M)$ an exact symplectomorphism such that $\varphi^*\lambda - \lambda = df$. For $H,F \in C^\infty_\varphi(M \times \mathbb{R})$ define the \bld{perturbed twisted Rabinowitz action functional}
	\begin{equation*}
		\mathscr{A}^{(H,F)}_\varphi \colon \mathscr{L}_\varphi M \times \mathbb{R} \to \mathbb{R}
	\end{equation*}
	\noindent by
	\begin{equation*}
		\mathscr{A}^{(H,F)}_\varphi(\gamma,\tau) := \int_0^1 \gamma^*\lambda - \tau \int_0^1 H_t(\gamma(t))dt - \int_0^1 F_t(\gamma(t))dt - f(\gamma(0)).
	\end{equation*}
	If $F = 0$ and $H \in C^\infty_\varphi(M)$, we write $\mathscr{A}^H_\varphi$ for $\mathscr{A}^{(H,F)}_\varphi$ and call $\mathscr{A}^H_\varphi$ the \bld{twisted Rabinowitz action functional}.
	\label{def:perturbed_twisted_Rabinowitz_functional}
\end{definition}

\begin{remark}
	\label{rem:finite_order}
	Assume that $m := \ord \varphi < \infty$. Then
	\begin{equation*}
		\mathscr{A}^{(H,F)}_\varphi(\gamma,\tau) = \frac{1}{m} \mathscr{A}^{(H,F)}(\overline{\gamma},\tau) - \frac{1}{m} \sum_{k = 0}^{m - 1}f(\gamma(k)),
	\end{equation*}
	\noindent for all $(\gamma,\tau) \in \mathscr{L}_\varphi M$, where $\overline{\gamma} \in \mathscr{L} M$ is defined by $\overline{\gamma}(t) := \gamma(mt)$.
\end{remark}

\begin{definition}[Differential of the Perturbed Twisted Rabinowitz Action Functional]
	Let $\varphi \in \Diff(M)$ be an exact symplectomorphism of an exact symplectic manifold $(M,\lambda)$. For $H,F \in C^\infty_\varphi(M \times \mathbb{R})$, define the \bld{differential of the perturbed twisted Rabinowitz action functional}
	\begin{equation*}
		d\mathscr{A}^{(H,F)}_\varphi\vert_{(\gamma,\tau)} \colon T_\gamma \mathscr{L}_\varphi M \times \mathbb{R} \to \mathbb{R}
	\end{equation*}
	\noindent for all $(\gamma,\tau) \in \mathscr{L}_\varphi M \times \mathbb{R}$ by
	\begin{equation*}
		d\mathscr{A}^{(H,F)}_\varphi\vert_{(\gamma,\tau)}(X,\eta) := \frac{d}{d\varepsilon}\bigg\vert_{\varepsilon = 0}\mathscr{A}^{(H,F)}_\varphi(\gamma_\varepsilon, \tau + \varepsilon \eta),
	\end{equation*}
	\noindent where $\gamma_\varepsilon$ is a twisted variation of $\gamma$ such that $\delta \gamma = X$.
\end{definition}

\begin{proposition}[Differential of the Perturbed Twisted Rabinowitz Action Functional]
	\label{prop:differential_perturbed_twisted_Rabinowitz_functional}
	Let $\varphi \in \Diff(M)$ be an exact symplectomorphism of an exact symplectic manifold $(M,\lambda)$ and $H,F \in C^\infty_\varphi(M \times \mathbb{R})$. Then
	\begin{multline}
		\label{eq:differential_perturbed_twisted_Rabinowitz_functional}
		d\mathscr{\mathscr{A}}^{(H,F)}_\varphi \vert_{(\gamma,\tau)}(X,\eta) = \int_0^1 d\lambda(X(t),\dot{\gamma}(t) - \tau X_{H_t}(\gamma(t)) -  X_{F_t}(\gamma(t)))dt\\
		- \eta\int_0^1 H_t(\gamma(t))dt	
	\end{multline}
	\noindent for all $(\gamma,\tau) \in \mathscr{L}_\varphi M \times \mathbb{R}$ and $(X,\eta) \in T_\gamma\mathscr{L}_\varphi M \times \mathbb{R}$. Moreover, $(\gamma,\tau) \in \Crit \mathscr{A}^{(H,F)}_\varphi$ if and only if
	\begin{equation}
		\dot{\gamma}(t) = \tau X_{H_t}(\gamma(t)) + X_{F_t}(\gamma(t)) \qquad \text{and} \qquad \int_0^1 H_t(\gamma(t))dt = 0
		\label{eq:critical_points}
	\end{equation}
	\noindent for all $t \in \mathbb{R}$.
\end{proposition}

\begin{proof}
	In order to show \eqref{eq:differential_perturbed_twisted_Rabinowitz_functional}, we compute
	\allowdisplaybreaks
	\begin{align*}
		d\mathscr{\mathscr{A}}^{(H,F)}_\varphi \vert_{(\gamma,\tau)}(X,\eta) =& \frac{d}{d\varepsilon}\bigg\vert_{\varepsilon = 0}\mathscr{A}^{(H,F)}_\varphi(\gamma_\varepsilon, \tau + \varepsilon \eta)\\
		=& \int_0^1 \gamma^* i_X d\lambda + \int_0^1 di_X\lambda - \tau \int_0^1 dH_t(X(t))dt\\
		&- \eta\int_0^1 H_t(\gamma(t))dt - \int_0^1 dF_t(X(t))dt - df_\varphi(X(0))\\
		=& \int_0^1 d\lambda(X(t), \dot{\gamma}(t) - \tau X_{H_t}(\gamma(t)) - X_{F_t}(\gamma(t)))dt\\
		& - \eta\int_0^1 H_t(\gamma(t))dt + \lambda(X)\vert_0^1 - df_\varphi(X(0))\\
		=& \int_0^1 d\lambda(X(t), \dot{\gamma}(t) - \tau X_{H_t}(\gamma(t)) - X_{F_t}(\gamma(t)))dt\\
		& - \eta\int_0^1 H_t(\gamma(t))dt + (\varphi^*\lambda - \lambda)(X(0)) - df_\varphi(X(0))\\
		=& \int_0^1 d\lambda(X(t), \dot{\gamma}(t) - \tau X_{H_t}(\gamma(t)) - X_{F_t}(\gamma(t)))dt\\
		& - \eta\int_0^1 H_t(\gamma(t))dt.
	\end{align*}
	Let $(\gamma,\tau) \in \Crit \mathscr{A}^{(H,F)}_\varphi$. It follows immediately from \eqref{eq:differential_perturbed_twisted_Rabinowitz_functional} that
	\begin{equation*}
		\int_0^1 H_t(\gamma(t))dt = 0
	\end{equation*}
	\noindent and
	\begin{equation*}
		\int_0^1 d\lambda(X(t),\dot{\gamma}(t) - \tau X_{H_t}(\gamma(t)) -  X_{F_t}(\gamma(t)))dt = 0
	\end{equation*}
	\noindent for all $X \in T_\gamma \mathscr{L}_\varphi M$. Suppose there exists $t_0 \in \Int I$ such that
	\begin{equation*}
		\dot{\gamma}(t_0) - \tau X_{H_{t_0}}(\gamma(t_0)) -  X_{F_{t_0}}(\gamma(t_0)) \neq 0.
	\end{equation*}
	By nondegeneracy of the symplectic form $d\lambda$ there exists $v \in T_{\gamma(t_0)}M$ with
	\begin{equation*}
		d\lambda(v,\dot{\gamma}(t_0) - \tau X_{H_{t_0}}(\gamma(t_0)) -  X_{F_{t_0}}(\gamma(t_0))) \neq 0.
	\end{equation*}
	Fix a Riemannian metric on $M$ and let $X_v$ denote the unique parallel vector field along $\gamma \vert_I$ such that $X_v(t_0) = v$. As $\Int I$ is open, there exists $\delta > 0$ such that $\overline{B}_\delta(t_0) \subseteq \Int I$. Fix a smooth bump function $\beta \in C^\infty(I)$ for $t_0$ supported in $B_\delta(t_0)$. By shrinking $\delta$ if necessary, we may assume that
	\begin{equation*}
		\int_{t_0 - \delta}^{t_0 + \delta} d\lambda(\beta(t)X_v(t),\dot{\gamma}(t) - \tau X_{H_t}(\gamma(t)) -  X_{F_t}(\gamma(t)))dt \neq 0.
	\end{equation*}
	Extending
	\begin{equation*}
		(\beta X_v)(t + k) := D\varphi^k(\beta(t)X_v(t)) \qquad \forall t \in I, k \in \mathbb{Z},
	\end{equation*}
	\noindent we have that $\beta X_v \in T_\gamma \mathscr{L}_\varphi M$ and thus we compute
	\begin{align*}
		0 &= d\mathscr{\mathscr{A}}^{(H,F)}_\varphi \vert_{(\gamma,\tau)}(\beta X_v,0)\\
		&= \int_{t_0 - \delta}^{t_0 + \delta} d\lambda(\beta(t)X_v(t),\dot{\gamma}(t) - \tau X_{H_t}(\gamma(t)) -  X_{F_t}(\gamma(t)))dt\\
		&\neq 0.
	\end{align*}
	Hence
	\begin{equation*}
		\dot{\gamma}(t) = \tau X_{H_t}(\gamma(t)) + X_{F_t}(\gamma(t)) \qquad \forall t \in I,
	\end{equation*}
	\noindent implying
	\begin{align*}
		\dot{\gamma}(t + k) &= D\varphi^k(\dot{\gamma}(t))\\
		&= \tau (D\varphi^k \circ X_{H_t})(\gamma(t)) + (D\varphi^k \circ X_{F_t})(\gamma(t))\\
		&= \tau (D\varphi^k \circ X_{H_t} \circ \varphi^{-k} \circ \varphi^k)(\gamma(t)) + (D\varphi^k \circ X_{F_t} \circ \varphi^{-k} \circ \varphi^k)(\gamma(t))\\
		&= \tau \varphi^k_* X_{H_t}(\gamma(t + k)) + \varphi^k_* X_{F_t}(\gamma(t + k))\\
		&= \tau X_{\varphi^k_*H_t}(\gamma(t + k)) + X_{\varphi^k_*F_t}(\gamma(t + k))\\
		&= \tau X_{H_{t + k}}(\gamma(t + k)) + X_{F_{t + k}}(\gamma(t + k))
	\end{align*}
	\noindent for all $t \in I$ and $k \in \mathbb{Z}$. The other direction is immediate.
\end{proof}

\begin{corollary}
	The differential of the perturbed twisted Rabinowitz action functional is well-defined, that is, independent of the choice of twisted variation, and linear.
\end{corollary}

Preservation of energy \ref{cor:preservation_of_energy} implies the following corollary.

\begin{corollary}
	\label{cor:critical_points_perturbed_twisted_Rabinowitz_functional}
	Let $\varphi \in \Diff(M)$ be an exact symplectomorphism of an exact symplectic manifold $(M,\lambda)$ and $H \in C^\infty_\varphi(M)$. Then $\Crit\mathscr{A}^H_\varphi$ consists precisely of all $(\gamma,\tau) \in \mathscr{L}_\varphi M \times \mathbb{R}$ such that $\gamma(\mathbb{R}) \subseteq H^{-1}(0)$ and $\gamma$ is an integral curve of $\tau X_H$.	
\end{corollary}

There is a natural $\mathbb{R}$-action on the twisted loop space $\mathscr{L}_\varphi M$ given by
\begin{equation*}
	(s \cdot \gamma)(t) := \gamma(t + s) \qquad \forall t \in \mathbb{R}.
\end{equation*}
If $(M,\lambda)$ is an exact symplectic manifold and $H \in C^\infty_\varphi(M)$ for an exact symplectomorphism $\varphi \in \Diff(M)$ of finite order such that $\supp f \cap H^{-1}(0) = \emptyset$, then the twisted Rabinowitz action functional $\mathscr{A}^H_\varphi$ is invariant under the induced $\mathbb{S}^1$-action on $\Crit \mathscr{A}^H_\varphi$. In particular, the unperturbed twisted Rabinowitz action functional is never a Morse function.

\begin{definition}[Hessian of the Twisted Rabinowitz Action Functional]
	Let $\varphi$ be an exact symplectomorphism of an exact symplectic manifold $(M,\lambda)$ and suppose that $H \in C^\infty_\varphi(M)$. For $(\gamma,\tau) \in \Crit \mathscr{A}^H_\varphi$, define the \bld{Hessian of the twisted Rabinowitz action functional} 
	\begin{equation*}
		\Hess \mathscr{A}^H_\varphi\vert_{(\gamma,\tau)} \colon (T_\gamma\mathscr{L}_\varphi M \times \mathbb{R}) \times (T_\gamma \mathscr{L}_\varphi M \times \mathbb{R}) \to \mathbb{R}
	\end{equation*}
	\noindent by
	\begin{equation*}
		\Hess \mathscr{A}^H_\varphi\vert_{(\gamma,\tau)}((X,\eta),(Y,\sigma)) := \frac{\partial^2}{\partial\varepsilon_1\partial\varepsilon_2}\bigg\vert_{\varepsilon_1 = \varepsilon_2 = 0}\mathscr{A}^H_\varphi(\gamma_{\varepsilon_1,\varepsilon_2},\tau + \varepsilon_1\eta + \varepsilon_2\sigma),
	\end{equation*}
	\noindent for a smooth two-parameter family $\gamma_{\varepsilon_1,\varepsilon_2}$ of twisted loops with
	\begin{equation*}
		\frac{\partial}{\partial \varepsilon_1}\bigg\vert_{\varepsilon_1 = 0}\gamma_{\varepsilon_1,0} = X \qquad \text{and} \qquad \frac{\partial}{\partial \varepsilon_2}\bigg\vert_{\varepsilon_2 = 0}\gamma_{0,\varepsilon_2} = Y.
	\end{equation*}
\end{definition}

\begin{remark}
	Traditionally, the differential and the Hessian of the twisted Rabinowitz action functional are called the first and second variation of the twisted Rabinowitz action functional.
\end{remark}

\begin{definition}[Symplectic Connection]
	Let $(M,\omega)$ be a symplectic manifold. A \bld{symplectic connection on $(M,\omega)$} is defined to be a torsion-free connection $\nabla$ in the tangent bundle $TM$ such that $\nabla \omega = 0$.
\end{definition}

\begin{remark}
	Every symplectic manifold admits a symplectic connection by \cite[p.~308]{gutt:symplectic:2006}, but in sharp contrast to the Riemannian case, a symplectic connection on a given symplectic manifold is in general not unique.
\end{remark}

\begin{lemma}
	Let $\varphi \in \Diff(M)$ be an exact symplectomorphism of an exact symplectic manifold $(M,\lambda)$. Fix a symplectic connection $\nabla$ on $(M,d\lambda)$ and a twisted Hamiltonian function $H \in C^\infty_\varphi(M)$. If $(\gamma,\tau) \in \Crit \mathscr{A}^H_\varphi$, then
	\begin{multline}
		\Hess \mathscr{A}^H_\varphi\vert_{(\gamma,\tau)}((X,\eta),(Y,\sigma)) = \int_0^1 d\lambda(Y,\nabla_t X)\\ - \tau\int_0^1 \Hess^\nabla H(X,Y) - \eta\int_0^1 dH(Y) - \sigma \int_0^1 dH(X)
		\label{eq:formula_twisted_Hessian}
	\end{multline}
	\noindent for all $(X,\eta),(Y,\sigma) \in T_\gamma \mathscr{L}_\varphi M \times \mathbb{R}$.
	\label{lem:twisted_Hessian}
\end{lemma}

\begin{proof}
	We compute
	\allowdisplaybreaks
	\begin{multline}
		\label{eq:Hessian_second_term}
		\frac{\partial^2}{\partial \varepsilon_1 \partial \varepsilon_2}\bigg\vert_{\varepsilon_1 = \varepsilon_2 = 0} (\tau + \varepsilon_1 \eta + \varepsilon_2 \sigma)\int_0^1 H \circ \gamma_{\varepsilon_1,\varepsilon_2}\\
		=\eta \int_0^1 dH(Y) + \sigma \int_0^1 dH(X) + \tau\int_0^1 \frac{\partial^2}{\partial \varepsilon_1 \partial \varepsilon_2}\bigg\vert_{\varepsilon_1 = \varepsilon_2 = 0} H(\gamma_{\varepsilon_1,\varepsilon_2}),
	\end{multline}
	\noindent and
	\begin{align*}
		\frac{\partial^2}{\partial \varepsilon_1 \partial \varepsilon_2}\bigg\vert_{\varepsilon_1 = \varepsilon_2 = 0} H(\gamma_{\varepsilon_1,\varepsilon_2}) =& \frac{\partial}{\partial\varepsilon_1}\bigg\vert_{\varepsilon_1 = 0}dH\del[1]{\partial_{\varepsilon_2}\vert_{\varepsilon_2 = 0}\gamma_{\varepsilon_1,\varepsilon_2}}\\
		=& -\frac{\partial}{\partial\varepsilon_1}\bigg\vert_{\varepsilon_1 = 0}d\lambda\del[1]{X_H(\gamma_{\varepsilon_1,0}),\partial_{\varepsilon_2}\vert_{\varepsilon_2 = 0}\gamma_{\varepsilon_1,\varepsilon_2}}\\
		=& -d\lambda\del[1]{\nabla_{\varepsilon_1}\vert_{\varepsilon_1 = 0}X_H(\gamma_{\varepsilon_1,0}),Y}\\
		&-d\lambda\del[1]{X_H(\gamma),\nabla_{\varepsilon_1}\vert_{\varepsilon_1 = 0}\partial_{\varepsilon_2}\vert_{\varepsilon_2 = 0}\gamma_{\varepsilon_1,\varepsilon_2}}.
	\end{align*}
	The $d\lambda$-compatibility of $\nabla$ implies
	\allowdisplaybreaks
	\begin{align}
		\label{eq:covariant_Hessian}
		d\lambda\del[1]{\nabla_{\varepsilon_1}\vert_{\varepsilon_1 = 0}X_H(\gamma_{\varepsilon_1,0}),Y} &= d\lambda\del[1]{\nabla_{\partial_{\varepsilon_1}\vert_{\varepsilon_1 = 0}\gamma_{\varepsilon_1,0}}X_H,Y}\nonumber\\
		&= d\lambda(\nabla_{X}X_H,Y)\nonumber\\
		&= \nabla_Xd\lambda(X_H,Y) - d\lambda(X_H,\nabla_XY)\nonumber\\
		&= -\nabla_{X}dH(Y) + dH(\nabla_XY)\nonumber\\
		&= -X(Y(H)) + (\nabla_XY)H\nonumber\\
		&= -\Hess^\nabla H\vert_\gamma(X,Y).
	\end{align}
	Next we compute
	\allowdisplaybreaks
	\begin{align}
		\label{eq:Hessian_first_term}
		\frac{\partial^2}{\partial \varepsilon_1 \partial \varepsilon_2}\bigg\vert_{\varepsilon_1 = \varepsilon_2 = 0} \int_0^1 \gamma_{\varepsilon_1,\varepsilon_2}^*\lambda =& \int_0^1 \frac{\partial}{\partial \varepsilon_1}\bigg\vert_{\varepsilon_1 = 0}\gamma_{\varepsilon_1,0}^*i_{\partial_{\varepsilon_2}\vert_{\varepsilon_2 = 0}\gamma_{\varepsilon_1,\varepsilon_2}}(d\lambda \circ \gamma_{\varepsilon_1,0})\nonumber\\
		&+ \int_0^1 \frac{\partial}{\partial \varepsilon_1}\bigg\vert_{\varepsilon_1 = 0} d\gamma^*_{\varepsilon_1,0}i_{\partial_{\varepsilon_2}\vert_{\varepsilon_2 = 0}\gamma_{\varepsilon_1,\varepsilon_2}}(\lambda \circ \gamma_{\varepsilon_1,0})\nonumber\\
		=& \int_0^1 \frac{\partial}{\partial \varepsilon_1}\bigg\vert_{\varepsilon_1 = 0} d\lambda\del[1]{\partial_{\varepsilon_2}\vert_{\varepsilon_2 = 0}\gamma_{\varepsilon_1,\varepsilon_2},\dot{\gamma}_{\varepsilon_1,0}}\nonumber\\
		&+ \frac{\partial}{\partial \varepsilon_1}\bigg\vert_{\varepsilon_1 = 0}\lambda(\partial_{\varepsilon_2}\vert_{\varepsilon_2 = 0}\gamma_{\varepsilon_1,\varepsilon_2})\vert_0^1\nonumber\\
		=& \int_0^1 d\lambda\del[1]{\nabla_{\varepsilon_1}\vert_{\varepsilon_1 = 0}\partial_{\varepsilon_2}\vert_{\varepsilon_2 = 0}\gamma_{\varepsilon_1,\varepsilon_2},\dot{\gamma}}\nonumber\\
		&+ \int_0^1 d\lambda\del[1]{Y,\nabla_{\varepsilon_1}\vert_{\varepsilon_1 = 0}\partial_t\gamma_{\varepsilon_1,0}}\nonumber\\
		&+ \frac{\partial}{\partial \varepsilon_1}\bigg\vert_{\varepsilon_1 = 0}\lambda(\partial_{\varepsilon_2}\vert_{\varepsilon_2 = 0}\gamma_{\varepsilon_1,\varepsilon_2})\vert_0^1\nonumber\\
		=& \tau\int_0^1 d\lambda\del[1]{\nabla_{\varepsilon_1}\vert_{\varepsilon_1 = 0}\partial_{\varepsilon_2}\vert_{\varepsilon_2 = 0}\gamma_{\varepsilon_1,\varepsilon_2},X_H(\gamma)}\nonumber\\
		&+ \int_0^1 d\lambda\del[1]{Y,\nabla_t\partial_{\varepsilon_1}\vert_{\varepsilon_1 = 0}\gamma_{\varepsilon_1,0}}\nonumber\\
		&+ \frac{\partial}{\partial \varepsilon_1}\bigg\vert_{\varepsilon_1 = 0}\lambda(\partial_{\varepsilon_2}\vert_{\varepsilon_2 = 0}\gamma_{\varepsilon_1,\varepsilon_2})\vert_0^1\nonumber\\
		=& \>\tau\int_0^1 d\lambda\del[1]{\nabla_{\varepsilon_1}\vert_{\varepsilon_1 = 0}\partial_{\varepsilon_2}\vert_{\varepsilon_2 = 0}\gamma_{\varepsilon_1,\varepsilon_2},X_H(\gamma)}\nonumber\\
		&+ \int_0^1 d\lambda\del[1]{Y,\nabla_tX}\nonumber\\
		&+ \frac{\partial}{\partial \varepsilon_1}\bigg\vert_{\varepsilon_1 = 0}\lambda(\partial_{\varepsilon_2}\vert_{\varepsilon_2 = 0}\gamma_{\varepsilon_1,\varepsilon_2})\vert_0^1.
	\end{align}
	Moreover
	\allowdisplaybreaks
	\begin{align}
		\label{eq:Hessian_symmetric}
		\int_0^1 d\lambda\del[1]{Y,\nabla_t X} =& \int_0^1 \frac{d}{dt}d\lambda(Y,X) - \int_0^1 d\lambda(\nabla_t Y, X)\nonumber\\
		=&\> d\lambda(Y,X)\vert_0^1 + \int_0^1 d\lambda(X,\nabla_t Y)\nonumber\\ 
		=&\> d\lambda\del[1]{D\varphi(Y(0)),D\varphi(X(0))} - d\lambda(Y(0),X(0))\nonumber\\ 
		&+ \int_0^1 d\lambda(X,\nabla_t Y)\nonumber\\
		=& \int_0^1 d\lambda(X,\nabla_t Y).
	\end{align}
	\noindent Finally
	\allowdisplaybreaks
	\begin{align}
		\label{eq:Hessian_third_term}
		\lambda(\partial_{\varepsilon_2}\vert_{\varepsilon_2 = 0}\gamma_{\varepsilon_1,\varepsilon_2})\vert_0^1 &= \lambda\del[1]{\partial_{\varepsilon_2}\vert_{\varepsilon_2 = 0}\gamma_{\varepsilon_1,\varepsilon_2}(1)} - \lambda\del[1]{\partial_{\varepsilon_2}\vert_{\varepsilon_2 = 0}\gamma_{\varepsilon_1,\varepsilon_2}(0)}\nonumber\\
		&= \lambda\del[1]{\partial_{\varepsilon_2}\vert_{\varepsilon_2 = 0}\varphi(\gamma_{\varepsilon_1,\varepsilon_2}(0))} - \lambda\del[1]{\partial_{\varepsilon_2}\vert_{\varepsilon_2 = 0}\gamma_{\varepsilon_1,\varepsilon_2}(0)}\nonumber\\
		&= \lambda\del[1]{D\varphi\del[1]{\partial_{\varepsilon_2}\vert_{\varepsilon_2 = 0}\gamma_{\varepsilon_1,\varepsilon_2}(0)}} - \lambda\del[1]{\partial_{\varepsilon_2}\vert_{\varepsilon_2 = 0}\gamma_{\varepsilon_1,\varepsilon_2}(0)}\nonumber\\
		&= (\varphi^*\lambda - \lambda)(\partial_{\varepsilon_2}\vert_{\varepsilon_2 = 0}\gamma_{\varepsilon_1,\varepsilon_2}(0))\nonumber\\
		&= df_\varphi(\partial_{\varepsilon_2}\vert_{\varepsilon_2 = 0}\gamma_{\varepsilon_1,\varepsilon_2}(0)).
	\end{align}
	Combining \eqref{eq:Hessian_first_term}, \eqref{eq:Hessian_second_term}, \eqref{eq:Hessian_symmetric}, and \eqref{eq:Hessian_third_term} yields \eqref{eq:formula_twisted_Hessian}.
\end{proof}

\begin{corollary}
	The Hessian of the twisted Rabinowitz action functional is a well-defined, that is, independent of the choice of twisted two-parameter family, symmetric bilinear form.
\end{corollary}


\begin{lemma}
	\label{lem:twisted_Hessian}
	Let $\varphi \in \Diff(M)$ be an exact symplectomorphism of an exact symplectic manifold $(M,\lambda)$ and $H \in C^\infty_\varphi(M)$. If $(\gamma,\tau) \in \Crit \mathscr{A}^H_\varphi$, then
	\begin{multline}
		\label{eq:formula_twisted_Hessian_Lie}
		\Hess \mathscr{A}^H_\varphi\vert_{(\gamma,\tau)}((X,\eta),(Y,\sigma)) = \int_0^1 d\lambda(Y,L_{\tau X_H}X - \eta X_H(\gamma))\\ - \sigma\int_0^1 dH(X)
	\end{multline}
	\noindent for all $(X,\eta),(Y,\sigma) \in T_\gamma \mathscr{L}_\varphi M \times \mathbb{R}$, where
	\begin{equation*}
		L_{\tau X_H}X(t) = \frac{d}{ds}\bigg\vert_{s = 0} D\phi^{X_H}_{-s\tau}X(s + t) \qquad \forall t \in \mathbb{R}.
	\end{equation*}
\end{lemma}

\begin{proof}
	Inserting $\Hess^\nabla(X,Y) = d\lambda(Y,\nabla_X X_H)$ into \eqref{eq:formula_twisted_Hessian} yields
	\begin{multline*}
		\Hess \mathscr{A}^H_\varphi\vert_{(\gamma,\tau)}((X,\eta),(Y,\sigma)) = \int_0^1 d\lambda(Y,\nabla_t X - \tau\nabla_X X_H)\\ - \eta \int_0^1 dH(Y) - \sigma \int_0^1 dH(X).
	\end{multline*}
	But as $\nabla$ has no torsion by assumption, we compute
	\begin{equation*}
		\nabla_t X - \tau \nabla_X X_H = \nabla_{\dot{\gamma}}X - \tau \nabla_X X_H = \nabla_{\tau X_H}X - \tau\nabla_X X_H = [\tau X_H,X],
	\end{equation*}
	\noindent and 
	\allowdisplaybreaks
	\begin{align*}
		[\tau X_H,X](t) &= L_{\tau X_H}X(t)\\
		&= \frac{d}{ds}\bigg\vert_{s = 0} D\phi_{-s\tau}^{X_H}(X(\phi_{s\tau}^{X_H}(\gamma(t)))\\
		&= \frac{d}{ds}\bigg\vert_{s = 0} D\phi_{-s\tau}^{X_H}(X(\phi_{s\tau}^{X_H}(\phi_{t\tau}^{X_H}(\gamma(0)))))\\
		&= \frac{d}{ds}\bigg\vert_{s = 0} D\phi_{-s\tau}^{X_H} X(s + t)
	\end{align*}
	\noindent for all $t \in \mathbb{R}$.
\end{proof}

\begin{corollary}
	\label{cor:kernel_Hessian}
	Let $\varphi \in \Diff(M)$ be an exact symplectomorphism of an exact symplectic manifold $(M,\lambda)$ and let $H \in C^\infty_\varphi(M)$. The kernel of the Hessian of the twisted Rabinowitz action functional $\mathscr{A}^H_\varphi$ at $(\gamma,\tau) \in \Crit \mathscr{A}^H_\varphi$ consists precisely of all $(X,\eta) \in T_\gamma \mathscr{L}_\varphi M \times \mathbb{R}$ satisfying
	\begin{equation*}
		L_{\tau X_H} X = \eta X_H(\gamma) \qquad \text{and} \qquad \int_0^1 dH(X) = 0.
	\end{equation*}
\end{corollary}


\begin{lemma}
	\label{lem:kernel_of_the_Hessian}
	Let $\varphi \in \Diff(M)$ be an exact symplectomorphism of an exact symplectic manifold $(M,\lambda)$ and $H \in C^\infty_\varphi(M)$. For every $(\gamma,\tau) \in \Crit \mathscr{A}^H_\varphi$, there is a canonical isomorphism
	\begin{equation}
		\ker \Hess \mathscr{A}^H_\varphi\vert_{(\gamma,\tau)} \cong \mathfrak{K}(\gamma,\tau),
		\label{eq:canonical_isomorphism}
	\end{equation}
	\noindent where
	\begin{equation*}
		\mathfrak{K}(\gamma,\tau) := \cbr[0]{(v_0,\eta) \in T_{\gamma(0)}M \times \mathbb{R} : \textup{solution of \eqref{eq:initial_value_system}}}
	\end{equation*}
	\noindent with
	\begin{equation}
		\label{eq:initial_value_system}
		D(\phi_{-\tau}^{X_H} \circ \varphi)v_0 = v_0 + \eta X_H(\gamma(0)) \qquad \text{and} \qquad dH(v_0) = 0.
	\end{equation}	
\end{lemma}

\begin{proof}
	We follow \cite[p.~99--100]{frauenfelderkoert:3bp:2018}. Let $(X,\eta) \in \ker \Hess \mathscr{A}^H_\varphi\vert_{(\gamma,\tau)}$ and define
	\begin{equation*}
		v \colon I \to T_{\gamma(0)}M, \qquad v(t) := D\phi_{-\tau t}^{X_H} X(t).
	\end{equation*}
	We claim that
	\begin{equation}
		\ker \Hess \mathscr{A}^H_\varphi\vert_{(\gamma,\tau)} \to \mathfrak{K}(\gamma,\tau), \qquad (X,\eta) \mapsto (v(0),\eta)
		\label{eq:canonical_isomorphism}
	\end{equation}
	\noindent is an isomorphism. First, we show that the above homomorphism is indeed well-defined. The assumption that $(X,\eta)$ lies in the kernel of the Hessian of the twisted Rabinowitz action functional at the critical point $(\gamma,\tau)$ is by Corollary \ref{cor:kernel_Hessian} equivalent to the system
	\begin{equation}
		\label{eq:kernel_of_the_Hessian}
		\dot{v} = \eta X_H(\gamma(0)) \qquad \text{and} \qquad \int_0^1 dH(v) = 0.
	\end{equation}
	Integrating the first equation yields
	\begin{equation*}
		v(t) = v_0 + t\eta X_H(\gamma(0)) \qquad \forall t \in I,
	\end{equation*}
	\noindent with $v_0 :=v(0)$. Thus $(v_0,\eta) \in \mathfrak{K}(\gamma,\tau)$ follows from
	\begin{align}
		\label{eq:twist_condition}
		v(1) &= D\phi_{-\tau}^{X_H} X(1)\nonumber\\
		&= D\phi_{-\tau}^{X_H} D\varphi(X(0))\nonumber\\
		&= D(\phi_{-\tau}^{X_H} \circ \varphi)X(0)\nonumber\\
		&= D(\phi_{-\tau}^{X_H} \circ \varphi)v_0.
	\end{align}
	That \eqref{eq:canonical_isomorphism} is an isomorphism follows by considering the inverse
	\begin{equation*}
		\mathfrak{K}(\gamma,\tau) \to \ker \Hess \mathscr{A}^H_\varphi\vert_{(\gamma,\tau)}, \qquad (v_0,\eta) \mapsto (X,\eta),
	\end{equation*}
	\noindent where $X \in T_\gamma \mathscr{L}_\varphi M$ is defined by
	\begin{equation*}
		X(t) := D\phi^{X_H}_{\tau t}(v_0 + t\eta X_H(\gamma(0))) \qquad \forall t \in \mathbb{R}.
	\end{equation*}
	This establishes the canonical isomorphism \eqref{eq:canonical_isomorphism}.
\end{proof}

Recall, that a strict contactomorphism of a contact manifold $(\Sigma,\alpha)$ is defined to be a diffeomorphism $\varphi \in \Diff(\Sigma)$ such that $\varphi^*\alpha = \alpha$. 

\begin{definition}[Parametrised Twisted Reeb Orbit]
	For a contact manifold $(\Sigma,\alpha)$ and a strict contactomorphism $\varphi \colon (\Sigma,\alpha) \to (\Sigma,\alpha)$ define the set of \bld{parametrised twisted Reeb orbits on $(\Sigma,\alpha)$} by
	\begin{equation*}
		\mathscr{P}_\varphi(\Sigma,\alpha) := \cbr[0]{(\gamma,\tau) \in \mathscr{L}_\varphi \Sigma \times \mathbb{R} : \dot{\gamma}(t) = \tau R(\gamma(t)) \> \forall t \in \mathbb{R}}.
	\end{equation*}
\end{definition}

\begin{definition}[Twisted Spectrum]
	For a contact manifold $(\Sigma,\alpha)$ and a strict contactomorphism $\varphi \colon (\Sigma,\alpha) \to (\Sigma,\alpha)$ define the \bld{twisted spectrum} by
		\begin{equation*}
		\Spec(\Sigma,\alpha) := \cbr[0]{\tau \in \mathbb{R} : \exists \gamma \in \mathscr{L}_\varphi \Sigma \> \text{such that } (\gamma,\tau) \in \mathscr{P}_\varphi(\Sigma,\alpha)}.
	\end{equation*}
\end{definition}

\begin{lemma}
	\label{lem:strict_contactomorphism}
	Let $\varphi \colon(\Sigma,\alpha) \to (\Sigma,\alpha)$ be a strict contactomorphism of a compact contact manifold $(\Sigma,\alpha)$. Then
	\begin{equation*}
		\varphi \circ \phi^R_t = \phi^R_t \circ \varphi \qquad \forall t \in \mathbb{R}.
	\end{equation*}	
\end{lemma}

\begin{proof}
	If $\varphi_* R = R$, then we compute
	\begin{equation*}
		\frac{d}{dt} \varphi \circ \phi^R_t = D\varphi \circ R \circ \phi_t^R = \varphi_*R \circ \varphi \circ \phi_t^R= R \circ \varphi \circ \phi_t^R,
	\end{equation*}
	\noindent for all $t \in \mathbb{R}$. To prove $\varphi_*R = R$, just observe that
	\begin{align*}
		i_{\varphi_*R}d\alpha &= \varphi^*d\alpha(R \circ \varphi^{-1}, D\varphi^{-1} \cdot)\\
		&= d\varphi^*\alpha(R \circ \varphi^{-1}, D\varphi^{-1} \cdot)\\
		&= d\alpha(R \circ \varphi^{-1}, D\varphi^{-1} \cdot)\\
		&= \varphi_*(i_R d\alpha)\\
		&= 0,
	\end{align*}
	\noindent and
	\begin{equation*}
		\alpha(\varphi_*R) = \alpha(D\varphi \circ R \circ \varphi^{-1}) = \varphi^*\alpha(R \circ \varphi^{-1}) = \alpha(R \circ \varphi^{-1}) = 1.
	\end{equation*}
	Hence the statement follows from the uniqueness of integral curves.
\end{proof}

\begin{proposition}[Kernel of the Hessian of the Twisted Rabinowitz Action Functional]
	\label{prop:kernel_hessian_contact}
	Let $(\Sigma,\lambda\vert_\Sigma)$ be a regular energy surface of restricted contact type in an exact Hamiltonian system $(M,\lambda,H)$ with $X_H\vert_\Sigma = R$. Suppose $\varphi \in \Diff(M)$ is an exact symplectomorphism such that $H \in C^\infty_\varphi(M)$ and $\varphi^*\lambda\vert_\Sigma = \lambda\vert_\Sigma$. Then
	\begin{equation*}
		\Crit \mathscr{A}^H_\varphi = \mathscr{P}_\varphi(\Sigma,\lambda\vert_\Sigma)
	\end{equation*}
	\noindent and
	\begin{equation*}
		\ker \Hess \mathscr{A}^H_\varphi\vert_{(\gamma,\tau)} \cong \ker \del[1]{D(\phi^R_{-\tau} \circ \varphi)\vert_{\gamma(0)} - \id_{T_{\gamma(0)}\Sigma}}
	\end{equation*}
	\noindent for all $ (\gamma,\tau) \in \mathscr{P}_\varphi(\Sigma,\lambda\vert_\Sigma)$. Moreover, we have $R(\gamma(0)) \in \ker \Hess \mathscr{A}^H_\varphi\vert_{(\gamma,\tau)}$ and if $\mathscr{P}_\varphi(\Sigma,\lambda\vert_\Sigma) \subseteq \Sigma \times \mathbb{R}$ is an embedded submanifold, then $\Spec(\Sigma,\lambda\vert_\Sigma)$ is discrete.	
\end{proposition}

\begin{remark}
	\label{rem:period-action_equality}
	If $(\gamma,\tau) \in \mathscr{P}_\varphi(\Sigma,\lambda\vert_\Sigma)$, we have the period-action equality
	\begin{equation*}
		\mathscr{A}^H_\varphi(\gamma,\tau) = \int_0^1 \gamma^* \lambda = \int_0^1 \lambda(\dot{\gamma}) = \tau \int_0^1 \lambda(R(\gamma)) = \tau.
	\end{equation*}	
\end{remark}

\begin{proof}
	The identity $\Crit \mathscr{A}^H_\varphi = \mathscr{P}_\varphi(\Sigma,\lambda\vert_\Sigma)$ immediately follows from Corollary \ref{cor:critical_points_perturbed_twisted_Rabinowitz_functional} together with \cite[Corollary~5.30]{lee:dt:2012}. The proof of the formula for the kernel of the Hessian of $\mathscr{A}^H_\varphi$ is inspired by \cite[p.~102]{frauenfelderkoert:3bp:2018}. By Lemma \ref{lem:kernel_of_the_Hessian} we have that
	\begin{equation*}
		\ker \Hess \mathscr{A}^H_\varphi\vert_{(\gamma,\tau)} \cong \mathfrak{K}(\gamma,\tau),
	\end{equation*}
	\noindent where $(v_0,\eta) \in T_{\gamma(0)} M \times \mathbb{R}$ belongs to $\mathfrak{K}(\gamma,\tau)$ if and only if
	\begin{equation*}
		D(\phi^{X_H}_{-\tau} \circ \varphi)v_0 = v_0 + \eta X_H(\gamma(0)) \qquad \text{and} \qquad dH(v_0) = 0.
	\end{equation*}
	Thus in our setting, the second condition implies $v_0 \in T_{\gamma(0)}\Sigma$. Decompose
	\begin{equation*}
		v_0 = v^\xi_0 + a R(\gamma(0)) \qquad v^\xi_0 \in \xi_{\gamma(0)}, a \in \mathbb{R},
	\end{equation*}
	\noindent where $\xi := \ker \lambda\vert_\Sigma$ denotes the contact distribution. Then we compute
	\begin{align*}
		D(\phi^R_{-\tau} \circ \varphi)R(\gamma(0)) &= D(\phi^R_{-\tau} \circ \varphi)\del[3]{\frac{d}{dt}\bigg\vert_{t = 0} \phi^R_t(\gamma(0))}\\
		&= \frac{d}{dt}\bigg\vert_{t = 0} (\phi^R_{-\tau} \circ \varphi \circ \phi^R_t)(\gamma(0))\\
		&= \frac{d}{dt}\bigg\vert_{t = 0} (\phi^R_t \circ \varphi \circ \phi^R_{-\tau})(\gamma(0))\\
		&= \frac{d}{dt}\bigg\vert_{t = 0} \phi^R_t(\gamma(0))\\
		&= R(\gamma(0)),
	\end{align*}
	\noindent using Lemma \ref{lem:strict_contactomorphism}. Hence
	\begin{equation*}
		v_0 + \eta R(\gamma(0)) = D(\phi^R_{-\tau} \circ \varphi)v_0 = D^\xi(\phi^R_{-\tau} \circ \varphi)v_0^\xi + aR(\gamma(0)),
	\end{equation*}
	\noindent where 
	\begin{equation*}
		D^\xi(\phi^R_{-\tau} \circ \varphi) := D(\phi^R_{-\tau} \circ \varphi)\vert_\xi \colon \xi \to \xi,
	\end{equation*}
	\noindent implies 
	\begin{equation*}
		\eta = 0 \qquad \text{and} \qquad D^\xi(\phi^R_{-\tau} \circ \varphi)v_0^\xi = v_0^\xi
	\end{equation*}
	\noindent by considering the splitting $T\Sigma = \xi \oplus \langle R \rangle$. Consequently
	\begin{equation*}
		\mathfrak{K}(\gamma,\tau) = \ker \del[1]{D(\phi^R_{-\tau} \circ \varphi)\vert_{\gamma(0)} - \id_{T_{\gamma(0)}\Sigma}} \times \cbr[0]{0}.
	\end{equation*}

	Finally, assume that $\mathscr{P}_\varphi(\Sigma,\lambda\vert_\Sigma) \subseteq \Sigma \times \mathbb{R}$ is an embedded submanifold via the obvious identification of $(\gamma,\tau) \in \mathscr{P}_\varphi(\Sigma,\lambda\vert_\Sigma)$ with $(\gamma(0),\tau) \in \Sigma \times \mathbb{R}$. Fix a path $(\gamma_s,\tau_s)$ in $\mathscr{P}_\varphi(\Sigma,\lambda\vert_\Sigma) = \Crit \mathscr{A}^H_\varphi$ from $(\gamma_0,\tau_0)$ to $(\gamma_1,\tau_1)$. Using Remark \ref{rem:period-action_equality} we compute
	\begin{equation*}
		\partial_s \tau_s = \partial_s \mathscr{A}^H_\varphi(\gamma_s,\tau_s) = d\mathscr{A}^H_\varphi\vert_{(\gamma_s,\tau_s)}(\partial_s\gamma_s,\partial_s \tau_s) = 0,
	\end{equation*}
	\noindent implying that $\tau_s$ is constant, and in particular $\tau_0 = \tau_1$. Consequently, $\mathscr{A}^H_\varphi$ is constant on each path-connected component of $\mathscr{P}_\varphi(\Sigma,\lambda\vert_\Sigma)$. As $\mathscr{P}_\varphi(\Sigma,\lambda\vert_\Sigma)$ is a submanifold of $\Sigma \times \mathbb{R}$, there are only countably many connected components by definition, implying that $\Spec(\Sigma,\lambda\vert_\Sigma)$ is discrete.
\end{proof}

\subsection{Compactness of the Moduli Space of Twisted Negative Gradient Flow Lines}
\label{sec:compactness}

\begin{definition}[Twisted Defining Hamiltonian Function]
	Let $(W,\lambda)$ be a Liouville domain with boundary $\Sigma$ and $\varphi \in \Aut(W,\lambda)$. A \bld{twisted defining Hamiltonian function for $\Sigma$} is a Hamiltonian function $H \in C^\infty(M)$ on the completion $(M,\lambda)$ of $(W,\lambda)$, satisfying the following conditions:
	\begin{enumerate}[label=\textup{(\roman*)}]
		\item $H^{-1}(0) = \Sigma$ and $\Sigma \cap \Crit H = \varnothing$.
		\item $H \in C^\infty_\varphi(M)$. 
		\item $dH$ is compactly supported.
		\item $X_H\vert_\Sigma = R$ is the Reeb vector field of the contact form $\lambda\vert_\Sigma$.
	\end{enumerate}
	Denote by $\mathscr{F}_\varphi(\Sigma)$ the set of twisted defining Hamiltonian functions for $\Sigma$.
\end{definition}

\begin{remark}
	A necessary condition for $\mathscr{F}_\varphi(\Sigma) \neq \emptyset$ is that $\varphi^*R = R$. This is not true in general if $\varphi$  does not induce a strict contactomorphism on $\Sigma$.
\end{remark}

\begin{definition}[Adapted Almost Complex Structure]
	Let $(W,\lambda)$ be a Liouville domain with boundary $\Sigma$. An \bld{adapted almost complex structure on $(W,\lambda)$} is a $d\lambda$-compatible almost complex structure $J$ on $(W,\lambda)$ such that $J$ restricts to define a compatible almost complex structure on the contact distribution $\ker \lambda\vert_\Sigma$ and $JR = \partial_r$ holds near the boundary.
\end{definition}

\begin{definition}[Rabinowitz--Floer Data]
	Let $(M,\lambda)$ be the completion of a Liouville domain $(W,\lambda)$ with boundary $\Sigma$ and $\varphi \in \Aut(W,\lambda)$. \bld{Rabinowitz--Floer data for $\varphi$} is defined to be a pair $(H,J)$ consisting of a twisted defining Hamiltonian function $H \in \mathscr{F}_\varphi(\Sigma)$ for $\Sigma$ and an adapted almost complex structure $J$ on $(W,\lambda)$ such that $\varphi^*J = J$.
\end{definition}

\begin{lemma}
	\label{lem:defining_Hamiltonian}
	Let $(W,\lambda)$ be a Liouville domain and $\varphi \in \Aut(W,\lambda)$. Then there exists Rabinowitz--Floer data for $\varphi$.	
\end{lemma}

\begin{proof}
	The construction of the twisted defining Hamiltonian $H$ for $\Sigma$ is inspired by the proof of \cite[Proposition~4.1]{cieliebakfrauenfelderpaternain:mane:2010}. Fix $\delta > 0$ such that the exact symplectic embedding
	\begin{equation*}
		\psi \colon \del[1]{\intoc[0]{-\delta,0} \times \Sigma,e^r\lambda\vert_\Sigma} \hookrightarrow (W,\lambda)
	\end{equation*}
	\noindent defined by
	\begin{equation*}
		\psi(r,x) := \phi^X_r(x)
	\end{equation*}
	\noindent satisfies
	\begin{equation}
		\label{eq:assumption_delta}
		U_\delta := \psi(\intoc[0]{-\delta,0} \times \Sigma) \cap \supp f_\varphi = \emptyset.
	\end{equation}
	Indeed, that $\psi$ is an exact symplectic embedding follows from the computation 
	\begin{align*}
		\frac{d}{dr} \psi^*_r \lambda &= \frac{d}{dr}\del[1]{\phi_r^X}^* \lambda\\
		&= (\phi^X_r)^*L_X\lambda\\
		&= \del[1]{\phi_r^X}^*(di_X\lambda + i_Xd\lambda)\\
		&= \del[1]{\phi_r^X}^*(di_Xi_Xd\lambda + \lambda)\\
		&= \del[1]{\phi_r^X}^* \lambda\\
		&= \psi^*_r \lambda
	\end{align*}
	\noindent implying
	\begin{equation*}
		\psi^*_r\lambda = e^r\lambda\vert_\Sigma \qquad \forall r \in \intoc[0]{-\delta,0},
	\end{equation*}
	\noindent by $\psi_0 = \iota_\Sigma$, where $\iota_\Sigma \colon \Sigma \hookrightarrow W$ denotes the inclusion. Note that $\psi^*_rX = \partial_r$. We claim  
	\begin{equation}
		\label{eq:phi-invariance}
		\varphi(\psi(r,x)) = \psi(r,\varphi(x)) \qquad \forall (r,x) \in \intoc[0]{-\delta,0} 	\times \Sigma,
	\end{equation}
	\noindent that is, $\varphi$ and $\psi$ commute. Note that \eqref{eq:phi-invariance} is well-defined since $\varphi(\Sigma) = \Sigma$ by assumption. Indeed, \eqref{eq:phi-invariance} follows from the uniqueness of integral curves and the computation
	\begin{align*}
		\frac{d}{dr}\varphi(\psi(r,x)) &= \frac{d}{dr}\varphi(\phi_r^X(x))\\
		&= D\varphi(X(\phi_r^X(x)))\\
		&= (D\varphi \circ X\vert_{U_\delta} \circ \varphi^{-1} \circ \varphi)(\phi_r^X(x))\\
		&= (\varphi_*X\vert_{\varphi(U_\delta)} \circ \varphi)(\phi_r^X(x))\\
		&= (X\vert_{\varphi(U_\delta)} \circ \varphi)(\phi_r^X(x))\\
		&= X(\varphi(\psi(r,x)))
	\end{align*}
	\noindent where we used the $\varphi$-invariance of the Liouville vector field on $U_\delta$, that is
	\begin{equation*}
		 \varphi_*X\vert_{\varphi(U_\delta)} = X\vert_{\varphi(U_\delta)},
	\end{equation*}
	\noindent which in turn follows from
	\allowdisplaybreaks
	\begin{align*}
		i_{\varphi_*X}d\lambda &= d\lambda(\varphi_*X,\cdot)\\
		&= d\lambda(D\varphi \circ X \circ \varphi^{-1},\cdot)\\
		&= d\lambda\del[1]{D\varphi \circ X \circ \varphi^{-1}, D\varphi \circ D\varphi^{-1}\cdot}\\
		&= \varphi^*d\lambda(X \circ \varphi^{-1}, D\varphi^{-1} \cdot)\\
		&= d\varphi^*\lambda(X \circ \varphi^{-1}, D\varphi^{-1} \cdot)\\
		&= d\lambda(X \circ \varphi^{-1}, D\varphi^{-1} \cdot)\\
		&= \varphi_*(i_Xd\lambda)\\
		&= \varphi_*\lambda\\
		&= \lambda - d(f_\varphi \circ \varphi^{-1})
	\end{align*}
	\noindent and assumption \eqref{eq:assumption_delta}.

	Next we construct the defining Hamiltonian $H \in C^\infty(M)$. Let $h \in C^\infty(\mathbb{R})$ be a sufficiently small mollification of the piecewise linear function
	\begin{equation*}
		h(r) := \begin{cases}
			r & r \in \intcc[1]{-\frac{\delta}{2},\frac{\delta}{2}},\\
			\frac{\delta}{2} & r \in \intco[1]{\frac{\delta}{2},+\infty},\\
			-\frac{\delta}{2} & r \in \intoc[1]{-\infty,-\frac{\delta}{2}},
		\end{cases}
	\end{equation*}
	\noindent as in Figure \ref{fig:h}. 

\begin{figure}[h!tb]
	\centering
	\includegraphics[width=.41\textwidth]{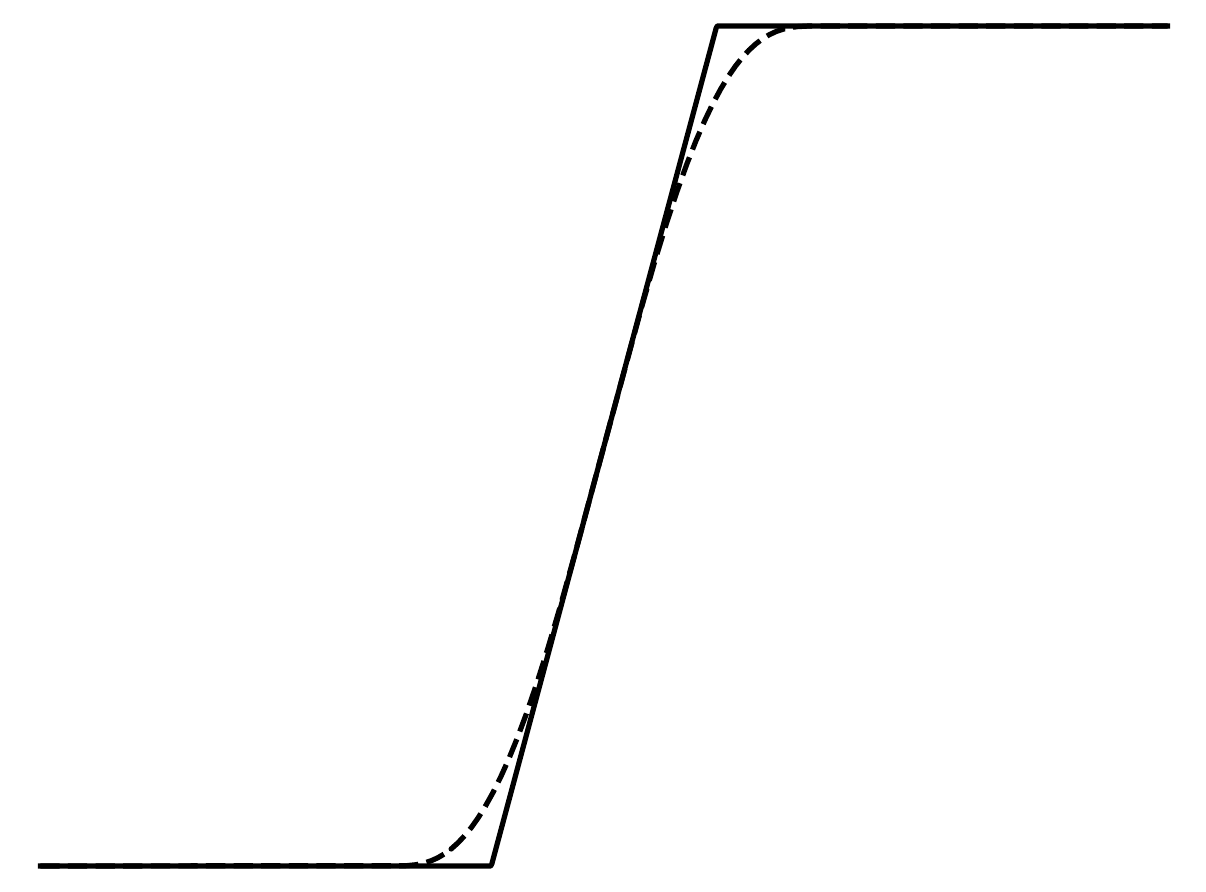}
	\caption{Mollification of the piecewise linear function $h$.}
	\label{fig:h}
\end{figure}

	Define $H \in C^\infty(M)$ by
	\begin{equation}
		\label{eq:modified_Hamiltonian}
		H(p) := \begin{cases}
			h(r) & p = \psi(r,x) \in U_\delta,\\
			h(r) & p = (r,x) \in \intco[0]{0,+\infty} \times \Sigma,\\
			-\frac{\delta}{2} & p \in W \setminus U_\delta.
		\end{cases}
	\end{equation}
	Then $H$ is a defining Hamiltonian for $\Sigma$ and $dH$ is compactly supported by construction. Moreover, $H$ is $\varphi$-invariant by \eqref{eq:phi-invariance}. Finally, $X_H\vert_\Sigma = R$ follows from the observation $X_H = h'(r)e^{-r}R$. Indeed, on $U_\delta$ we compute
	\begin{align*}
		i_{h'(t)e^{-r}R}\psi^*d\lambda &= i_{h'(r)e^{-r}R} d(e^r\lambda\vert_\Sigma)\\
		&= i_{h'(r)e^{-r}R}(e^rdr \wedge \lambda\vert_\Sigma + e^r d\lambda\vert_\Sigma)\\
		&= -h'(r)dr\\
		&= -dH.
	\end{align*}
	
	Next we construct the adapted almost complex structure $J$. Fix a $\varphi$-invariant compatible almost complex structure $J^\Sigma$ on the contact distribution $\ker \lambda\vert_\Sigma$. Extend this family to $(\intoo[0]{-\delta,+\infty} \times \Sigma, d(e^r\lambda\vert_\Sigma))$ by setting
\begin{equation}
	\label{eq:SFT-like}
	J^\Sigma\vert_{(a,x)}(b,v) := \del[1]{\lambda_x(v),  J^\Sigma\vert_x(\pi(v)) - bR(x)},
\end{equation}
\noindent where $\pi \colon \ker \lambda\vert_\Sigma \oplus \langle R \rangle \to \ker \lambda\vert_\Sigma$ denotes the projection. Choose a $\varphi$-invariant $d\lambda$-compatible almost complex structure $J^{W \setminus \Sigma}$ on $W \setminus \Sigma$ and let $\cbr[1]{\beta^\Sigma,\beta^{W \setminus \Sigma}}$ be a partition of unity subordinate to the open cover $\cbr[0]{U_\delta,W \setminus \Sigma}$ of $W$. The compatible almost complex structure $J$ associated with the Riemannian metric
\begin{equation*}
	m_J(\cdot,\cdot) := \beta^\Sigma d\lambda(J^\Sigma\cdot,\cdot) + \beta^{W \setminus \Sigma}d\lambda(J^{W \setminus \Sigma}\cdot,\cdot)
\end{equation*}
\noindent on $W$ is adapted.
\end{proof}

\begin{definition}[\bld{$L^2$}-Metric]
	Let $(H,J)$ be Rabinowitz--Floer data for a Liouville automorphism $\varphi \in \Aut(W,\lambda)$. Define an \bld{$L^2$-metric} on $\mathscr{L}_\varphi M \times \mathbb{R}$
	\begin{equation}
        \label{eq:L^2-metric}
        \langle (X,\eta),(Y,\sigma)\rangle_J := \int_0^1 d\lambda(JX(t),Y(t))dt + \eta\sigma
	\end{equation}
	\noindent for all $(X,\eta),(Y,\sigma) \in T_\gamma\mathscr{L}_\varphi M \times \mathbb{R}$ and $\gamma \in \mathscr{L}_\varphi M$.
\end{definition}

With respect to the $L^2$-metric \eqref{eq:L^2-metric}, the gradient of the twisted Rabinowitz action functional $\grad_J\mathscr{A}^H_\varphi \in \mathfrak{X}(\mathscr{L}_\varphi M \times \mathbb{R})$ is given by
\begin{equation*}
	\grad_J \mathscr{A}^H_\varphi\vert_{(\gamma,\tau)}(t) = \begin{pmatrix}
		J(\dot{\gamma}(t) - \tau X_H(\gamma(t)))\\
		\displaystyle-\int_0^1 H \circ \gamma
	\end{pmatrix} \qquad \forall t \in \mathbb{R}.
\end{equation*}

\begin{lemma}[Fundamental Lemma]
	\label{lem:fundamental_lemma}
	Let $(H,J)$ be Rabinowitz--Floer data for a Liouville automorphism $\varphi \in \Aut(W,\lambda)$ of a Liouville domain $(W,\lambda)$. Then there exists a constant $C = C(\lambda,H,J,f_\varphi) > 0$ such that
	\begin{equation*}
		\norm[0]{\grad_J \mathscr{A}^H_\varphi\vert_{(\gamma,\tau)}}_J < \frac{1}{C} \quad \Rightarrow \quad \abs[0]{\tau} \leq C(\abs[0]{\mathscr{A}^H_\varphi(\gamma,\tau)} + 1)
	\end{equation*}
	\noindent for all $(\gamma,\tau) \in \mathscr{L}_\varphi M \times \mathbb{R}$.	
\end{lemma}

\begin{proof}
	We proceed in three steps.
	
	\emph{Step 1: There exist constants $\delta > 0$ and $0 < C_\delta < +\infty$ such that if $(\gamma,\tau) \in \mathscr{L}_\varphi M$ with $\gamma(I) \subseteq H^{-1}\del[1]{\intoo[0]{-\delta,\delta}} =: U_\delta$, then}
	\begin{equation*}
		\abs[0]{\tau} \leq 2\abs[0]{\mathscr{A}^H_\varphi(\gamma,\tau)} + C_\delta \norm[0]{\grad_J \mathscr{A}^H_\varphi\vert_{(\gamma,\tau)}}_J + 2\norm[0]{f_\varphi}_\infty.
	\end{equation*}
	Choose $\delta > 0$ such that $U_\delta \subseteq \supp X_H$ and 
	\begin{equation*}
		\lambda_x(X_H(x)) \geq \frac{1}{2} + \delta \qquad \forall x \in U_\delta.
	\end{equation*}
	This is possible as $X_H \vert_\Sigma = R$. Moreover, set	
	\begin{equation*}
		C_\delta := 2\norm[0]{\lambda\vert_{U_\delta}}_\infty.
	\end{equation*}
	Then $C_\delta < +\infty$ as $dH$ is compactly supported. We estimate
	\begin{align*}
		\abs[0]{\mathscr{A}^H_\varphi(\gamma,\tau)} &= \abs[3]{\int_0^1 \gamma^*\lambda - \tau \int_0^1 H(\gamma) - f_\varphi(\gamma(0))}\\
		&= \abs[3]{\tau\int_0^1\lambda(X_H(\gamma)) + \int_0^1 \lambda(\dot{\gamma} - \tau X_H(\gamma)) - \tau \int_0^1 H(\gamma) - f_\varphi(\gamma(0))}\\
		&\geq \abs[0]{\tau}\del[3]{\frac{1}{2} + \delta} - \abs[3]{\int_0^1 \lambda(\dot{\gamma} - \tau X_H(\gamma))} - \abs[0]{\tau}\delta - \norm[0]{f_\varphi}_\infty\\
		&\geq \frac{\abs[0]{\tau}}{2} - \frac{C_\delta}{2}\int_0^1 \norm[0]{\dot{\gamma}(t) - \tau X_H(\gamma(t))}_J dt - \norm[0]{f_\varphi}_\infty\\
		&\geq \frac{\abs[0]{\tau}}{2} - \frac{C_\delta}{2}\sqrt{\int_0^1 \norm[0]{\dot{\gamma}(t) - \tau X_H(\gamma(t))}^2_J dt} - \norm[0]{f_\varphi}_\infty\\
		&\geq \frac{\abs[0]{\tau}}{2} - \frac{C_\delta}{2} \norm[0]{\grad_J\mathscr{A}^H_\varphi\vert_{(\gamma,\tau)}}_J - \norm[0]{f_\varphi}_\infty
	\end{align*}
	\noindent by Jensen's inequality.

	\emph{Step 2: For all $\delta > 0$, there exists $\varepsilon = \varepsilon(\delta) > 0$ with the property that if there exists $t_0 \in I$ with $\abs[0]{H(\gamma(t_0))} \geq \delta$ for $(\gamma,\tau) \in \mathscr{L}_\varphi M$, then $\norm[0]{\grad_J\mathscr{A}^H_\varphi\vert_{(\gamma,\tau)}}_J \geq \varepsilon$.}
	Assume first that $\gamma(t) \in M \setminus U_{\frac{\delta}{2}}$ for all $t \in I$. In this case we estimate
	\begin{equation*}
		\norm[0]{\grad_J\mathscr{A}^H_\varphi\vert_{(\gamma,\tau)}}_J \geq \abs[3]{\int_0^1 H(\gamma)} \geq \frac{\delta}{2}.
	\end{equation*}
	Otherwise, we may assume without loss of generality that there exists $t_1 \in I$ such that $\abs[0]{H(\gamma(t_1))} \leq \frac{\delta}{2}$, $t_0 < t_1$ and $\frac{\delta}{2} \leq \abs[0]{H(\gamma(t))} \leq \delta$ for all $t \in \intcc[0]{t_0,t_1}$. Set
	\begin{equation*}
		\kappa := \max_{x \in \overline{U}_\delta} \norm[0]{\grad_J H}_J > 0,
	\end{equation*}
	\noindent as $dH \neq 0$ in a neighbourhood of $\Sigma$. We estimate
	\allowdisplaybreaks
	\begin{align*}
		\norm[0]{\grad_J\mathscr{A}^H_\varphi\vert_{(\gamma,\tau)}}_J &\geq \sqrt{\int_0^1\norm[0]{\dot{\gamma}(t) - \tau X_H(\gamma(t))}_J^2 dt}\\
		&\geq \int_0^1\norm[0]{\dot{\gamma}(t) - \tau X_H(\gamma(t))}_J dt\\
		&\geq \int_{t_0}^{t_1} \norm[0]{\dot{\gamma}(t) - \tau X_H(\gamma(t))}_J dt\\
		&\geq \frac{1}{\kappa} \int_{t_0}^{t_1} \norm[0]{\grad_J H(\gamma(t))}_J \norm[0]{\dot{\gamma}(t) - \tau X_H(\gamma(t))}_J dt\\
		&\geq \frac{1}{\kappa} \int_{t_0}^{t_1} \abs[0]{m_J( \grad_J H(\gamma(t)), \dot{\gamma}(t) - \tau X_H(\gamma(t)))}dt\\
		&= \frac{1}{\kappa} \int_{t_0}^{t_1} \abs[0]{dH(\dot{\gamma}(t) - \tau X_H(\gamma(t)))}dt\\
		&= \frac{1}{\kappa} \int_{t_0}^{t_1} \abs[0]{dH(\dot{\gamma}(t))}dt\\
		&= \frac{1}{\kappa} \int_{t_0}^{t_1} \abs[3]{\frac{d}{dt}(H \circ \gamma)}\\
		&\geq \frac{1}{\kappa} \abs[3]{\int_{t_0}^{t_1} \frac{d}{dt} (H \circ \gamma)}\\
		&= \frac{1}{\kappa}\abs[0]{H(\gamma(t_1)) - H(\gamma(t_0))}\\
		&\geq \frac{1}{\kappa}\del[1]{\abs[0]{H(\gamma(t_0))} - \abs[0]{H(\gamma(t_1))}}\\
		&\geq \frac{\delta}{2\kappa}
	\end{align*}
	\noindent by Cauchy--Schwarz. Hence $\norm[0]{\grad_J\mathscr{A}^H_\varphi\vert_{(\gamma,\tau)}}_J \geq \varepsilon$ for
	\begin{equation*}
		\varepsilon = \varepsilon(\delta) := \frac{\delta}{2\max\cbr[0]{1,\kappa}}.
	\end{equation*}

	\emph{Step 3: We prove the Fundamental Lemma.} Choose $\delta > 0$ and $0 < C_\delta < +\infty$ as in Step 1 and $\varepsilon = \varepsilon(\delta) > 0$ as in Step 2. Set
	\begin{equation*}
		C_0 := \max\cbr[1]{2,C_\delta\varepsilon + 2\norm[0]{f_\varphi}_\infty}.
	\end{equation*}
	Assume that $\norm[0]{\grad_J\mathscr{A}^H_\varphi\vert_{(\gamma,\tau)}}_J < \varepsilon$ for $(\gamma,\tau) \in \mathscr{L}_\varphi M$. Then $\gamma(I) \subseteq U_\delta$, as otherwise there exists $t_0 \in I$ with $\abs[0]{H(\gamma(t_0))} \geq \delta$ implying $\norm[0]{\grad_J\mathscr{A}^H_\varphi\vert_{(\gamma,\tau)}}_J \geq \varepsilon$ by Step 2. Thus with Step 1 we estimate
	\begin{equation}
		\label{eq:estimate}
		\abs[0]{\tau} \leq C_0(\abs[0]{\mathscr{A}^H_\varphi(\gamma,\tau)} + 1).
	\end{equation}
	Finally, set
	\begin{equation*}
		C := \max\cbr[3]{C_0,\frac{1}{\varepsilon}}.
	\end{equation*}
	This proves the Fundamental Lemma.
\end{proof}

\begin{definition}[Twisted Negative Gradient Flow Line]
	Let $(H,J)$ be Rabinowitz--Floer data for a Liouville automorphism $\varphi \in \Aut(W,\lambda)$. A \bld{twisted negative gradient flow line} is a tuple $(u,\tau) \in C^\infty(\mathbb{R}, \mathscr{L}_\varphi M \times \mathbb{R})$ such that
	\begin{equation*}
		\partial_s(u,\tau) = -\grad_J \mathscr{A}^H_\varphi\vert_{(u(s),\tau(s))} \qquad \forall s \in \mathbb{R}.
	\end{equation*}
\end{definition}

\begin{definition}[Energy]
	Let $(H,J)$ be Rabinowitz--Floer data for a Liouville automorphism $\varphi \in \Aut(W,\lambda)$. The \bld{energy of a twisted negative gradient flow line $(u,\tau)$} is defined by
	\begin{equation*}
		E_J(u,\tau) := \int_{-\infty}^{+\infty} \norm[0]{\partial_s(u,\tau)}^2_J ds = \int_{-\infty}^{+\infty} \norm[1]{\grad_J \mathscr{A}^H_\varphi\vert_{(u(s),\tau(s))}}^2_J ds.
	\end{equation*}
\end{definition}

\begin{theorem}[Compactness]
	\label{thm:compactness}
	Let $(H,J)$ be Rabinowitz--Floer data for a Liouville automorphism $\varphi \in \Aut(W,\lambda)$. Suppose $(u_k,\tau_k)$ is a sequence of negative gradient flow lines of the twisted Rabinowitz action functional $\mathscr{A}^H_\varphi$ such that there exist constants $a,b \in \mathbb{R}$ with
	\begin{equation*}
		a \leq \mathscr{A}^H_\varphi\del[1]{u_k(s),\tau_k(s)} \leq b \qquad \forall k \in \mathbb{N}, s \in \mathbb{R}.
	\end{equation*}
	For every reparametrisation sequence $(s_k) \subseteq \mathbb{R}$ there exists a subsequence $(s_{k_l})$ and a negative gradient flow line $(u_\infty,\tau_\infty)$ of $\mathscr{A}^H_\varphi$ such that 
	\begin{equation*}
		\del[1]{u_{k_l}(\cdot + s_{k_l}),\tau_{k_l}(\cdot + s_{k_l})} \xrightarrow{C^\infty_{\loc}} (u_\infty,\tau_\infty) \qquad \text{as } l \to \infty.
	\end{equation*}
\end{theorem}

\begin{proof}
	In order to show $C^\infty_{\loc}$-convergence, we need to establish
	\begin{itemize}
		\item a uniform $L^\infty$-bound on $u_k$,
		\item a uniform $L^\infty$-bound on $\tau_k$,
		\item a uniform $L^\infty$-bound on the derivatives of $u_k$.
	\end{itemize}
	Indeed, through elliptic bootstrapping \cite[Theorem~B.4.1]{mcduffsalamon:J-holomorphic_curves:2012} the negative gradient flow equation, we will obtain $C^\infty_{\loc}$-convergence by \cite[Theorem~B.4.2]{mcduffsalamon:J-holomorphic_curves:2012}. To obtain a uniform $L^\infty$-bound on the sequence of twisted negative gradient flow lines $u_k$, observe that by definition of Rabinowitz--Floer data for $\varphi$, there exists $r \in \intoo[0]{0,+\infty}$ such that
	\begin{equation*}
		\supp X_H \cap \intco[0]{r,+\infty} \times \Sigma = \emptyset
	\end{equation*}
	\noindent and $J$ is adapted to the boundary of $W \cup_\Sigma \intcc[0]{0,r} \times \Sigma$. Consequently, the Maximum Principle \cite[Corollary~9.2.11]{mcduffsalamon:J-holomorphic_curves:2012} implies that every $u_k$ remains inside the compact set $W \cup_\Sigma \intcc[0]{0,r} \times \Sigma$ as the asymptotics belong to $W \cup_\Sigma \intco[0]{0,r} \times \Sigma$ for all $k \in \mathbb{N}$. Indeed, this follows from
	\allowdisplaybreaks
	\begin{align*}
		E_J(u_k,\tau_k) &= \int_{-\infty}^{+\infty} \norm[0]{\partial_s(u_k,\tau_k)}^2_J ds\\
		&= \int_{-\infty}^{+\infty} \langle \partial_s (u_k,\tau_k),\partial_s (u_k,\tau_k)\rangle_Jds\\
		&= -\int_{-\infty}^{+\infty} \langle\grad_J\mathscr{A}^H_\varphi\vert_{(u_k(s),\tau_k(s))},\partial_s (u_k,\tau_k)\rangle_Jds\\
		&= -\int_{-\infty}^{+\infty} d\mathscr{A}^H_\varphi(\partial_s (u_k,\tau_k))ds\\
		&= -\int_{-\infty}^{+\infty} \partial_s \mathscr{A}^H_\varphi(u_k,\tau_k)ds\\
		&= \lim_{s \to -\infty}\mathscr{A}^H_\varphi(u_k(s),\tau_k(s)) - \lim_{s \to +\infty} \mathscr{A}^H_\varphi(u_k(s),\tau_k(s))\\
		&\leq b - a,
	\end{align*}
	\noindent as this implies
	\begin{equation*}
		\lim_{s \to \pm \infty} \norm[0]{\partial_s(u_k,\tau_k)}_J = \lim_{s \to \pm \infty}\norm[1]{\grad_J \mathscr{A}^H_\varphi\vert_{(u_k(s),\tau_k(s))}}_J = 0
	\end{equation*}
	\noindent by the negative gradient flow equation. 

	The uniform $L^\infty$-bound on the Lagrange multipliers $\tau_k$ follows from the Fundamental Lemma \ref{lem:fundamental_lemma}. Fix a twisted negative gradient flow line $(u,\tau)$ and let $C > 0$ as in the Fundamental Lemma \ref{lem:fundamental_lemma}. For every $\sigma \in \mathbb{R}$ we can define $\zeta(\sigma) \geq 0$ by
	\begin{equation*}
		\zeta(\sigma) := \inf\cbr[3]{\zeta \geq 0 : \norm[1]{\grad_J \mathscr{A}^H_\varphi\vert_{(u(\sigma + \zeta),\tau(\sigma + \zeta))}}_J < \frac{1}{C}}. 
	\end{equation*}
	We estimate
	\begin{equation*}
		b - a \geq E_J(u,\tau) \geq \int_\sigma^{\sigma + \zeta(\sigma)} \norm[1]{\grad_J \mathscr{A}^H_\varphi\vert_{(u_s,\tau(s))}}^2_J ds \geq \frac{\zeta(\sigma)}{C^2}.
	\end{equation*}
	By the Fundamental Lemma \ref{lem:fundamental_lemma} we have that
	\begin{equation*}
		\abs[0]{\tau(\sigma + \zeta(\sigma))} < C(\max\{\abs[0]{a},\abs[0]{b}\} + 1) \qquad \forall \sigma \in \mathbb{R},
	\end{equation*}
	\noindent and thus using the negative gradient flow equation again we estimate
	\begin{align*}
		\abs[0]{\tau(\sigma)} &\leq \abs[0]{\tau(\sigma + \zeta(\sigma))} + \int_\sigma^{\sigma + \zeta(\sigma)}\abs[0]{\partial_s \tau(s)} ds\\
		&= \abs[0]{\tau(\sigma + \zeta(\sigma))} + \int_\sigma^{\sigma + \zeta(\sigma)}\abs[3]{\int_0^1 H(u(s,t))dt} ds\\
		&\leq C(\max\{\abs[0]{a},\abs[0]{b}\} + 1) + \zeta(\sigma)\norm[0]{H}_\infty\\
		&\leq C(\max\{\abs[0]{a},\abs[0]{b}\} + 1) + C^2(b - a)\norm[0]{H}_\infty.
	\end{align*}
	\noindent for all $\sigma \in \mathbb{R}$. Hence
	\begin{equation*}
		\norm[0]{\tau}_\infty \leq C(\max\{\abs[0]{a},\abs[0]{b}\} + 1) + C^2(b - a)\norm[0]{H}_\infty
	\end{equation*}
	\noindent is independent of the twisted negative gradient flow line $(u,\tau)$.

	The uniform $L^\infty$-bound on the derivatives of $u_k$ follows from Corollary \ref{cor:bubbling_rfh} as an exact symplectic manifold is symplectically aspherical.
\end{proof}

\subsection{Definition of Twisted Rabinowitz--Floer Homology}
In this section we make implicit use of the requirement that a Liouville automorphism has finite order. This is crucial because then the arguments proceed as in the case of loops by Remark \ref{rem:finite_order}. 

\begin{definition}[Transverse Conley--Zehnder Index]
	\label{def:CZ-index}
	Let $(W^{2n},\lambda)$ be a Liouville domain with boundary $\Sigma$. Let $(\gamma_0,\tau_0), (\gamma_1,\tau_1) \in \mathscr{P}_\varphi(\Sigma,\lambda\vert_\Sigma)$ for some Liouville automorphism $\varphi \in \Aut(W,\lambda)$ such that there exists a path $\gamma_\sigma$ in $\mathscr{L}_\varphi \Sigma$ from $\gamma_0$ to $\gamma_1$. Define the \bld{transverse Conley--Zehnder index} by
	\begin{equation*}
		\mu((\gamma_0,\tau_0),(\gamma_1,\tau_1)) := \mu_{\CZ}(\Psi^1) - \mu_{\CZ}(\Psi^0) \in \mathbb{Z},
	\end{equation*}
	\noindent with
	\begin{align*}
		\Psi^0 \colon I \to \Sp(n - 1), \qquad &\Psi^0_t := \Phi_{t,0}^{-1} \circ D^\xi \phi^R_{\tau_0t} \circ \Phi_{0,0},\\
		\Psi^1 \colon I \to \Sp(n - 1), \qquad &\Psi^1_t := \Phi_{t,1}^{-1} \circ D^\xi \phi^R_{\tau_1t} \circ \Phi_{0,1},
	\end{align*}
	\noindent where $\Phi_{t,\sigma} \colon \mathbb{R}^{2n - 2} \to \xi_{\gamma_\sigma(t)}$ is a symplectic trivialisation of $F^*\xi$, $\xi := \ker \lambda\vert_\Sigma$ with $F \in C^\infty(\mathbb{R} \times I,M)$ being defined by $F(t,\sigma) := \gamma_\sigma(t)$, satisfying
	\begin{equation}
		\label{eq:twist_condition_trivialisation}
		\Phi_{t + 1,\sigma} = D\varphi \circ \Phi_{t,\sigma} \qquad \forall (t,\sigma) \in \mathbb{R} \times I.
	\end{equation}
\end{definition}

\begin{lemma}
	In the setup of Definition \ref{def:CZ-index}, the transverse Conley--Zehnder index is well-defined, that is, independent of the choice of symplectic trivialisation.	
\end{lemma}

\begin{proof}
	First we need to show that one can always construct a symplectic trivialisation
	\begin{equation*}
		\Phi_{t,\sigma} \colon \mathbb{R}^{2n - 2} \to \xi_{F(t,\sigma)} \qquad \forall (t,\sigma) \in \mathbb{R} \times I,
	\end{equation*}
	\noindent of $F^*\xi$ satisfying the twist condition \eqref{eq:twist_condition_trivialisation}. By \cite[Theorem~2.1.3]{mcduffsalamon:st:2017}, there exists a linear symplectomorphism $\Phi_{0,0} \colon \mathbb{R}^{2n - 2} \to \xi_{F(0,0)}$. By \cite[Lemma~2.6.6]{mcduffsalamon:st:2017}, we get a symplectic trivialisation $\Phi_{t,0} \colon \mathbb{R}^{2n - 2} \to \xi_{F(t,0)}$ for all $t \in I$ with $\Phi_{1,0} = D\varphi \circ \Phi_{0,0}$. Extend this trivialisation to $\mathbb{R}$ by setting
	\begin{equation*}
		\Phi_{t + k,0} := D\varphi^k \circ \Phi_{t,0} \qquad \forall k \in \mathbb{Z}.
	\end{equation*}
	Next, trivialise along each ray $\sigma \mapsto F(t,\sigma)$ for fixed $t \in \mathbb{R}$. Hence we get a symplectic trivialisation $\Phi_{t,\sigma} \colon \mathbb{R}^{2n - 2} \to \xi_{F(t,\sigma)}$ of $F^*\xi$ satisfying \eqref{eq:twist_condition_trivialisation}.

Now we show that the transverse Conley--Zehnder index $\mu$ is independent of the choice of trivialisation.
Suppose that $\widetilde{\Phi}_{t,\sigma} \colon \mathbb{R}^{2n - 2} \to \xi_{F(t,\sigma)}$ is another symplectic trivialisation of $F^*\xi$ satisfying
\begin{equation*}
	\widetilde{\Phi}_{t + 1,\sigma} = D\varphi \circ \widetilde{\Phi}_{t,\sigma} \qquad \forall (t,\sigma) \in \mathbb{R} \times I.
\end{equation*}
Then we have
\begin{equation*}
	\widetilde{\Psi}_{t,\sigma} = \widetilde{\Phi}_{t,\sigma}^{-1} \circ \Phi_{t,\sigma} \circ \Psi_{t,\sigma} \circ \Phi_{0,\sigma}^{-1} \circ \widetilde{\Phi}_{0,\sigma} \qquad \forall (t,\sigma) \in \mathbb{R} \times \partial I, 
\end{equation*}
\noindent where 
\begin{equation*}
	\Psi_{t,\sigma} = \Phi_{t,\sigma}^{-1} \circ D^\xi \theta^{\tau_\sigma R}_t \circ \Phi_{0,\sigma}.
\end{equation*}
Define
\begin{equation*}
	\eta \colon \mathbb{T} \times I \to \Sp(n - 1), \qquad \eta(t,\sigma) := \Phi_{0,\sigma}^{-1} \circ \widetilde{\Phi}_{0,\sigma} \circ \widetilde{\Phi}^{-1}_{t,\sigma} \circ \Phi_{t,\sigma}.
\end{equation*}
Indeed, we compute
\begin{align*}
	\eta(t + 1,\sigma) &= \Phi_{0,\sigma}^{-1} \circ \widetilde{\Phi}_{0,\sigma} \circ \widetilde{\Phi}^{-1}_{t + 1,\sigma} \circ \Phi_{t + 1,\sigma}\\
	&= \Phi_{0,\sigma}^{-1} \circ \widetilde{\Phi}_{0,\sigma} \circ \widetilde{\Phi}^{-1}_{t,\sigma} \circ D\varphi^{-1} \circ D\varphi \circ \Phi_{t,\sigma}\\
	&= \eta(t,\sigma),
\end{align*}
\noindent for all $(t,\sigma) \in \mathbb{R} \times I$. Using the naturality as well as the loop property \cite[p.~20]{salamon:fh:1999} of the Conley--Zehnder index we compute
\begin{equation*}
	\mu_{\CZ}(\widetilde{\Psi}^s) = \mu_{\CZ}(\eta_s \cdot \Psi^s) = \mu_{\CZ}(\Psi^s) + 2\mu_{\M}(\eta_s) \qquad \forall s \in \partial I, 
\end{equation*}
\noindent where $\mu_{\M}$ denotes the Maslov index. In particular,
\begin{align*}
	\mu_{\CZ}(\widetilde{\Psi}^1) - \mu_{\CZ}(\widetilde{\Psi}^0) &= \mu_{\CZ}(\Psi^1) - \mu_{\CZ}(\Psi^0) + 2\del[1]{\mu_{\M}(\eta_1) - \mu_{\M}(\eta_0)}\\
	&= \mu_{\CZ}(\Psi^1) - \mu_{\CZ}(\Psi^0)
\end{align*}
\noindent by the invariance property of the Maslov index \cite[p.~195]{frauenfelderkoert:3bp:2018}.
\end{proof}

\begin{remark}
	\label{rem:grading}
Denote by 
\begin{equation*}
	\Sigma_\varphi := \frac{\Sigma \times \mathbb{R}}{(\varphi(x),t + 1){\sim}(x,t)}
\end{equation*}
\noindent the mapping torus of $\varphi$ giving rise to the fibration
\begin{equation*}
	\pi_\varphi \colon \Sigma_\varphi \to \mathbb{T}, \qquad \pi_\varphi([x,t]) := [t].
\end{equation*}
The vertical bundle $\ker D^\xi \pi_\varphi \to \Sigma_\varphi$ is a symplectic vector bundle. If $\widetilde{F}$ is another homotopy in $\mathscr{L}_\varphi \Sigma$ from $\gamma_0$ to $\gamma_1$, the concatenation with the reversed path $F^-$ can be identified with the map 
\begin{equation*}
	\widetilde{F} \# F^- \colon \mathbb{T}^2 \to \Sigma_\varphi, \qquad (t,\sigma) \mapsto [(\widetilde{F}\# F^-(t,\sigma),t)].
\end{equation*}
Hence using the concatenation property of the Conley--Zehnder index \cite[p.~195]{frauenfelderkoert:3bp:2018} as well as the functoriality of the Chern number \cite[p.~85]{mcduffsalamon:st:2017}, for $s \in \partial I$ we compute
\begin{align*}
	\mu_{\CZ}(\widetilde{\Psi}^s) - \mu_{\CZ}(\Psi^s) &= \mu_{\CZ}\del[1]{\widetilde{\Psi}^s \#(\Psi^s)^-}\\
	&= 2\mu_{\M}\del[1]{\widetilde{\Psi}^s \#(\Psi^s)^-}\\
	&= 2c_1\del[1]{(\widetilde{F}\#F^-)^*\ker D^\xi \pi_\varphi}\\
	&= 2c_1(\ker D^\xi\pi_\varphi). 
\end{align*}
Thus if the transverse Conley--Zehnder index is viewed in $\mathbb{Z}_2$ or $c_1(\ker D^\xi \pi_\varphi) = 0$, then it additionally does not depend on the choice of path in $\mathscr{L}_\varphi \Sigma$.
\end{remark}
 
Let $(H,J)$ be Rabinowitz--Floer data for $\varphi \in \Aut(W,\lambda)$. Set
\begin{equation*}
	\Sigma := \partial W \qquad \text{and} \qquad M := W \cup_\Sigma \intco[0]{0, +\infty} \times \Sigma.
\end{equation*}
Fix $(\eta,\tau_\eta) \in \mathscr{P}_\varphi(\Sigma,\lambda\vert_\Sigma)$ and denote by $[\eta]$ the corresponding class in $\pi_0 \mathscr{L}_\varphi \Sigma$. Assume that the twisted Rabinowitz action functional $\mathscr{A}^H_\varphi$ is Morse--Bott, that is, $\Crit \mathscr{A}^H_\varphi \subseteq \Sigma \times \mathbb{R}$ is a properly embedded submanifold by Proposition \ref{prop:kernel_hessian_contact}, and fix a Morse function $h \in C^\infty(\Crit \mathscr{A}^H_\varphi)$. Define the twisted Rabinowitz--Floer chain group $\RFC^\varphi(\Sigma,M)$ to be the $\mathbb{Z}_2$-vector space consisting of all formal linear combinations
\begin{equation*}
	\zeta = \sum_{\substack{(\gamma,\tau) \in \Crit(h)\\{[\gamma] = [\eta]}}} \zeta_{(\gamma,\tau)} (\gamma,\tau)
\end{equation*}
\noindent satisfying the Novikov finiteness condition
\begin{equation*}
	\#\cbr[0]{(\gamma,\tau) \in \Crit(h) : \zeta_{(\gamma,\tau)} \neq 0, \mathscr{A}^H_\varphi(\gamma,\tau) \geq \kappa} < \infty \qquad \forall \kappa \in \mathbb{R}.
\end{equation*}
Define a boundary operator
\begin{equation*}
	\partial \colon \RFC^\varphi(\Sigma,M) \to \RFC^\varphi(\Sigma,M) 
\end{equation*}
\noindent by
\begin{equation*}
	\partial (\gamma^-,\tau^-) := \sum_{\substack{(\gamma^+,\tau^+) \in \Crit(h)\\{[\gamma^+] = [\gamma^-]}}} n_\varphi(\gamma^\pm,\tau^\pm)(\gamma^+,\tau^+),
\end{equation*}
\noindent where 
\begin{equation*}
	n_\varphi(\gamma^\pm,\tau^\pm) := \#_2\mathscr{M}^0_\varphi(\gamma^\pm,\tau^\pm) \in \mathbb{Z}_2,
\end{equation*}
\noindent with $\mathscr{M}^0_\varphi(\gamma^\pm,\tau^\pm)$ denoting the zero-dimensional component of the moduli space of all unparametrised twisted negative gradient flow lines with cascades from $(\gamma^-,\tau^-)$ to $(\gamma^+,\tau^+)$. This is well-defined by Theorem \ref{thm:compactness}. Define the \bld{twisted Rabinowitz--Floer homology of $\Sigma$ and $\varphi$} by
\begin{equation*}
	\RFH^\varphi(\Sigma,M) := \frac{\ker \partial}{\im \partial}.
\end{equation*}
 
\begin{proposition}
	\label{prop:fixed_points}
	Let $(W,\lambda)$ be a Liouville domain with simply connected boundary $\Sigma$ and $\varphi \in \Aut(W,\lambda)$. If there do not exist any nonconstant twisted periodic Reeb orbits on $\Sigma$, then
	\begin{equation*}
		\RFH^\varphi_\ast(\Sigma,M) \cong \operatorname{H}_\ast(\Fix(\varphi\vert_\Sigma);\mathbb{Z}_2).
	\end{equation*}	
\end{proposition}

\begin{proof}
	If there do not exist any nonconstant twisted periodic Reeb orbits,
	\begin{equation*}
		\Crit \mathscr{A}^H_\varphi = \cbr[0]{(c_x,0) : x \in \Fix(\varphi\vert_\Sigma)} \cong \Fix(\varphi\vert_\Sigma)
	\end{equation*}
	\noindent for any $H \in \mathscr{F}_\varphi(\Sigma)$. Since $\Fix(\varphi\vert_\Sigma)$ is a properly embedded submanifold of $\Sigma$ by \cite[Problem~8-32]{lee:dg:2018} or \cite[Lemma~5.5.7]{mcduffsalamon:st:2017}, $\mathscr{A}^H_\varphi$ is a Morse--Bott function. Let $x,y \in \Fix(\varphi\vert_\Sigma)$. As $\Sigma$ is simply connected by assumption, there exists some path $\gamma$ from $x$ to $y$ in $\Sigma$ and a homotopy from $\gamma$ to $\varphi \circ \gamma$ with fixed endpoints. Extend this homotopy to a path in $\mathscr{L}_\varphi \Sigma$ from $c_x$ to $c_y$. Choose a Morse function $h$ on $\Fix(\varphi\vert_\Sigma)$ and any critical point $c_x \in \Fix(\varphi\vert_\Sigma)$. Then we can define a $\mathbb{Z}$-grading of $\RFC^\varphi(\Sigma,M)$ by 
	\begin{equation*}
		\mu((c_y,0),(c_x,0)) + \ind_h(c_y) = \ind_h(c_y) \qquad \forall c_y \in \Crit(h),
	\end{equation*}
	\noindent and consequently,
	\begin{equation*}
		\RFH^\varphi_\ast(\Sigma,M) = \HM_\ast(\Fix(\varphi\vert_\Sigma);\mathbb{Z}_2) \cong \operatorname{H}_\ast(\Fix(\varphi\vert_\Sigma);\mathbb{Z}_2)
	\end{equation*}
	\noindent as there are only twisted negative gradient flow lines with zero cascades, that is, ordinary Morse gradient flow lines of $h$. Indeed, suppose that there is a nonconstant twisted negative gradient flow line $(u,\tau)$ of $\mathscr{A}^H_\varphi$ with asymptotics $(\gamma^\pm,\tau^\pm)$. Using the twisted negative gradient flow equation we estimate
	\allowdisplaybreaks
	\begin{equation*}
		\tau^- - \tau^+  = \int_{-\infty}^{+\infty} \norm[1]{\grad_J \mathscr{A}^H_\varphi\vert_{(u(s),\tau(s))}}^2_J ds > 0.
	\end{equation*}
	Hence $\tau^+ < \tau^-$, contradicting $\tau^\pm = 0$. 
\end{proof}

\subsection[Invariance of Twisted Rabinowitz--Floer Homology Under Twisted Homotopies]{Invariance of Twisted Rabinowitz--Floer Homology Under Twisted Homotopies of Liouville Domains}

\begin{definition}[Twisted Homotopy of Liouville Domains]
	Given the completion $(M,\lambda)$ of a Liouville domain $(W_0,\lambda)$ and $\varphi \in \Aut(W,\lambda)$, a \bld{twisted homotopy of Liouville domains in $M$} is a time-dependent Hamiltonian function $H \in C^\infty(M \times I)$ such that
	\begin{enumerate}[label=\textup{(\roman*)}]
		\item $W_\sigma := H^{-1}_\sigma(\intoc[0]{-\infty,0})$ is a Liouville domain with symplectic form $d\lambda\vert_{W_\sigma}$ and boundary $\Sigma_\sigma := H^{-1}_\sigma(0)$ for all $\sigma \in I$,
		\item $H_\sigma \in \mathscr{F}_\varphi(\Sigma_\sigma)$ for all $\sigma \in I$,
		\item $\Sigma_\sigma \cap \supp f_\varphi = \emptyset$ for all $\sigma \in I$.
	\end{enumerate}
	We write $(H_\sigma)_{\sigma \in I}$ for a twisted homotopy of Liouville domains.
\end{definition}

\begin{theorem}[Invariance of Twisted Rabinowitz--Floer Homology]
	\label{thm:invariance}
	If $(H_\sigma)_{\sigma \in I}$ is a twisted homotopy of Liouville domains such that both $\mathscr{A}^{H_0}_\varphi$ and $\mathscr{A}^{H_1}_\varphi$ are Morse--Bott, then there is a canonical isomorphism
	\begin{equation*}
		\RFH^\varphi(\Sigma_0,M) \cong \RFH^\varphi(\Sigma_1,M).
	\end{equation*}	
\end{theorem}

\begin{proof}
	The proof follows from the same adiabatic argument as in \cite[p.~275--277]{cieliebakfrauenfelder:rfh:2009}. Crucial is that \cite[Theorem~3.6]{cieliebakfrauenfelder:rfh:2009} remains true in our setting,  as well as the genericness of the Morse--Bott condition. Indeed, if $(M,\lambda)$ is an exact symplectic manifold and $\varphi \in \Diff(M)$ is of finite order such that $\varphi^*\lambda = \lambda$, then we have the following generalisation of \cite[Theorem~B.1]{cieliebakfrauenfelder:rfh:2009}. Adapting the proof accordingly, one can show that there exists a subset
	\begin{equation*}
		\mathscr{U} \subseteq \{H \in C^\infty_\varphi(M) : \supp dH \text{ compact}\},
	\end{equation*}
	\noindent of the second category such that for every $H \in \mathscr{U}$, $\mathscr{A}^H_\varphi$ is Morse--Bott with critical manifold being $\Fix(\varphi\vert_{H^{-1}(0)})$ together with a disjoint union of circles. Again, this works only since the contact condition is an open condition.	
\end{proof}

\begin{remark}
	Invariance of twisted Rabinowitz--Floer homology allows us to define twisted Rabinowitz--Floer homology also in the case where $\mathscr{A}^H_\varphi$ is not necessarily Morse--Bott. Indeed, as the proof of Theorem \ref{thm:invariance} shows, we can perturb the hypersurface $\Sigma$ slightly to make it Morse--Bott. Thus we can define the twisted Rabinowitz--Floer homology of such a hypersurface to be the twisted Rabinowitz--Floer homology of any Morse--Bott perturbation. By Theorem \ref{thm:invariance}, this is indeed well-defined.
\end{remark}

\begin{corollary}[Independence]
	Let $\varphi \in \Aut(W,\lambda)$ and $H_0, H_1 \in \mathscr{F}_\varphi(\Sigma)$ be such that either $\mathscr{A}^{H_0}_\varphi$ or $\mathscr{A}^{H_1}_\varphi$ is Morse--Bott. Then the definition of twisted Rabinowitz--Floer homology $\RFH^\varphi(\Sigma,M)$ is independent of the choice of twisted defining Hamiltonian function for $\Sigma$.
\end{corollary}

\begin{proof}
	Note that $\mathscr{F}_\varphi(\Sigma)$ is a convex space. Indeed, set 
	\begin{equation*}
		H_\sigma := (1 - \sigma)H_0 + \sigma H_1 \qquad \sigma \in I.
	\end{equation*}
	Then $\varphi^*H_\sigma = H_\sigma$, $dH_\sigma$ has compact support and $X_{H_\sigma}\vert_\Sigma = R$ for all $\sigma \in I$. Moreover, for the Liouville vector field $X \in \mathfrak{X}(M)$ we compute
	\begin{equation*}
		\frac{d}{dt}\bigg\vert_{t = 0} H \circ \phi^X_t\vert_\Sigma = dH(X)\vert_\Sigma = d\lambda(X,X_H)\vert_\Sigma = \lambda(X_H)\vert_\Sigma = \lambda(R) = 1, 
	\end{equation*}
	\noindent for any $H \in \mathscr{F}_\varphi(\Sigma)$, and thus $H < 0$ on $\Int W$ and $H > 0$ on $M \setminus W$. Consequently, $H^{-1}_\sigma(0) = \Sigma$ and so $H_\sigma \in \mathscr{F}_\varphi(\Sigma)$ for all $\sigma \in I$. Hence $(H_\sigma)_{\sigma \in I}$ is a twisted homotopy of Liouville domains in $M$ and Theorem \ref{thm:invariance} implies the claim.
\end{proof}

\subsection{Twisted Leaf-Wise Intersection Points}

\begin{definition}[Twisted Leaf-Wise Intersection Point]
	Let $(M,\lambda)$ be the completion of a Liouville domain $(W,\lambda)$ and let $\varphi \in \Aut(W,\lambda)$ be a Liouville automorphism. A point $x \in \Sigma$ is a \bld{twisted leaf-wise intersection point} for a Hamiltonian symplectomorphism $\varphi_F \in \Ham(M,d\lambda)$, if
	\begin{equation*}
		\varphi_F(x) \in L_{\varphi(x)} := \cbr[0]{\phi^R_t(\varphi(x)) : t \in \mathbb{R}}.
	\end{equation*}
\end{definition}

\begin{definition}[Twisted Moser Pair]
	Let $\varphi \in \Aut(W,\lambda)$. A \bld{twisted Moser pair} is defined to be a tuple $\mathfrak{M} := (\chi H,F)$, where 
	\begin{enumerate}[label=\textup{(\roman*)}]
		\item $H \in C^\infty_\varphi(M)$, $F \in C^\infty_\varphi(M \times \mathbb{R})$ and $\chi \in C^\infty(\mathbb{S}^1,I)$ such that $\int_0^1 \chi = 1$. Any time-dependent Hamiltonian function $\chi H$ is said to be \bld{weakly time-dependent}.
		\item $\supp \chi \subseteq \intoo[1]{0,\frac{1}{2}}$ and $F_t = 0$ for all $t \in \intcc[1]{0,\frac{1}{2}}$.
	\end{enumerate}
\end{definition}

\begin{lemma}	
	Let $\varphi \in \Aut(W,\lambda)$. For all $H \in \mathscr{F}_\varphi(\Sigma)$ and $\varphi_F \in \Ham(M,d\lambda)$ there exists a corresponding twisted Moser pair $\mathfrak{M}$ such that the flow of $\chi X_H$ is a time-reparametrisation of the flow of $X_H$.  
	\label{lem:twisted_Moser_pair}
\end{lemma}

\begin{proof}
	For constructing the Hamiltonian perturbation $\widetilde{F}$, fix $\rho \in C^\infty(I,I)$ such that 
	\begin{equation*}
		\rho(t) = \begin{cases}
			0 & t \in \intcc[1]{0,\frac{1}{2}},\\
			1 & t \in \intcc[1]{\frac{2}{3},1}.
		\end{cases}
	\end{equation*}
	See Figure \ref{fig:rho}. Then define $\widetilde{F} \in C^\infty_\varphi(M \times \mathbb{R})$ by 
	\begin{equation*}
		\widetilde{F}(x,t) := \dot{\rho}(t - k)F\del[1]{\varphi^{-k}(x),\rho(t - k)} \qquad \forall t \in \intcc[0]{k,k + 1},
	\end{equation*}
	\noindent for $k \in \mathbb{Z}$. See Figure \ref{fig:rho_dot}. Then $\widetilde{F}_t = 0$ for all $t \in \intcc[1]{0,\frac{1}{2}}$, and 
	\begin{equation*}
		\phi^{X_{\widetilde{F}}}_t = \phi^{X_F}_{\rho(t)} \qquad \forall t \in I.
	\end{equation*}
	Indeed, we compute
	\begin{equation*}
		\frac{d}{dt}\phi^{X_F}_{\rho(t)} = \dot{\rho}(t)\frac{d}{d\rho}\phi^{X_F}_{\rho(t)} = \dot{\rho}(t)\del[1]{X_{F_{\rho(t)}} \circ \phi^{X_F}_{\rho(t)}} = X_{\widetilde{F}_t} \circ \phi^{X_F}_{\rho(t)}.
	\end{equation*}
	In particular
	\begin{equation*}
		\varphi_{\widetilde{F}} = \phi^{X_{\widetilde{F}}}_1 = \phi^{X_F}_{\rho(1)} = \phi^{X_F}_1 = \varphi_F.
	\end{equation*}
	Finally, we have that
	\begin{equation*}
		\phi^{\chi X_H}_t = \phi^{X_H}_{\tau(t)} \qquad \text{with} \qquad \tau(t) := \int_0^t \chi,
	\end{equation*}
	\noindent as we compute
	\begin{equation*}
		\frac{d}{dt}\phi^{X_H}_ {\tau(t)} = \chi(t) \frac{d}{d\tau} \phi^{X_H}_{\tau(t)} = \chi(t)X_H \circ \phi^{X_H}_{\tau(t)},
	\end{equation*}
	\noindent and thus we conclude by the uniqueness of integral curves.
\end{proof}

\begin{figure}[h!tb]
    \centering
    \begin{subfigure}[b]{0.48\textwidth}
		\centering
        \includegraphics[width=\textwidth]{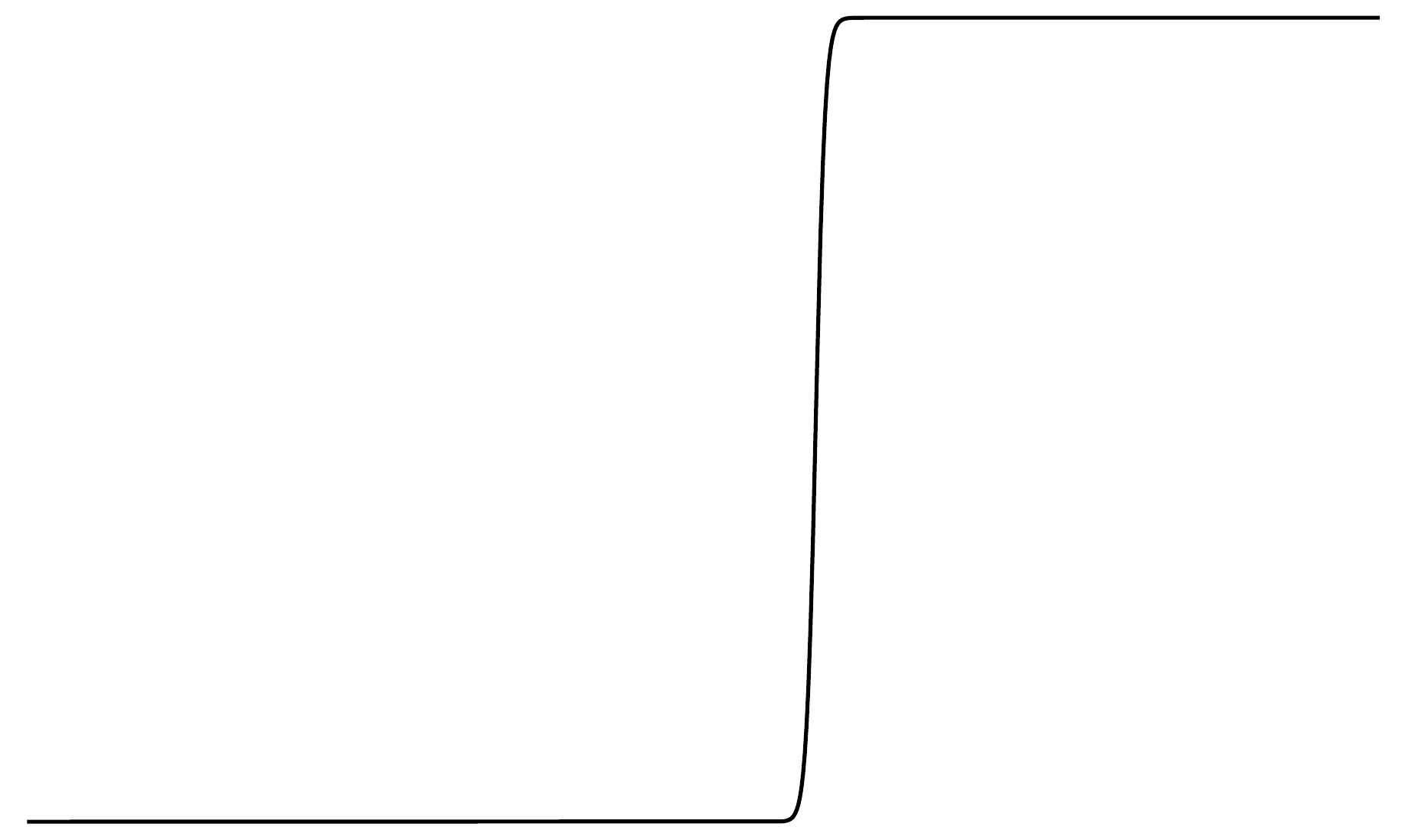}
        \caption{The smooth function $\rho$.}
        \label{fig:rho}
    \end{subfigure}
    \begin{subfigure}[b]{0.48\textwidth}
		\centering
        \includegraphics[width=\textwidth]{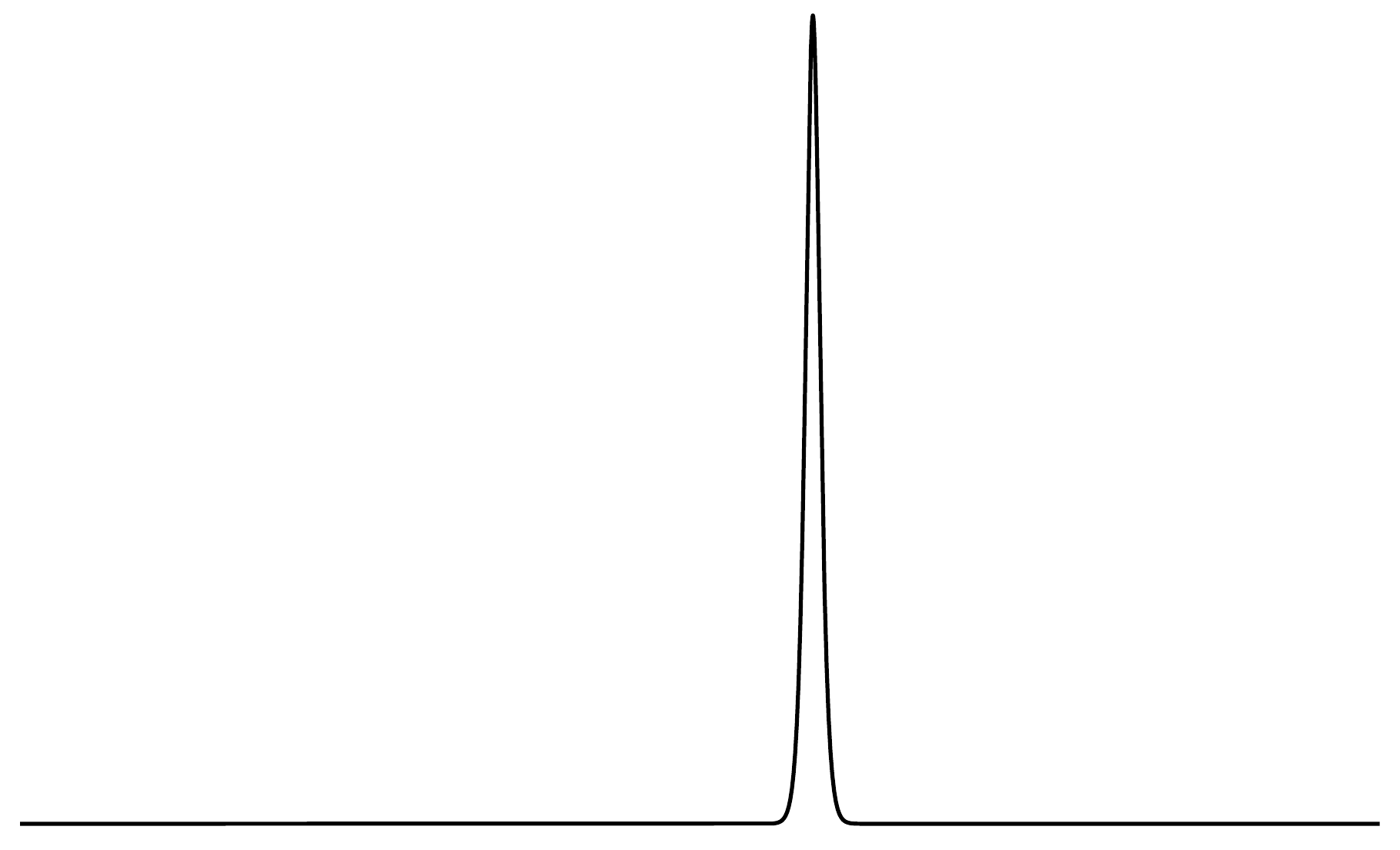}
		\caption{The derivative $\dot{\rho}$ of $\rho$.}
        \label{fig:rho_dot}
    \end{subfigure}
	\caption{}
\end{figure}


\begin{lemma}
	Let $\varphi \in \Aut(W,\lambda)$ and $\varphi_F \in \Ham(M,d\lambda)$ a Hamiltonian symplectomorphism. If $(\gamma,\tau) \in \Crit \mathscr{A}^{\mathfrak{M}}_\varphi$, then $x := \gamma\del[1]{\frac{1}{2}}$ is a twisted leaf-wise intersection point for $\varphi_F$.
	\label{lem:variational_characterisation_twisted_leaf-wise_intersections}
\end{lemma}

\begin{proof}
	Let $\mathfrak{M} = (\chi H,F)$ denote the twisted Moser pair from Lemma \ref{lem:twisted_Moser_pair}. Using Proposition \ref{prop:differential_perturbed_twisted_Rabinowitz_functional} we compute 
	\begin{align*}
		\partial_t H(\gamma(t)) &= dH(\dot{\gamma}(t))\\
		&= dH(\tau X_{\chi(t)H}(\gamma(t)) + X_{F_t}(\gamma(t)))\\
		&= dH(\tau \chi(t) X_H(\gamma(t)))\\
		&= \tau\chi(t)dH(X_H(\gamma(t)))\\
		&= 0
	\end{align*}
	\noindent for all $t \in \intcc[1]{0,\frac{1}{2}}$. Thus $H \circ \gamma = c \in \mathbb{R}$ on $\intcc[1]{0,\frac{1}{2}}$ with 
	\begin{equation*}
		0 = \int_0^1 \chi H(\gamma) = \int_0^{\frac{1}{2}}\chi H(\gamma) = c \int_0^{\frac{1}{2}}\chi = c\int_0^1\chi = c. 
	\end{equation*}
	Consequently, $\gamma(0) \in L_x$ and $x \in \Sigma$. Moreover, also $\gamma(1) = \varphi(\gamma(0)) \in \Sigma$ by the $\varphi$-invariance of $H$. For $t \in \intcc[1]{\frac{1}{2},1}$, $\dot{\gamma} = X_{F_t}(\gamma)$ and so $\gamma(1) = \varphi_F(x) \in \Sigma$. We conclude 
	\begin{equation*}
		L_{\varphi(x)} = \{\phi^R_t(\varphi(x)) : t \in \mathbb{R}\} = \{\varphi(\phi^R_t(x)) : t \in \mathbb{R}\} = \varphi(L_x),
	\end{equation*}
	\noindent and so $\varphi_F(x) = \gamma(1) = \varphi(\gamma(0)) \in L_{\varphi(x)}$.
\end{proof}

\begin{theorem}
	Let $(W,\lambda)$ be a Liouville domain with displaceable boundary in the completion $(M,\lambda)$ and $\varphi \in \Aut(W,\lambda)$. Then $\RFH^\varphi(\Sigma,M) \cong 0$.
	\label{thm:displaceable}
\end{theorem}

\begin{proof}
	Suppose that $\Sigma = \partial W$ is displaceable in $M$ via $\varphi_F \in \Ham_c(M,d\lambda)$ and choose Rabinowitz--Floer data $(H,J)$ for $\varphi$. Denote by $\mathfrak{M} = (\chi H,F)$ the associated twisted Moser pair from Lemma \ref{lem:twisted_Moser_pair}. Then $\Crit \mathscr{A}^{\mathfrak{M}}_\varphi = \emptyset$. Indeed, if there exists $(\gamma,\tau) \in \Crit \mathscr{A}^{\mathfrak{M}}_\varphi$, then $\gamma\del[1]{\frac{1}{2}}$ is a twisted leaf-wise intersection point for $\varphi_F$ by Lemma \ref{lem:variational_characterisation_twisted_leaf-wise_intersections}. However, this is impossible as by displaceability we have that $\varphi_F(\Sigma) \cap \Sigma = \emptyset$. Consequently, the perturbed twisted Rabinowitz action functional $\mathscr{A}^{\mathfrak{M}}_\varphi$ is a Morse function. By adapting the Fundamental Lemma to the current setting as in \cite[Theorem~2.9]{albersfrauenfelder:rfh:2010}, the Floer homology $\HF(\mathscr{A}^{\mathfrak{M}}_\varphi)$ is well-defined. By making use of continuation homomorphisms we have that
	\begin{equation*}
		0 = \HF(\mathscr{A}^{\mathfrak{M}}_\varphi) \cong \HF(\mathscr{A}^{(\chi H,0)}_\varphi) \cong \RFH^\varphi(\Sigma,M),
	\end{equation*}
	\noindent where the last equation is the observation that twisted Rabinowitz--Floer homology in the autonomous case extends to the weakly time-dependent case without any issues. Crucial is, that the period--action equality (see Remark \ref{rem:period-action_equality})  is still valid. Indeed, we compute
	\begin{equation*}
		\mathscr{A}^{(\chi H,0)}_\varphi(\gamma,\tau) = \int_0^1 \gamma^*\lambda = \int_0^1 \lambda(\dot{\gamma}) = \tau \int_0^1 \chi \lambda(R(\gamma)) = \tau \int_0^1 \chi = \tau
	\end{equation*}
	\noindent for all $(\gamma,\tau) \in \Crit \mathscr{A}^{(\chi H,0)}_\varphi$.
\end{proof}

\newpage
\section{Applications}
\label{sec:applications}

In this chapter we give two applications of the abstract machinery developed in the previous section and prove Theorem \ref{thm:my_result} as well as Theorem \ref{thm:forcing} (see Theorem \ref{thm:forcing_theorem}).

\subsection{Existence of Noncontractible Periodic Reeb Orbits}
We define an equivariant version of twisted Rabinowitz--Floer homology for the discrete group $\mathbb{Z}_m$ following \cite[p.~487]{albersfrauenfelder:eh:2012}. In general, suppose that a topological manifold $M$ admits a group action by a topological group $G$. Then there exists a principal $G$-bundle $EG \to BG$, where $BG = EG/G$ denotes the classifying space of $G$ and $EG$ is weakly contractible. Then $G$ acts freely on $EG \times M$ via the diagonal action. Thus we can define the \bld{$G$-equivariant homology of $M$} by
\begin{equation*}
	\operatorname{H}_\ast^G(M;R) := \operatorname{H}_\ast(EG \times M/G;R),
\end{equation*}
\noindent for any coefficient ring $R$ and where $EG \times M/G$ is the homotopy quotient of $M$ by $G$. See \cite[p.~30--31]{tu:equivariant:2020}.  For example, if $G = \mathbb{Z}_m$, then $E\mathbb{Z}_m = \mathbb{S}^\infty$ and $B\mathbb{Z}_m = \mathbb{S}^\infty/\mathbb{Z}_m$ is a lens space. Since $\mathbb{Z}_m$ acts freely on $\mathbb{S}^{2n - 1}$, there is a fibre bundle
\begin{equation*}
	EG \to EG \times M/G \to M/G,
\end{equation*}
\noindent inducing an isomorphism
\begin{equation*}
	\operatorname{H}_*^G(M) = \operatorname{H}_*(M/G)
\end{equation*}
\noindent by \cite[Corollary~9.6]{tu:equivariant:2020} and \cite[Theorem~3.3]{tu:equivariant:2020}. This observation will be crucial in widehat follows. Explicitly, consider the free smooth discrete action on the odd-dimensional sphere generated by the rotation $\varphi$ from Example \ref{ex:rotation}. Define a twisted defining Hamiltonian function $H \in \mathscr{F}_\varphi(\mathbb{S}^{2n - 1})$ by
\begin{equation*}
	H(z) := \frac{1}{2}\del[1]{\beta(\abs[0]{z}^2) - 1}
\end{equation*}
\noindent for some sufficiently small mollification of the piecewise linear function 
\begin{equation*}
	\beta(r) = \begin{cases}
		\frac{1}{2} & r \in \intoc[1]{-\infty,\frac{1}{2}},\\
		r & r \in \intcc[1]{\frac{1}{2},\frac{3}{2}},\\
		\frac{3}{2} & r \in \intco[1]{\frac{3}{2},+\infty}. 
	\end{cases}
\end{equation*}
Fix a $\varphi$-invariant $\omega$-compatible almost complex structure on $(\mathbb{C}^n,\lambda)$, where $\lambda$ is given by \eqref{eq:Liouville_form}. Then the rotation $\varphi$ induces a free $\mathbb{Z}_m$-action on $\Crit \mathscr{A}^H_\varphi$ and on the moduli space of twisted negative gradient flow lines with cascades of $\mathscr{A}^H_\varphi$. Therefore, we can define the \bld{$\mathbb{Z}_m$-equivariant twisted Rabinowitz-Floer homology}
\begin{equation*}
	\overline{\RFH}^\varphi_k(\mathbb{S}^{2n - 1}/\mathbb{Z}_m) := \frac{\ker \overline{\partial}_k}{\im \overline{\partial}_{k + 1}} \qquad \forall k \in \mathbb{Z},
\end{equation*}
\noindent as the homology of the $\mathbb{Z}$-graded chain complex (see Remark \ref{rem:grading})
\begin{equation*}
	\overline{\partial}_k \colon \RFC^\varphi_k(\mathbb{S}^{2n - 1},\mathbb{C}^n)/\mathbb{Z}_m \to \RFC^\varphi_{k - 1}(\mathbb{S}^{2n - 1},\mathbb{C}^n)/\mathbb{Z}_m
\end{equation*}
\noindent given by
\begin{equation*}
	\overline{\partial}_k[(\gamma,\tau)] := [\partial_k(\gamma,\tau)] \qquad (\gamma,\tau) \in \Crit h,
\end{equation*}
\noindent for some $\varphi$-invariant Morse function $h$ on $\Crit \mathscr{A}^H_\varphi$. More generally, if $G$ is a finite symmetry of a Hamiltonian system which acts free on the displaceable regular energy hypersurface, one can define the $G$-equivariant twisted Rabinowitz--Floer homology as above if the twisted Rabinowitz--Floer homology is defined. Under some mild index assumption on the Conley--Zehnder index, the resulting $G$-equivariant twisted Rabinowitz--Floer homology is isomorphic to the Tate homology of $G$ with coefficients in $\mathbb{Z}_2$. See \cite[Theorem~5.6]{ruck:equivariant:2022} for a proof and \cite[Definition~6.2.4]{weibel:homological_algebra:1994} as well as \cite[p.~135]{brown:groups:1982} for a definition of Tate homology. 

\begin{theorem}
	Let $n \geq 2$. For $m \geq 1$ consider the rotation
	\begin{equation*}
		\varphi \colon \mathbb{C}^n \to \mathbb{C}^n, \quad \varphi(z_1,\dots,z_n) := \del[1]{e^{2\pi i k_1/m}z_1,\dots,e^{2\pi i k_n/m}z_n}
	\end{equation*}
	\noindent for $k_1,\dots,k_n \in \mathbb{Z}$ coprime to $m$. Then
	\begin{equation*}
		\overline{\RFH}^\varphi_k(\mathbb{S}^{2n - 1}/\mathbb{Z}_m) \cong \begin{cases}
			\mathbb{Z}_2 & m \text{ even},\\
			0 & m \text{ odd},
		\end{cases} \qquad \forall k \in \mathbb{Z},
	\end{equation*}
	If $m$ is even, then $\overline{\RFH}^\varphi_k(\mathbb{S}^{2n - 1}/\mathbb{Z}_m)$ is generated by a noncontractible periodic Reeb orbit in the lens space $\mathbb{S}^{2n - 1}/\mathbb{Z}_m$ for all $k \in \mathbb{Z}$. 
	\label{thm:equivariant_twisted_rfh}
\end{theorem}

\begin{proof}
	First we consider the special case
	\begin{equation*}
		\varphi \colon \mathbb{C}^n \to \mathbb{C}^n, \qquad \varphi(z) = e^{2\pi i/m}z.
	\end{equation*}
	The hypersurface $\mathbb{S}^{2n - 1} \subseteq (\mathbb{C}^n,\lambda)$ is of restricted contact type with contact form $\lambda\vert_{\mathbb{S}^{2n - 1}}$ and associated Reeb vector field
\begin{equation*}
	R = 2\sum_{j = 1}^n\del[3]{y_j \frac{\partial}{\partial x_j} - x_j\frac{\partial}{\partial y_j}}\bigg\vert_{\mathbb{S}^{2n - 1}}= 2i\sum_{j = 1}^n\del[3]{\overline{z}_j\frac{\partial}{\partial \overline{z}_j} - z_j\frac{\partial}{\partial z_j}}\bigg\vert_{\mathbb{S}^{2n - 1}}.
\end{equation*}
Suppose $(\gamma,\tau) \in \Crit \mathscr{A}^H_\varphi$. If $\tau = 0$, then $\gamma$ is constant. This cannot happen as $\Fix(\varphi\vert_{\mathbb{S}^{2n - 1}}) = \emptyset$. So we assume $\tau \neq 0$. Define a reparametrisation
	\begin{equation*}
		\gamma_\tau \colon \mathbb{R} \to \mathbb{S}^{2n - 1}, \qquad \gamma_\tau(t) := \gamma(t/\tau).
	\end{equation*}
	Then $\gamma_\tau$ is the unique integral curve of $R$ starting at $z := \gamma(0)$ and thus
	\begin{equation*}
		\gamma_\tau(t) = e^{-2it}z \qquad \forall t \in \mathbb{R}.
	\end{equation*}
	From $\gamma(t) = \gamma_\tau(\tau t)$ and the requirement
	\begin{equation*}
		e^{-2i\tau}z = \gamma(1) = \varphi(\gamma(0)) = \varphi(z) = e^{2\pi i/m}z,
	\end{equation*}
	\noindent we conclude $\tau \in \frac{\pi}{m}(m\mathbb{Z} - 1)$. Hence
	\begin{equation*}
		\Crit \mathscr{A}^H_\varphi = \cbr[1]{\del[1]{\phi^{\tau_k R}(z),\tau_k} : k \in \mathbb{Z}, z \in \mathbb{S}^{2n - 1}} \cong \mathbb{S}^{2n - 1} \times \mathbb{Z},
	\end{equation*}
	\noindent for any $H \in \mathscr{F}_\varphi(\mathbb{S}^{2n - 1})$, where we define $\tau_k := \frac{\pi}{m}(mk - 1)$. By Proposition \ref{prop:kernel_hessian_contact}, $(z_0,\eta) \in T_z\mathbb{S}^{2n - 1} \times \mathbb{R}$ belongs to the kernel of the Hessian at $(z,k) \in \Crit \mathscr{A}^H_\varphi$ if and only if $\eta = 0$ and
	\begin{equation*}
		z_0 \in \ker \del[1]{D(\phi^R_{-\tau_k} \circ \varphi)\vert_z - \id_{T_z\mathbb{S}^{2n - 1}}}.
	\end{equation*}
	A direct computation yields $D(\phi^R_{-\tau_k} \circ \varphi)\vert_z = \id_{T_z \mathbb{S}^{2n - 1}}$ and thus the twisted Rabinowitz action functional $\mathscr{A}^H_\varphi$ is Morse--Bott with spheres.

\begin{figure}[h!tb]
	\centering
	\includegraphics[width=\textwidth]{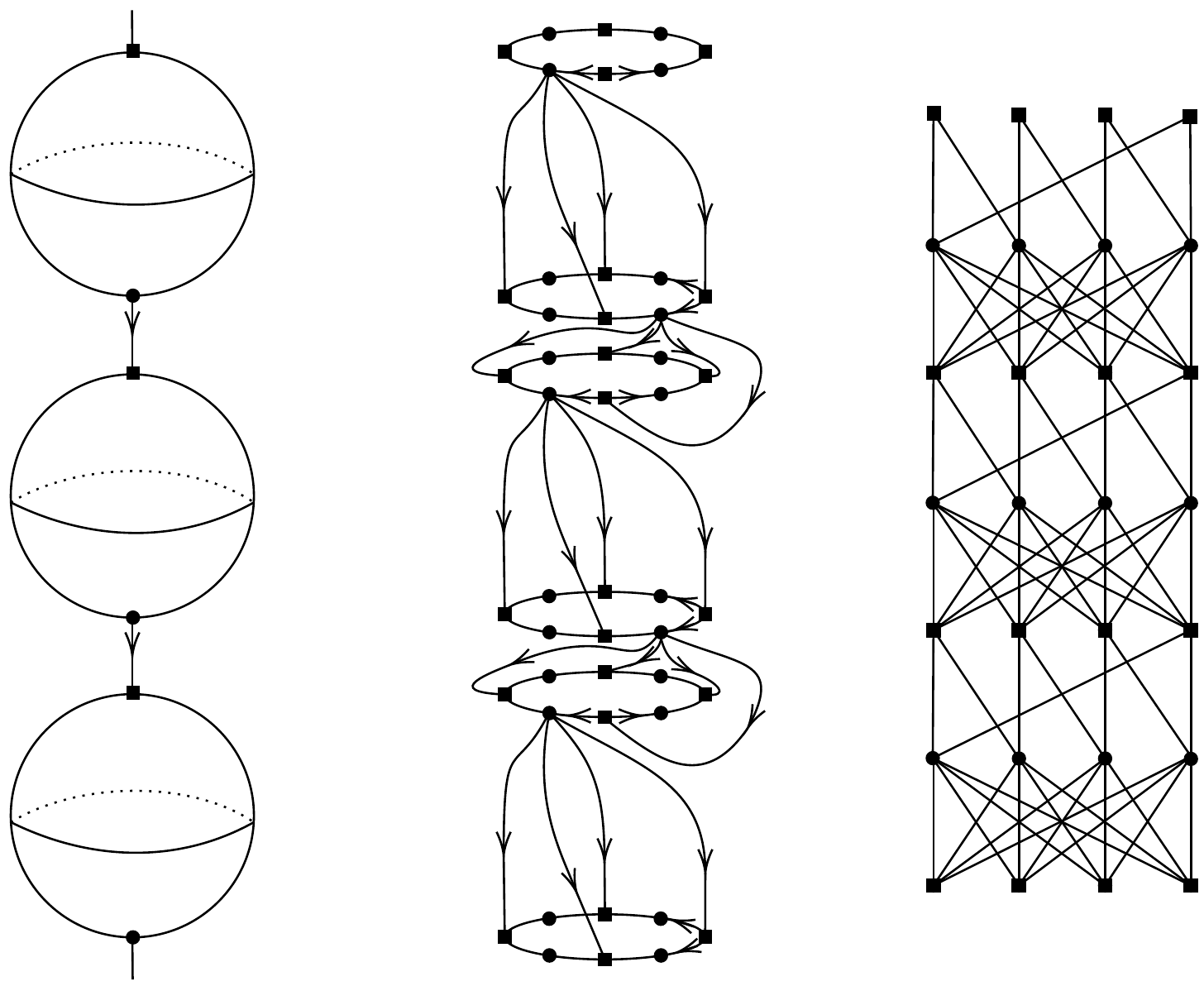}
	\caption{The critical manifold $\mathbb{S}^{2n - 1} \times \mathbb{Z}$ with the standard height function, the Morse--Bott function $f$ and the resulting chain complex.}
	\label{fig:equivariant_rfh_1}
\end{figure}

	The full Conley--Zehnder index \cite[Definition~10.4.1]{frauenfelderkoert:3bp:2018} gives rise to a locally constant function
	\begin{equation*}
		\widehat{\mu}_{\mathrm{CZ}} \colon \Crit \mathscr{A}^H_\varphi \to \mathbb{Z}, \qquad \widehat{\mu}(z,k) = (2k-1)n.
	\end{equation*}
	Note that the definition of the Conley--Zehnder index also applies in this degenerate case, compare \cite[Remark~10.4.2]{frauenfelderkoert:3bp:2018}. By the adapted proof of the Hofer--Wysocki--Zehnder Theorem \cite[Theorem~12.2.1]{frauenfelderkoert:3bp:2018} to the $n$-dimensional setting, the full Conley--Zehnder index coincides with the transverse Conley--Zehnder index $\mu_{\CZ}$. Indeed, for a critical point $(\gamma,\tau) \in \Crit \mathscr{A}^H_\varphi$ define a smooth path
	\begin{equation*}
		\Psi \colon I \to \Sp(n), \qquad \Psi_t := D\phi^H_{\tau t}\vert_{\gamma(0)} \colon \mathbb{C}^n \to \mathbb{C}^n.
	\end{equation*}
	Adapting the proof of \cite[Lemma~12.2.3~(iii)]{frauenfelderkoert:3bp:2018}, we get that
	\begin{equation*}
		\Psi_1(R(\gamma(0))) = R(\gamma(1)) \qquad \text{and} \qquad \Psi_1(\gamma(0)) = \gamma(1).
	\end{equation*}
	Arguing as in \cite[p.~235--236]{frauenfelderkoert:3bp:2018} we conclude
	\begin{equation*}
		\mu_{\CZ}(\gamma,\tau) = \widehat{\mu}_{\CZ}(\gamma,\tau).
	\end{equation*}

	Fix $z_0 \in \mathbb{S}^{2n - 1}$ and define $\eta := \phi^{\tau_0 R}(z_0)$. Note that $\phi^{\tau_k R}(z)$ belongs to the same equivalence class in $\pi_0 \mathscr{L}_\varphi \mathbb{S}^{2n - 1}$ as $\eta$ for all $z \in \mathbb{S}^{2n - 1}$ and $k \in \mathbb{Z}$ because $\mathbb{S}^{2n - 1}$ is simply connected for $n\geq 2$. Let $h \in C^\infty(\mathbb{S}^{2n - 1})$ be the standard height function. By Remark \ref{rem:grading}, $\RFH^\varphi(\mathbb{S}^{2n - 1},\mathbb{C}^n)$ carries the $\mathbb{Z}$-grading
	\begin{equation*}
		\mu((z,k),(z_0,0)) + \ind_h(z) = 2kn + \ind_h(z) \qquad \forall (z,k) \in \mathbb{S}^{2n - 1} \times \mathbb{Z}.
	\end{equation*}
	We claim that the number of twisted negative gradient flow lines between the minimum of $\mathbb{S}^{2n - 1} \times \{k + 1\}$ and the maximum of $\mathbb{S}^{2n - 1} \times \{k\}$ must be odd, so that the critical manifold $\Crit \mathscr{A}^H_\varphi$ looks like a string of pearls, see Figure \ref{fig:equivariant_rfh_1}. Indeed, if there is an even number of such negative gradient flow lines, then $\RFH^\varphi_\ast(\mathbb{S}^{2n - 1},\mathbb{C}^n) \neq 0$, contradicting Theorem \ref{thm:displaceable} as $\mathbb{S}^{2n - 1}$ is displaceable in the completion $\mathbb{C}^n$. To compute the $\mathbb{Z}_m$-equivariant twisted Rabinowitz--Floer homology, choose the additional $\mathbb{Z}_m$-invariant Morse--Bott function $f$ from Example \ref{ex:sphere}. Additionaly, choose a $\mathbb{Z}_m$-invariant Morse function on $\Crit f$. For example, one can take
	\begin{equation*}
		h \colon \mathbb{T} \to \mathbb{R}, \qquad h(t) := \cos(2\pi m t).
	\end{equation*}
	The resulting chain complex is given by
	\begin{equation*}
		\begin{tikzcd}[column sep = scriptsize]
			\dots \arrow[r] & \mathbb{Z}_2^m \arrow[r,"\mathbbm{1}"] & \mathbb{Z}^m_2 \arrow[r,"A"] & \mathbb{Z}^m_2 \arrow[r,"\mathbbm{1}"] & \mathbb{Z}^m_2 \arrow[r,"A"] & \mathbb{Z}^m_2 \arrow[r,"\mathbbm{1}"] & \mathbb{Z}^m_2 \arrow[r] & \dots 
		\end{tikzcd}
	\end{equation*}
	\noindent where $\mathbbm{1} \in M_{m \times m}(\mathbb{Z}_2)$ has every entry equal to $1$ and $A \in M_{m \times m}(\mathbb{Z}_2)$ is defined by
	\begin{equation*}
		A := I_{m \times m} + \sum_{j = 1}^{m - 1} e_{(j + 1)j} + e_{1 m},
	\end{equation*}
	\noindent where $e_{ij} \in M_{m \times m}(\mathbb{Z}_2)$ satisfies $(e_{ij})_{kl} = \delta_{ik}\delta_{jl}$. Thus the resulting chain complex looks like a rope ladder. Compare Figure \ref{fig:equivariant_rfh_1}. Passing to the quotient via the free $\mathbb{Z}_m$-action, we get the acyclic chain complex
	\begin{equation*}
		\begin{tikzcd}
			\dots \arrow[r] & \mathbb{Z}_2 \arrow[r,"0"] & \mathbb{Z}_2 \arrow[r,"0"] & \mathbb{Z}_2 \arrow[r,"0"] & \mathbb{Z}_2 \arrow[r] & \dots 
		\end{tikzcd}
	\end{equation*}
	\noindent if $m$ is even and the alternating chain complex 
	\begin{equation*}
		\begin{tikzcd}[column sep = scriptsize]
			\dots \arrow[r] & \mathbb{Z}_2 \arrow[r,"1"] & \mathbb{Z}_2 \arrow[r,"0"] & \mathbb{Z}_2 \arrow[r,"1"] & \mathbb{Z}_2 \arrow[r,"0"] & \mathbb{Z}_2 \arrow[r,"1"] & \mathbb{Z}_2 \arrow[r] & \dots 
		\end{tikzcd}
	\end{equation*}
	\noindent if $m$ is odd.

	For the general case, we note that $\Crit \mathscr{A}^H_\varphi$ is a disjoint union of spheres and different copies of $\mathbb{Z}$. It is therefore rather hard to compute the equivariant twisted Rabinowitz--Floer homology directly via analysing the critical manifold. By Theorem \ref{thm:displaceable}, there is a canonical isomorphism
	\begin{equation*}
		\RFH^\varphi_*(\mathbb{S}^{2n - 1},\mathbb{C}^n) \cong \RFH_*(\mathbb{S}^{2n - 1},\mathbb{C}^n)
	\end{equation*}
	\noindent inducing a canonical isomorphism
	\begin{equation*}
		\overline{\RFH}^\varphi_*(\mathbb{S}^{2n - 1}/\mathbb{Z}_m) \cong \RFH^{\mathbb{Z}_m}_*(\mathbb{S}^{2n - 1},\mathbb{C}^n),
	\end{equation*}
	\noindent where $\RFH^{\mathbb{Z}_m}_*(\mathbb{S}^{2n - 1},\mathbb{C}^n)$ denotes the $\mathbb{Z}_m$-equivariant Rabinowitz--Floer homology constructed in \cite[p.~487]{albersfrauenfelder:eh:2012}. A computation similar as before shows
	\begin{equation*}
		\RFH^{\mathbb{Z}_m}_k(\mathbb{S}^{2n - 1},\mathbb{C}^n) \cong \begin{cases}
			\mathbb{Z}_2 & m \text{ even},\\
			0 & m \text{ odd},
		\end{cases} \qquad \forall k \in \mathbb{Z}.
	\end{equation*}
	The crucial observation is, that $\Crit \mathscr{A}^H \cong \mathbb{S}^{2n - 1} \times \mathbb{Z}$. In particular, the string of pearls looks like in Figure \ref{fig:equivariant_rfh_2}.

	Lastly, $\overline{\RFH}^\varphi_k(\mathbb{S}^{2n - 1}/\mathbb{Z}_m)$ is generated by a noncontractible periodic Reeb orbit in $\mathbb{S}^{2n - 1}/\mathbb{Z}_m$ for all $k \in \mathbb{Z}$ by Lemma \ref{lem:noncontractible}.
\end{proof}

\begin{figure}[h!tb]
	\centering
	\includegraphics[width=.3\textwidth]{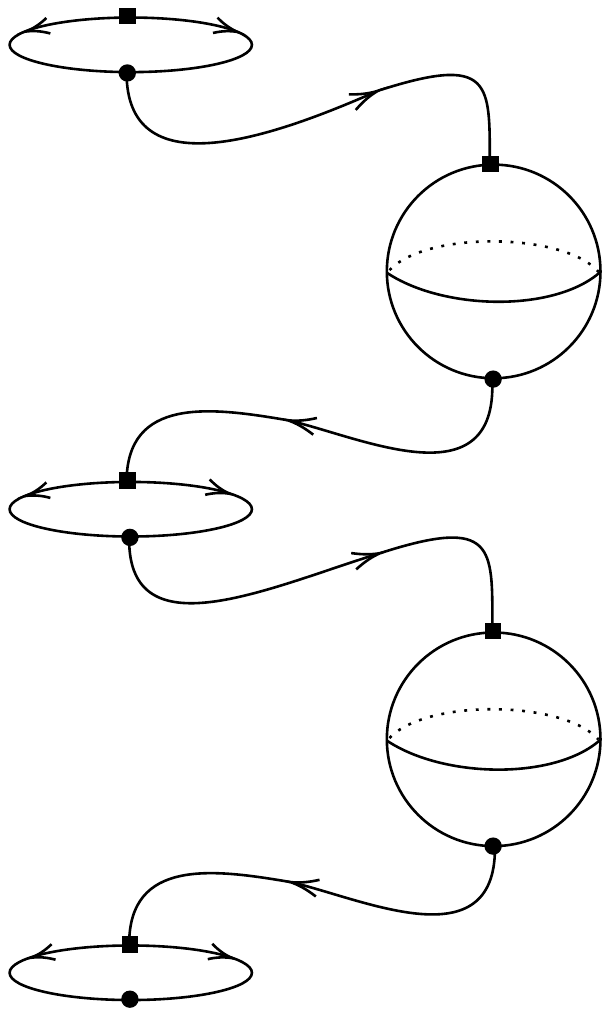}
	\caption{The critical manifold $\Crit \mathscr{A}^H_\varphi$ together with the standard height function.}
	\label{fig:equivariant_rfh_2}
\end{figure}

\begin{remark}[Coefficients]
	As $\overline{\RFH}^\varphi_*(\mathbb{S}^{2n - 1}/\mathbb{Z}_m)$ vanishes for odd $m$, one should rather consider twisted Rabinowitz--Floer homology with coefficients in $\mathbb{Z}$ in this case. Using the polyfold approach, it might be possible to invoke \cite[Chapter~6]{hoferwysockizehnder:polyfolds:2021} to define coherent orientations on the moduli spaces. However, Lagrangian Floer homology admits an abstract polyfold description, but it is not always possible to define coherent orientations.
\end{remark}

Using Theorem \ref{thm:equivariant_twisted_rfh} we can finally prove Theorem \ref{thm:my_result}.

\begin{proof}[of Theorem \ref{thm:my_result}]
	By assumption, $\Sigma$ bounds a star-shaped domain $D$ with respect to the origin. Thus $(D \cup \Sigma,\lambda)$ is a Liouville domain with $\lambda$ given by \eqref{eq:Liouville_form}. By rescaling we may assume that $\mathbb{S}^{2n - 1} \subseteq D$. Define a smooth function
\begin{equation*}
	\delta \colon \Sigma \to \intoo[0]{-\infty,0}
\end{equation*}
\noindent by requiring $\delta(x)$ to be the unique number such that $\phi^X_{\delta(x)}(x) \in \mathbb{S}^{2n - 1}$, $x \in \Sigma$, where $X \in \mathfrak{X}(\mathbb{C}^n)$ denotes the Liouville vector field \eqref{eq:Liouville_vector_field}. We claim that $\delta \circ \varphi = \delta$. Indeed, $\delta(\varphi(x))$ is the unique number such that $\phi^X_{\delta(\varphi(x))}(\varphi(x)) \in \mathbb{S}^{2n - 1}$. As the flow of $X$ and $\varphi$ commute by the proof of Lemma \ref{lem:defining_Hamiltonian}, we conclude that $\phi^X_{\delta(\varphi(x))}(x) \in \mathbb{S}^{2n - 1}$. Define a smooth family of star-shaped hypersurfaces $(\Sigma_\sigma)_{\sigma \in I}$
	\begin{equation*}
		\Sigma_\sigma := \cbr[1]{\phi^X_{\sigma\delta(x)}(x) : x \in \Sigma} \subseteq \mathbb{C}^n.
	\end{equation*}
	Then we compute
	\begin{align*}
		\varphi(\Sigma_\sigma) &= \cbr[1]{\varphi\del[1]{\phi^X_{\sigma\delta(x)}(x)} : x \in \Sigma}\\
		&= \cbr[1]{\phi^X_{\sigma\delta(x)}(\varphi(x)) : x \in \Sigma}\\
		&= \cbr[1]{\phi^X_{\sigma\delta(\varphi(x))}(\varphi(x)) : x \in \Sigma}\\
		&= \cbr[1]{\phi^X_{\sigma\delta(y)}(y) : y \in \varphi(\Sigma)}\\
		&= \cbr[1]{\phi^X_{\sigma\delta(y)}(y) : y \in \Sigma}\\
		&= \Sigma_\sigma
	\end{align*}
	\noindent for all $\sigma \in I$ and therefore we can find a twisted homotopy $(H_\sigma)_{\sigma \in I}$ of Liouville domains in $\mathbb{C}^n$. By Theorem \ref{thm:invariance} we have that
	\begin{equation*}
		\RFH^\varphi_*(\Sigma,\mathbb{C}^n) \cong \RFH_*^\varphi(\mathbb{S}^{2n - 1},\mathbb{C}^n),
	\end{equation*}
	\noindent giving rise to a canonical isomorphism of the associated $\mathbb{Z}_m$-equivariant twisted Rabinowitz--Floer homology
	\begin{equation*}
		\overline{\RFH}^\varphi_*(\Sigma/\mathbb{Z}_m) \cong \overline{\RFH}^\varphi_*(\mathbb{S}^{2n - 1}/\mathbb{Z}_m). 
	\end{equation*}
	However, by Theorem \ref{thm:equivariant_twisted_rfh} the latter does not vanish as $m \geq 2$ is even.
\end{proof}

Using Theorem \ref{thm:equivariant_twisted_rfh} it is also possible to generalise \cite[Theorem~1.2]{albersfrauenfelder:eh:2012}. Define the set of \bld{$\varphi$-invariant Hamiltonian symplectomorphisms} by
	\begin{equation*}
		\Ham^\varphi(\mathbb{C}^n,d\lambda) := \{\varphi_F \in \Ham(\mathbb{C}^n,d\lambda) : F_t(x) = F_t(\varphi(x)) \> \forall (x,t) \in \mathbb{C}^n \times I\}.
	\end{equation*}
	If $\varphi_F \in \Ham^\varphi(\mathbb{C}^n,d\lambda)$, then $\varphi \circ \varphi_F = \varphi_F \circ \varphi$. In particular $0 \in \Fix(\varphi_F)$, and thus no element in $\Ham^\varphi(\mathbb{C}^n,d\lambda)$ can displace a star-shaped hypersurface with respect to the origin in $\mathbb{C}^n$. We have the following result.

\begin{theorem}
	\label{thm:leaf-wise}
	Let $\Sigma \subseteq \mathbb{C}^n$ be a compact connected star-shaped hypersurface invariant under the rotation $\varphi$. Every element in $\Ham^\varphi(\mathbb{C}^n,d\lambda)$ admits infinitely many leaf-wise intersection points on $\Sigma$ or there does exist a leaf-wise intersection point on a closed leaf.
\end{theorem}

\begin{proof}
	We reproduce the proof in \cite{albersfrauenfelder:eh:2012} for completeness with minor modifications. Let $\varphi_F \in \Ham_c^\varphi(\mathbb{C}^n,d\lambda)$ and for $r \in \intcc[0]{0,1}$ consider the smooth family of perturbed Rabinowitz action functionals
	\begin{equation*}
		\mathscr{A}_r \colon \mathscr{L}\mathbb{C}^n \times \mathbb{R} \to \mathbb{R}
	\end{equation*}
	\noindent defined by
	\begin{equation*}
		\mathscr{A}_r(\gamma,\tau) := \int_0^1 \gamma^*\lambda - \tau \int_0^1 H_r(\gamma(t))dt - r\int_0^1 F_t(\gamma(t)) dt,
	\end{equation*}
	\noindent where $(H_r)_{r \in I}$ is a twisted homotopy of Liouville domains from $\mathbb{S}^{2n - 1}$ to $\Sigma$. Clearly, every $\mathscr{A}_r$ is $\varphi$-invariant. As in \cite[Definition~5.1]{albersfrauenfelder:spectral:2010}, we define the spectral value $\sigma([\xi])$ of a homology class $[\xi] \in \RFH^{\mathbb{Z}_m}_\ast(\mathbb{S}^{2n - 1},\mathbb{C}^n)$ by
	\begin{equation*}
		\sigma([\xi]) := \inf_{\eta \in [\xi]}\max_{\eta_{(\gamma,\tau)} \neq 0} \mathscr{A}_0(\gamma,\tau) \in \mathbb{R} \cup \{-\infty\}.
	\end{equation*}
	Moreover, we define the set
	\begin{equation*}
		\mathfrak{S} := \cbr[0]{\sigma([\xi]) : [\xi] \in \RFH^{\mathbb{Z}_m}_\ast(\mathbb{S}^{2n - 1},\mathbb{C}^n)}.
	\end{equation*}
	By Theorem \ref{thm:equivariant_twisted_rfh}, we conclude that $\mathfrak{S} = 2\pi \mathbb{Z}$. Hence $\mathscr{A}_0$ has critical values of arbitrarily large critical value and so does $\mathscr{A}_1$ by \cite[Corollary~5.14]{albersfrauenfelder:spectral:2010}. Thus $\mathscr{A}_1$ has infinitely many critical points which give rise to leaf-wise intersection points by Lemma \ref{lem:variational_characterisation_twisted_leaf-wise_intersections}. The map
	\begin{equation*}
		\Crit \mathscr{A}_1 \to \cbr[0]{\text{leaf-wise intersection points}}
	\end{equation*}
	\noindent is injective unless there exists a leaf-wise intersection point on a closed leaf. For the general case, use cut-off functions.
\end{proof}

\subsection{A Forcing Theorem for Twisted Periodic Reeb Orbits}
\label{sec:forcing}

\begin{definition}[Twisted Stable Hypersurface]
	\label{def:twisted_stable_hypersurface}
	Let $(\Sigma,\omega\vert_\Sigma, \lambda)$ be a stable hypersurface in a connected symplectic manifold $(M,\omega)$ and $\varphi \in \Symp(M,\omega)$. We say that $\Sigma$ is \bld{twisted by $\varphi$}, if $\varphi(\Sigma) = \Sigma$, $\varphi$ is of finite order and $\varphi^*\lambda = \lambda$.
\end{definition}

\begin{example}
	\label{ex:twisted_stable_hypersurface}
	Consider the stable hypersurface $\Sigma_c \subseteq (T^*\mathbb{T}^n, \omega_\sigma,H)$ for $c > 0$ as in Example \ref{ex:twisted_torus}. Let $\varphi \in \Diff(\mathbb{T}^n)$ be an isometry of finite order such that 
	\begin{equation}
		\label{eq:twist_condition}
		D\varphi \circ J = J \circ D\varphi
	\end{equation}
	\noindent holds and consider the cotangent lift \eqref{eq:cotangent_lift}
	\begin{equation*}
		D\varphi^\dagger \colon \mathbb{T}^n \times \mathbb{R}^n \to \mathbb{T}^n \times \mathbb{R}^n, \qquad D\varphi^\dagger(q,p) = \del[1]{\varphi(q),\del[1]{D\varphi^{-1}(q)}^tp}.
	\end{equation*}
	Then clearly $\varphi(\Sigma_c) = \Sigma_c$ as $\varphi$ is an isometry and $D\varphi^\dagger$ is of finite order as $\varphi$ is. Moreover, $D\varphi^\dagger \in \Symp(T^*\mathbb{T}^n,\omega_\sigma)$, as $D\varphi^\dagger \in \Symp(T^*\mathbb{T}^n,\omega_0)$ by Proposition \ref{prop:physical_transformation} and $D\varphi^\dagger$ preserves $\sigma$ by assumption \eqref{eq:twist_condition}. Lastly, we have that $\varphi^*\lambda = \lambda$ as one sees by considering the formula \eqref{eq:stabilising_form} together with assumption \eqref{eq:twist_condition}. 
\end{example}

Let $(\Sigma,\omega\vert_\Sigma,\lambda)$ be a twisted stable hypersurface for $\varphi \in \Symp(M,\omega)$ in a connected symplectically aspherical symplectic manifold $(M,\omega)$, that is, $[\omega]\vert_{\pi_2(M)} = 0$. As $\varphi$ is of finite order by assumption, we can define the \bld{set of twisted contractible loops}, written $\Lambda_\varphi M \subseteq \Lambda M$, as follows. We say that a contractible free loop $v \in \Lambda M$ is in $\Lambda_\varphi M$, if there exists $\gamma \in \mathscr{L}_\varphi M$ such that
\begin{equation*}
	v(t) = \gamma(mt) \qquad \forall t \in \mathbb{T},
\end{equation*}
\noindent where $m := \ord \varphi$. Then we can define a generalisation of the twisted Rabinowitz action functional
\begin{equation}
	\label{eq:Rabinowitz_functional}
	\mathscr{A}^H_\varphi \colon \Lambda_\varphi M \times \mathbb{R} \to \mathbb{R}, \quad \mathscr{A}^H_\varphi(v,\tau) := \frac{1}{m}\int_{\mathbb{D}} \overline{v}^*\omega - \tau \int_0^1 H(v(t))dt,
\end{equation}
\noindent where $\overline{v} \in C^\infty(\mathbb{D},M)$ is a filling of $v$ and $H$ is any twisted defining Hamiltonian function for $\Sigma$. Then $(v,\tau) \in \Crit \mathscr{A}^H_\varphi$ if and only if $(\gamma,\tau) \in \mathscr{L}_\varphi \Sigma$ solves
\begin{equation*}
	\dot{\gamma}(t) = \tau R(\gamma(t)) \qquad \forall t \in \mathbb{R},
\end{equation*}
\noindent where $R \in \mathfrak{X}(\Sigma)$ denotes the stable Reeb vector field \ref{def:Reeb}. We call the projection of the set of critical points of $\mathscr{A}^H_\varphi$ to $\Lambda_\varphi M$ \bld{contractible twisted closed characteristics} and denote it by $\mathscr{C}_\varphi(\Sigma)$. Define a function, called the \bld{$\omega$-energy}, by
\begin{equation*}
	\Omega \colon \mathscr{C}_\varphi(\Sigma) \to \mathbb{R}, \qquad \Omega(v) := \frac{1}{m}\int_{\mathbb{D}} \overline{v}^*\omega.
\end{equation*}
It follows that $\Omega(v) = \mathscr{A}^H_\varphi(v,\tau)$.

\begin{example}
	\label{ex:twisted_characteristic}
	Consider the twisted stable hypersurface $\Sigma_c \subseteq (T^*\mathbb{T}^n,\omega_\sigma,H)$ as in Example \ref{ex:twisted_stable_hypersurface}.	By adapting Example \ref{ex:magnetic_flow}, we have that $(q,p) \in \Sigma_k$ gives rise to a contractible twisted closed characteristic if and only if
	\begin{equation*}
		\int_0^\tau e^{sJ}p ds + q = \varphi(q), \quad e^{\tau J}p = \del[1]{D\varphi^{-1}(q)}^tp, \quad \text{and} \quad \norm[0]{p}^2 = 2c.
	\end{equation*}
	A computation similar to \cite[p.~1843]{cieliebakfrauenfelderpaternain:mane:2010} shows
	\begin{equation*}
		\Omega \colon \mathscr{C}_\varphi(\Sigma_c) \to \mathbb{R}, \qquad \Omega(v) = c\tau.
	\end{equation*}
\end{example}

In order to state the main result of this section, we need two additional preliminary definitions.

\begin{definition}[{Morse--Bott Component, \cite[p.~86]{albersfrauenfelder:rfh:2010}}]
	Let $M$ be a smooth manifold and $f \in C^\infty(M)$. A subset $C \subseteq \Crit f$ is called a \bld{Morse--Bott component}, if
	\begin{enumerate}[label=\textup{(\roman*)}]
		\item $C$ is a connected embedded submanifold of $M$.
		\item $T_xC = \ker \Hess f(x)$ for all $x \in C$.
	\end{enumerate}
\end{definition}

\begin{example}[{\cite[Lemma~2.12]{albersfrauenfelder:rfh:2010}}]
	In the setting of Proposition \ref{prop:kernel_hessian_contact}, any connected component of $\Fix(\varphi\vert_\Sigma) \subseteq \Crit \mathscr{A}^H_\varphi$ is a Morse--Bott component. Indeed, we have that
	\begin{equation*}
		\ker \Hess \mathscr{A}^H_\varphi\vert_{(x,0)} \cong \ker (D\varphi_x - \id_{T_x \Sigma}) = T_x\Fix(\varphi\vert_\Sigma)
	\end{equation*}
	\noindent for all $x \in \Fix(\varphi\vert_\Sigma)$.
\end{example}

\begin{definition}[{\cite[p.~1768]{cieliebakfrauenfelderpaternain:mane:2010}}]
    A symplectic manifold $(M,\omega)$ is called \bld{geometrically bounded}, if there exists an $\omega$-compatible almost complex structure $J$ and a complete Riemannian metric such that the following conditions hold.
    \begin{enumerate}[label=\textup{(\roman*)}]
        \item There are constants $C_0,C_1 > 0$ with
        \begin{equation*}
            \omega(Jv,v) \geq C_0\norm[0]{v}^2 \qquad \text{and} \qquad \abs[0]{\omega(u,v)} \leq C_1\norm[0]{u}\norm[0]{v}
        \end{equation*}
        \noindent for all $u,v \in T_x M$ and $x \in M$.
        \item The sectional curvature of the metric is bounded above, and its injectivity radius is bounded away from zero.
    \end{enumerate}
\end{definition}

\begin{example}[{\cite[p.~1768]{cieliebakfrauenfelderpaternain:mane:2010}}]
	\label{ex:geometrically_bounded}
	Twisted cotangent bundles are geometrically bounded.
\end{example}

\begin{theorem}[Forcing]
	\label{thm:forcing_theorem}
	Let $\Sigma$ be a twisted stable displaceable hypersurface in a symplectically aspherical, geometrically bounded, symplectic manifold $(M,\omega)$ for some $\varphi \in \Symp(M,\omega)$ and suppose that $v^- \in \mathscr{C}_\varphi(\Sigma)$ belongs to a Morse--Bott component $C$ of the twisted Rabinowitz action functional \eqref{eq:Rabinowitz_functional}. Then there exists a contractible twisted closed characteristic $v \in \mathscr{C}_\varphi(\Sigma) \setminus C$ such that
	\begin{equation*}
		\Omega(v) - \Omega(v^-) \leq e(\Sigma).
	\end{equation*}	
\end{theorem}

\begin{corollary}
    \label{cor:schlenk}
    Let $\Sigma$ be a twisted stable displaceable hypersurface in a symplectically aspherical, geometrically bounded, symplectic manifold $(M,\omega)$ for some symplectomorphism $\varphi \in \Symp(M,\omega)$. If $\Fix(\varphi\vert_\Sigma) \neq \emptyset$, then there exists a contractible twisted closed characteristic $v \in \mathscr{C}_\varphi(\Sigma) \setminus \Fix(\varphi\vert_\Sigma)$ such that
    \begin{equation*}
        \Omega(v) \leq e(\Sigma).
    \end{equation*}
\end{corollary}

\begin{corollary}[{\cite[Theorem~12.3.4]{mcduffsalamon:st:2017}, \cite[p.~171]{hoferzehnder:hd:1994}}]
	\label{cor:displacement_energy_ball}
	We have that
	\begin{equation*}
		e(\overline{B}_r^{2n}(0)) = \pi r^2 \qquad \forall r > 0,
	\end{equation*}
	\noindent where $\overline{B}^{2n}_r(0) \subseteq \mathbb{R}^{2n}$ denotes the closed ball around the origin of radius $r$.
\end{corollary}

\begin{proof}
	By monotonicity and \cite[Exercise~12.3.7]{mcduffsalamon:st:2017} we have that
	\begin{equation*}
		e(\partial \overline{B}^{2n}_r(0)) \leq e(\overline{B}^{2n}_r(0)) \leq \pi r^2 \qquad \forall r > 0.
	\end{equation*}
	The Reeb flow on $\partial \overline{B}^{2n}_r(0)$ is given by
	\begin{equation*}
		\phi^{R_r}_t(z) = e^{-2i t/r^2}z \qquad \forall z \in \partial \overline{B}^{2n}_r(0).
	\end{equation*}
	Hence the parametrised periodic Reeb orbits are $(\phi^{R_r}(z),\tau)$ with $\tau \in \pi r^2 \mathbb{Z}$. But Corollary \ref{cor:schlenk} implies the existence of a nonconstant closed characteristic $v$ on the hypersurface $\partial \overline{B}^{2n}_r(0)$ such that
	\begin{equation*}
		0 < \tau = \Omega(v) \leq e(\partial \overline{B}^{2n}_r(0)) \leq \pi r^2.
	\end{equation*}
	This is only possible for $\tau = \pi r^2$ and the statement follows.
\end{proof}

\begin{proof}[Proof of Theorem \ref{thm:forcing_theorem}]
	This proof uses a method called a ``homotopy of homotopies argument''. Fix $\varepsilon > 0$ and choose a Hamiltonian function $F \in C^\infty_c(M \times I)$ satisfying 
	\begin{equation*}
	    \norm{F} < e(\Sigma) + \varepsilon \qquad \text{and} \qquad \varphi_F(\Sigma) \cap \Sigma = \emptyset.
	\end{equation*}
	For an appropriate twisted defining Hamiltonian function $H$ for $\Sigma$ we denote by $\mathfrak{M}$ the associated twisted Moser pair. The actual construction of $H$ is very cumbersome and is carried out in \cite{cieliebakfrauenfelderpaternain:mane:2010}. The crucial observation here is that \cite[Proposition~2.6]{cieliebakfrauenfelderpaternain:mane:2010} gives a $\varphi$-invariant stable tubular neighbourhood of $\Sigma$ as $\varphi^*\lambda = \lambda$ by invoking the equivariant Darboux--Weinstein Theorem \cite[Theorem~22.1]{guilleminsternberg:sg:1984}. Moreover, we choose a smooth family $(\beta_r)_{r \in \intco[0]{0,+\infty}}$ of cutoff functions $\beta_r \in C^\infty(\mathbb{R},I)$ such that
	\begin{equation*}
		\begin{cases}
			\beta_r(s) = 0 & \abs[0]{s} \geq r,\\
			\beta_r(s) = 1 & \abs[0]{s} \leq r - 1,\\
			s\beta'_r(s) \leq 0 & \forall s \in \mathbb{R},
		\end{cases}
	\end{equation*}
	\noindent for all $r \in \intco[0]{0,+\infty}$. Define a family of twisted Rabinowitz action functionals
	\begin{equation*}
		\mathscr{A}_r \colon \Lambda_\varphi M \times \mathbb{R} \times \mathbb{R} \to \mathbb{R}
	\end{equation*}
	\noindent by
	\begin{equation*}
		\mathscr{A}_r(v,\tau,s) := \mathscr{A}^H_\varphi(v,\tau) - \beta_r(s)\int_0^1 F_t(v(t))dt
	\end{equation*}
	\noindent for all $r \in \intco[0]{0,+\infty}$. Note that $\mathscr{A}_0 = \mathscr{A}^H_\varphi$. For a suitable $\varphi$-invariant $\omega$-compatible almost complex structure we consider the moduli space
	\begin{equation*}
		\mathscr{M} := \cbr[0]{(u,\tau,r) \in C^\infty(\mathbb{R}, \mathscr{L}_\varphi M \times \mathbb{R}) \times \intco[0]{0,+\infty} : (u,\tau,r) \text{ solution of \eqref{eq:moduli_space}}},
	\end{equation*}
	\noindent where
	\begin{equation}
		\begin{cases}
	       \partial_s (u,\tau) = \grad \mathscr{A}_r\vert_{(u(s),\tau(s),s)} & \forall s \in \mathbb{R},\\
	       \displaystyle\lim_{s \to -\infty} (u(s),\tau(s)) = (v^-,\tau^-),\\
	       \displaystyle\lim_{s \to +\infty} (u(s),\tau(s)) \in C.
	   \end{cases}
		\label{eq:moduli_space}
	\end{equation}
	 Note that always $(v^-,\tau^-,0) \in \mathscr{M}$. The proof is now based on the following observation. If 
	 \begin{equation}
	 \label{eq:assumption_moduli_space}
	     \Omega(v) > \norm[0]{F} + \Omega(v^-) \qquad \forall v \in \mathscr{C}_\varphi(\Sigma) \setminus C
	 \end{equation}
	 \noindent holds, then $\mathscr{M}$ is compact. This is absurd. Indeed, the moduli space $\mathscr{M}$ is the zero level set of a Fredholm section of a bundle over a Banach manifold. As $v^-$ belongs to a Morse--Bott component, the Fredholm section is regular at the point $v^-$, that is, the linearisation of the gradient flow equation is surjective there. By compactness, we can therefore perturb the Fredholm section to make it transverse. Hence $\mathscr{M}$ is a compact smooth manifold with boundary consisting precisely of the point $v^-$. There do not exist such manifolds. Thus we conclude that there exists $v \in  \mathscr{C}_\varphi(\Sigma) \setminus C$ such that 
	 \begin{equation*}
	     \Omega(v) - \Omega(v^-) \leq \norm[0]{F} < e(\Sigma) + \varepsilon.
	 \end{equation*}
	 As $\varepsilon > 0$ was arbitrary, the statement follows. We prove the compactness of $\mathscr{M}$ under assumption \eqref{eq:assumption_moduli_space} in four steps.

	\emph{Step 1: If $(u,\tau,r) \in \mathscr{M}$, then $E(u,\tau) \leq \norm{F}$.} We estimate
	\allowdisplaybreaks
	\begin{align*}
		E(u,\tau) &= \int_{-\infty}^{+\infty} \norm[0]{\partial_s (u,\tau)}^2ds\\
		&= \int_{-\infty}^{+\infty} d\mathscr{A}_r(\partial_s (u,\tau),s) ds\\
		&= \int_{-\infty}^{+\infty} \frac{d}{ds}\mathscr{A}_r(u,\tau,s) ds - \int_{-\infty}^{+\infty} (\partial_s \mathscr{A}_r)(u,\tau,s)ds\\
		&= \lim_{s \to +\infty} \mathscr{A}_r(u,\tau,s) - \lim_{s \to -\infty} \mathscr{A}_r(u,\tau,s) - \int_{-\infty}^{+\infty} (\partial_s \mathscr{A}_r)(u,\tau,s)ds\\
		&= \mathscr{A}_0(v^+,\tau^+) - \mathscr{A}_0(v^-,\tau^-) - \int_{-\infty}^{+\infty} (\partial_s \mathscr{A}_r)(u,\tau,s)ds\\
		&= -\int_{-\infty}^{+\infty} (\partial_s \mathscr{A}_r)(u,\tau,s)ds\\
		&= \int_{-\infty}^{+\infty} \dot{\beta}_r(s)\int_0^1 F_t(u(s,t))dtds\\
		&\leq \norm[0]{F}_+ \int_{-\infty}^0 \dot{\beta}_r(s)ds - \norm[0]{F}_-\int_0^{+\infty}\dot{\beta}_r(s)ds\\
		&= \beta_r(0)(\norm[0]{F}_- + \norm[0]{F}_+)\\
		&= \beta_r(0)\norm[0]{F}\\
		&\leq \norm[0]{F},
	\end{align*}
    \noindent as $\mathscr{A}_0(v^+,\tau^+) = \mathscr{A}_0(v^-,\tau^-)$ since $C$ is connected. 
    
	\emph{Step 2: There exists $r_0 \in \mathbb{R}$ such that $r \leq r_0$ for all $(u,\tau,r) \in \mathscr{M}$.} Crucial is the existence of a constant $\delta > 0$ such that
	\begin{equation*}
		\norm[0]{\grad \mathscr{A}_r\vert_{(v,\tau,s)}} \geq \delta \qquad \forall (v,\tau,s) \in \Lambda_\varphi M \times \mathbb{R} \times \mathbb{R}.
	\end{equation*}
	This is proven along the lines of \cite[Lemma~3.9]{cieliebakfrauenfelder:rfh:2009}. With the above inequality and Step 1 we estimate
	\begin{equation*}
		\norm{F} \geq E(u,\tau) \geq \int_{-r}^r \norm[0]{\grad \mathscr{A}_r\vert_{(u(s),\tau(s),s)}}^2 ds \geq 2r\delta^2,
	\end{equation*}
	\noindent and thus we can set
	\begin{equation*}
		r_0 := \frac{\norm{F}}{2\delta^2}.
	\end{equation*}

\emph{Step 3: There exists $C > 0$ such that $\norm{\tau}_\infty \leq C$ for all $(u,\tau,r) \in \mathscr{M}$.} This is a delicate estimate based on the construction of the defining Hamiltonian $H$ for $\Sigma$ as well as an extension of the stablising form and proceeds as in \cite{cieliebakfrauenfelderpaternain:mane:2010}. Particularly crucial is \cite[Proposition~4.1]{cieliebakfrauenfelderpaternain:mane:2010}. 
	
\emph{Step 4: If \eqref{eq:assumption_moduli_space} holds, then $\mathscr{M}$ is compact.} Let $(u_k,\tau_k,r_k)$ be a sequence in the moduli space $\mathscr{M}$. By Step 2 and Step 3, the sequences $(r_k)$ and $(\tau_k)$ are uniformly bounded. Thus $(u_k,\tau_k,r_k)$ admits a $C^\infty_{\loc}$-convergent subsequence by standard arguments. Indeed, the uniform $L^\infty$-bound on the sequence $(u_k)$ follows from the assumption that $(M,\omega)$ is geometrically bounded and the uniform $L^\infty$-bound on the derivatives $(Du_k)$ follows from Corollary \ref{cor:bubbling_rfh} by the assumption that $(M,\omega)$ is symplectically aspherical. Denote the limit of this subsequence by $(u,\tau,r)$. This limit clearly satisfies the first equation in \eqref{eq:moduli_space}, thus one only needs to check the asymptotic conditions in \eqref{eq:moduli_space}. Again by compactness, $(u,\tau)$ converges to critical points $(w^\pm,\tau^\pm)$ of $\mathscr{A}_0$ at its asymptotic ends. We claim that
	\begin{equation}
		\label{eq:action_value}
	    \mathscr{A}_r(u(s),\tau(s),s) \in \intcc[0]{-\norm[0]{F} + \Omega(v^-), \norm[0]{F} + \Omega(v^-)} \qquad \forall s \in \mathbb{R}.
	\end{equation}
	In particular, $\Omega(w^\pm) \in \intcc[0]{-\norm[0]{F} + \Omega(v^-), \norm[0]{F} + \Omega(v^-)}$. So if \eqref{eq:action_value} holds, then by assumption \eqref{eq:assumption_moduli_space} we conclude $(w^\pm,\tau^\pm) \in C$ and $\mathscr{M}$ is indeed compact. It remains to prove \eqref{eq:action_value}. It is enough to prove
	\begin{equation*}
	    \mathscr{A}_r(u_k(s),\tau_k(s),s) \in \intcc[0]{-\norm[0]{F} + \Omega(v^-), \norm[0]{F} + \Omega(v^-)} \qquad \forall s \in \mathbb{R}
	\end{equation*}
	\noindent for every $k \in \mathbb{N}$. As in the proof of \cite[Lemma~2.8]{albersfrauenfelder:rfh:2010} we estimate 
	\allowdisplaybreaks
	\begin{align*}
	    0 &\leq \int^{+\infty}_{s_0} d\mathscr{A}_r(\partial_s (u_k,\tau_k),s) ds\\
		&= \int^{+\infty}_{s_0} \frac{d}{ds}\mathscr{A}_r(u_k,\tau_k,s) ds - \int^{+\infty}_{s_0} (\partial_s \mathscr{A}_r)(u_k,\tau_k,s)ds\\
		&= \lim_{s \to +\infty} \mathscr{A}_r(u_k,\tau_k,s) - \mathscr{A}_r(u_k(s_0),\tau_k(s_0),s_0) - \int^{+\infty}_{s_0} (\partial_s \mathscr{A}_r)(u_k,\tau_k,s)ds\\
		&= \mathscr{A}_0(v^+,\tau^+) - \mathscr{A}_r(u_k(s_0),\tau_k(s_0),s_0) + \int^{+\infty}_{s_0} \dot{\beta}_r(s)\int_0^1 F_t(u_k(s,t))dtds\\ 
		&\leq \mathscr{A}_0(v^+,\tau^+) - \mathscr{A}_r(u_k(s_0),\tau_k(s_0),s_0) + \int^{+\infty}_{-\infty} \norm[0]{\dot{\beta}_r(s)F}_+ds\\
		&\leq \mathscr{A}_0(v^+,\tau^+) - \mathscr{A}_r(u_k(s_0),\tau_k(s_0),s_0) + \norm[0]{F}\\
		&= \Omega(v^-) - \mathscr{A}_r(u_k(s_0),\tau_k(s_0),s_0) + \norm[0]{F}
	\end{align*}
	\noindent for all $s_0 \in \mathbb{R}$. Similarly, we compute
	\allowdisplaybreaks
	\begin{align*}
	    0 &\leq \int_{-\infty}^{s_0} d\mathscr{A}_r(\partial_s (u_k,\tau_k),s) ds\\
		&= \int_{-\infty}^{s_0} \frac{d}{ds}\mathscr{A}_r(u_k,\tau_k,s) ds - \int_{-\infty}^{s_0} (\partial_s \mathscr{A}_r)(u_k,\tau_k,s)ds\\
		&= \mathscr{A}_r(u_k(s_0),\tau_k(s_0),s_0) - \lim_{s \to -\infty} \mathscr{A}_r(u_k,\tau_k,s) - \int_{-\infty}^{s_0} (\partial_s \mathscr{A}_r)(u_k,\tau_k,s)ds\\
		&= \mathscr{A}_r(u_k(s_0),\tau_k(s_0),s_0) - \mathscr{A}_0(v^-,\tau^-) + \int_{-\infty}^{s_0} \dot{\beta}_r(s)\int_0^1 F_t(u_k(s,t))dtds
	\end{align*}
	\noindent and thus we estimate
	\begin{align*}
	    \mathscr{A}_r(u_k(s_0),\tau_k(s_0),s_0) &\geq \mathscr{A}_0(v^-,\tau^-) - \int_{-\infty}^{s_0} \dot{\beta}_r(s)\int_0^1 F_t(u_k(s,t))dtds\\
	    &\geq \mathscr{A}_0(v^-,\tau^-) - \int_{-\infty}^{+\infty} \norm[0]{\dot{\beta}(s)F}_+ds\\
	    &\geq \Omega(v^-) - \norm[0]{F}.
	\end{align*}
	This completes the proof of the Forcing Theorem \ref{thm:forcing_theorem}.
\end{proof}

We conlcude this section by applying the Forcing Theorem \ref{thm:forcing_theorem} to a displaceable twisted stable hypersurface.

\begin{example}
	Consider the isometry
	\begin{equation*}
		\varphi \colon \mathbb{T}^2 \times \mathbb{R}^2 \to \mathbb{T}^2 \times \mathbb{R}^2, \qquad \varphi(q_1,q_2) := (q_2,-q_1)
	\end{equation*}
	\noindent and its cotangent lift 
	\begin{equation*}
		D\varphi^\dagger \colon \mathbb{T}^2 \times \mathbb{R}^2 \to \mathbb{T}^2 \times \mathbb{R}^2, \quad D\varphi^\dagger(q_1,q_2,p_1,p_2) = (q_2,-q_1,p_2,-p_1).
	\end{equation*}
	Then $\Sigma_c \subseteq (T^*\mathbb{T}^2,\omega_\sigma,H)$ is a displaceable twisted stable hypersurface for the area form $\sigma = dq_1 \wedge dq_2$ by Example \ref{ex:twisted_stable_hypersurface}. By Example \ref{ex:twisted_characteristic}, we have that
	\begin{equation*}
		v \colon \mathbb{R} \to \Sigma_c, \qquad v(t) := \sqrt{2c}(\sin t, \cos t, \cos t, -\sin t)
	\end{equation*}
	\noindent is a $\tau$-periodic twisted characteristic for all periods $\tau \in 2\pi \mathbb{Z} + \frac{\pi}{2}$ and $c > 0$. Thus if we choose $v^- \in \mathscr{C}_\varphi(\Sigma_c)$ of period $\tau > 0$, then we compute for $v \in \mathscr{C}_\varphi(\Sigma_c)$ of period $\tau + 2\pi$
	\begin{equation*}
		\Omega(v) - \Omega(v^-) = c(\tau + 2\pi) - c\tau = 2\pi c = e(\Sigma_c) 
	\end{equation*}
	\noindent by Example \ref{ex:twisted_characteristic} and Example \ref{ex:displacement_energy}. Hence, we have verified the statement of the Forcing Theorem \ref{thm:forcing_theorem} for the displaceable twisted stable hypersurface $\Sigma_c$ in the symplectically aspherical and geometrically bounded symplectic manifold $(T^*\mathbb{T}^2,\omega_\sigma)$. Indeed, $(T^*\mathbb{T}^2,\omega_\sigma)$ is geometrically bounded by Example \ref{ex:geometrically_bounded}, and symplectically aspherical as
	\begin{equation*}
		\pi_2(T^*\mathbb{T}^n) \cong \pi_2(\mathbb{T}^n) \times \pi_2(\mathbb{R}^n) \cong 0.
	\end{equation*}
\end{example}

\newpage
\section{Further Steps in Twisted Rabinowitz--Floer Homology}
\label{sec:further_steps}
In this final chapter we discuss some possible further research in twisted Rabinowitz--Floer homology for future work. One can of course try to find a twisted version of every result provided by standard Rabinowitz--Floer homology. Following the survey article \cite{albersfrauenfelder:rfh:2012}, major results relate Rabinowitz--Floer homology to symplectic homology. 

In the first section we briefly outline a further computation of twisted Rabinowitz--Floer homology, where the hypersurface is not displaceable. 

In the second section, we explain an important physical setting where the Forcing Theorem \ref{thm:forcing_theorem} and Theorem \ref{thm:my_result} might be applicable. 

\subsection{Cotangent Bundles}
Let $(M,g)$ be a compact connected Riemannian manifold and let $(S^*M,pdq)$ be the spherisation of $M$ as in Example \ref{ex:cotangent_bundle}. By \cite{abbondandoloschwarz:rfh:2009} or \cite[Theorem~1.10]{cieliebakfrauenfelderoancea:rfh:2010}, we have 
\begin{equation*}
	\RFH_k(S^*M,T^*M) \cong \begin{cases}
		\operatorname{H}^{-k+1}(\mathscr{L}M) & k < 0,\\
		\operatorname{H}_k(\mathscr{L}M) & k > 1.
	\end{cases}
\end{equation*}
In the degrees $k = 0,1$ the answer is known and depends on the Euler class. The proof uses a relation between Rabinowitz--Floer homology and symplectic homology, respectively symplectic cohomology. If $\varphi \in \Aut(D^*M,pdq\vert_{D^*M})$ is a Liouville automorphism, then it is plausible to expect 
\begin{equation*}
	\RFH_k^\varphi(S^*M,T^*M) \cong \begin{cases}
		\operatorname{H}^{-k+1}(\mathscr{L}_\varphi M) & k < 0,\\
		\operatorname{H}_k(\mathscr{L}_\varphi M) & k > 1.
	\end{cases}
\end{equation*}
However, it would be also interesting to study the loop space homology $\operatorname{H}_\ast(\mathscr{L}_\varphi M)$ itself, because usually one computes the free loop space homology via Morse theory. For details see \cite[Chapter~2]{loop_spaces:2015}.

\subsection{Stark--Zeeman Systems}
Following \cite{cieliebakfrauenfelderzhao:J+:2019} we introduce Stark--Zeeman systems. For $\mu_\pm > 0$ define the potential functions
\begin{equation*}
	V_\pm \colon \mathbb{C} \setminus \{\pm 1\} \to \mathbb{R}, \qquad V_\pm(z) := -\frac{\mu_\pm}{\abs[0]{z \pm 1}}.
\end{equation*}
Let $U_0 \subseteq \mathbb{C}$ be open and star-shaped with respect to the origin such that $\pm 1 \in U_0$. Choose $V_0 \in C^\infty(U_0)$. Moreover, set
\begin{equation*}
	V := V_+ + V_- + V_0 \in C^\infty(U)
\end{equation*}
\noindent for $U := U_0 \setminus \{\pm 1\}$. For a function $B \in C^\infty(U_0)$, let $\sigma_B := B dq_1 \wedge dq_2$ and abbreviate by $(T^*U,\omega_B)$ the associated magnetic cotangent bundle. Fix a Riemannian metric $g$ on $U_0$ which is conformal to the standard metric. A \bld{Stark--Zeeman system} is the magnetic Hamiltonian system $(T^*U,\omega_B,H)$ with 
\begin{equation*}
	H(q,p) := \frac{1}{2}\norm[0]{p}^2_{g^*} + V(q) \qquad \forall (q,p) \in T^*U.
\end{equation*}
For $c \in \mathbb{R}$ a regular value of $H$, we consider a connected component $\Sigma_c \subseteq H^{-1}(c)$ such that $\mathfrak{K}_c \cup \{\pm 1\}$ is bounded and simply connected, where the Hill's region $\mathfrak{K}_c$ is defined by 
\begin{equation*}
	\mathfrak{K}_c := \pi(\Sigma_c) \subseteq \{q \in U : V(q) \leq c\}.
\end{equation*}
For example, the planar circular restricted three-body problem is a Stark--Zeeman system. In order to deal with collisions, we regularise $\Sigma_c$.

\begin{definition}[{Regularisation, \cite[p.~48]{frauenfelderkoert:3bp:2018}}]
	Let $(\Sigma,\omega)$ be a noncompact Hamiltonian manifold. A \bld{regularisation of $(\Sigma,\omega)$} is defined to be a compact Hamiltonian manifold $(\overline{\Sigma},\overline{\omega})$ such that there exists an embedding $\iota \colon \Sigma \hookrightarrow \overline{\Sigma}$ with $\iota^*\overline{\omega} = \omega$.
\end{definition}

\noindent Consider the \bld{Birkhoff regularisation map}
\begin{equation*}
	 \varphi \colon \mathbb{C}^* \to \mathbb{C}, \qquad B(z) := \frac{1}{2}\del[3]{z + \frac{1}{z}}.
\end{equation*}
By Example \ref{ex:holomorphic}, the cotangent lift $D\varphi^\dagger$ of $\varphi$ is given by
\begin{equation*}
	D\varphi^\dagger \colon T^*\mathbb{C}^* \to T^*\mathbb{C}, \qquad D\varphi^\dagger(z,w) = \del[3]{\frac{z^2 + 1}{2z},\frac{2\overline{z}^2w}{\overline{z}^2 - 1}} = (q,p).
\end{equation*} 
The regular energy surface $\Sigma_c$ gives rise to a regular energy surface $\Sigma^B_c \subseteq K^{-1}(0)$, where the rescaled Hamiltonian $K := H \circ D\varphi^\dagger$ is given by
\begin{equation*}
	K(z,w) = \frac{\norm[0]{w}^2_{g^*}}{2} - \frac{\mu_+\abs[0]{z + 1}^2}{2\abs[0]{z}^3} - \frac{\mu_-\abs[0]{z - 1}^2}{2\abs[0]{z}^3} + \frac{(V_0(q) - c)\abs[0]{z^2 - 1}^2}{4\abs[0]{z}^4}.
\end{equation*}
The compact regular energy surface $\Sigma^B_c$ is called the \bld{Birkhoff regularisaton of $\Sigma_c$}. This regularisation is invariant under the cotangent lift
\begin{equation*}
	\Phi \colon T^*\mathbb{C}^* \to T^*\mathbb{C}^*, \qquad \Phi(z,w) := \del[3]{\frac{1}{z},-\overline{z}^2 w}
\end{equation*}
As the induced action of $\Phi$ on $\Sigma_c^B$ is free, we obtain the cover 
\begin{equation*}
	\Sigma^B_c \to \Sigma^B_c/\mathbb{Z}_2.
\end{equation*}
Explicitly, there exist diffeomorphisms such that
\begin{equation*}
	\Sigma^B_c \cong \mathbb{S}^1 \times \mathbb{S}^2 \qquad \text{and} \qquad \Sigma^B_c/\mathbb{Z}_2 \cong \mathbb{RP}^3 \# \mathbb{RP}^3.
\end{equation*}
Therefore, it may be possible to apply the ideas developed in the proof of Theorem \ref{thm:my_result} or the Forcing Theorem \ref{thm:forcing_theorem} to these hypersurfaces. The analysis of these hypersurfaces is already quite delicate in the special case of the planar circular restricted three-body problem. Indeed, it requires some work to show that the regularised energy hypersurface is fibrewise star-shaped for energy values below the first critical value. For details, see \cite[Theorem~5.7.2]{frauenfelderkoert:3bp:2018}. Hence we cannot expect stability of the hypersurfaces in a general Stark--Zeeman system.

\section*{Acknowledgements}
First of all I would like to thank my supervisor Urs Frauenfelder for his inspiring guidance. I owe thanks to Kai Cieliebak and Igor Uljarevic for many helpful discussions. Moreover, I cordially thank Felix Schlenk for the pleasant stay at the University of Neuch\^{a}tel in April 2022. I also thank Will J. Merry for his constant encouragement during my Master's thesis as well as the doctoral students of the Mathematics institute at the university of Augsburg for enduring my numerous talks. Lastly, I thank Springer National for providing a freely available template for monographs. Without it, this thesis would not be what it is. This thesis is dedicated to my family, most importantly, my two loves Jil and Neil Kern. 

\appendix

\section{Twisted Loops in Universal Covering Manifolds}
\label{sec:twisted_loops_on_universal_covering_manifolds}

In this Appendix, we will consider the category of topological manifolds rather than the category of smooth manifolds, because smoothness does not add much to the discussion. Free and based loop spaces are fundamental objects in Algebraic Topology, for a vast treatment of the geometry and topology of based as well as free loop spaces see for example \cite{loop_spaces:2015}. But so-called twisted loop spaces are not considered that much. 

\begin{theorem}[Twisted Loops in Universal Covering Manifolds]
	Let $(M,x)$ be a connected pointed topological manifold and $\pi \colon \tilde{M} \to M$ the universal covering. 
	\begin{enumerate}[label=\textup{(\alph*)}]
		\item Fix $[\eta] \in \pi_1(M,x)$ and denote by $U_\eta \subseteq \mathscr{L}(M,x)$ the path component corresponding to $[\eta]$ via the bijection $\pi_0(\mathscr{L}(M,x)) \cong \pi_1(M,x)$. For every $e,e' \in \pi^{-1}(x)$ and $\varphi \in \Aut_\pi(\tilde{M})$ such that $\varphi(e) = \tilde{\eta}_e(1)$, where $\tilde{\eta}_e$ denotes the unique lift of $\eta$ with $\tilde{\eta}_e(0) = e$, we have a commutative diagram of homeomorphisms
		\begin{equation}
			\label{cd:twisted}
			\qquad \begin{tikzcd}
				\mathscr{L}_\varphi(\tilde{M},e) \arrow[rr,"L_\psi"] & & \mathscr{L}_{\psi \circ \varphi \circ \psi^{-1}}(\tilde{M},e')\\
				& U_\eta \arrow[lu,"\Psi_e"] \arrow[ru,"\Psi_{e'}"'],
			\end{tikzcd}
		\end{equation}
		\noindent where $\psi \in \Aut_\pi(\tilde{M})$ is such that $\psi(e) = e'$, 
		\begin{equation*}
			\qquad L_\psi \colon \mathscr{L}_\varphi(\tilde{M},e) \to \mathscr{L}_{\psi \circ \varphi \circ \psi^{-1}}(\tilde{M},e'), \qquad L_\psi(\gamma) := \psi \circ \gamma,
		\end{equation*}
		\noindent and
		\begin{align*}
			\qquad& \Psi_e \colon U_\eta \to \mathscr{L}_\varphi(\tilde{M},e), & \Psi_e(\gamma) := \tilde{\gamma}_e,\\
			\qquad& \Psi_{e'} \colon U_\eta \to \mathscr{L}_{\psi \circ \varphi \circ \psi^{-1}}(\tilde{M},e'), & \Psi_{e'}(\gamma) := \tilde{\gamma}_{e'}.
		\end{align*}
		Moreover, $U_{c_x} \cong \mathscr{L}_\varphi(\tilde{M},e)$ via $\Psi_e$ if and only if $\varphi = \id_{\tilde{M}}$, where $c_x$ denotes the constant loop at $x$.
	\item For every $\varphi \in \Aut_\pi(\tilde{M})$ and $e,e' \in \pi^{-1}(x)$ we have a commutative diagram of isomorphisms
	\begin{equation*}
		\qquad \begin{tikzcd}
		\Aut_\pi(\tilde{M}) \arrow[rr,"C_\psi"] & & \Aut_\pi(\tilde{M})\\
		& \pi_1(M,x) \arrow[lu,"\Phi_e"] \arrow[ru,"\Phi_{e'}"'],
	\end{tikzcd}
	\end{equation*}
	\noindent where for $\psi \in \Aut_\pi(\tilde{M})$ sucht that $\psi(e) = e'$
	\begin{equation*}
		\qquad C_\psi \colon \Aut_\pi(\tilde{M}) \to \Aut_\pi(\tilde{M}), \qquad C_\psi(\varphi) := \psi \circ \varphi \circ \psi^{-1},
	\end{equation*}
	\noindent and
	\begin{align*}
		\qquad & \Phi_e \colon \pi_1(M,x) \to \Aut_\pi(\tilde{M}), & \Phi_e([\gamma]) := \varphi^e_{[\gamma]},\\
		\qquad & \Phi_{e'} \colon \pi_1(M,x) \to \Aut_\pi(\tilde{M}), & \Phi_{e'}([\gamma]) := \varphi^{e'}_{[\gamma]},
	\end{align*}
	\noindent with $\varphi^e_{[\gamma]}(e) = \tilde{\gamma}_e(1)$ and $\varphi^{e'}_{[\gamma]}(e') = \tilde{\gamma}_{e'}(1)$.
\item The projection
	\begin{equation*}
		\qquad \tilde{\pi}_x \colon \coprod_{\substack{\varphi \in \Aut_\pi(\tilde{M})\\e \in \pi^{-1}(x)}} \mathscr{L}_\varphi(\tilde{M},e) \to \mathscr{L}(M,x)
	\end{equation*}
	\noindent defined by $\tilde{\pi}_x(\gamma) := \pi \circ \gamma$ is a covering map with number of sheets coinciding with the cardinality of $\pi_1(M,x)$. Moreover, $\tilde{\pi}_x$ restricts to define a covering map
	\begin{equation*}
		\qquad \tilde{\pi}_x\vert_{\id_{\tilde{M}}} \colon \coprod_{e \in \pi^{-1}(x)} \mathscr{L}(\tilde{M},e) \to U_{c_x},
	\end{equation*}
	\noindent and $\tilde{\pi}_x$ gives rise to a principal $\Aut_\pi(\tilde{M})$-bundle. If $M$ admits a smooth structure, then this bundle is additionally a bundle of smooth Banach manifolds.
\end{enumerate}
	\label{thm:cd_twisted}
\end{theorem} 

\begin{proof}
	For proving part (a), fix a path class $[\gamma] \in \pi_1(M,x)$. As any topological manifold is Hausdorff, paracompact and locally metrisable by definition, the Smirnov Metrisation Theorem \cite[Theorem~42.1]{munkres:topology:2000} implies that $M$ is metrisable. Let $d$ be a metric on $M$ and $\bar{d}$ be the standard bounded metric corresponding to $d$, that is,
	\begin{equation*}
		\bar{d}(x,y) = \min\cbr[0]{d(x,y),1} \qquad \forall x,y \in M.
	\end{equation*}
	The metric $\bar{d}$ induces the same topology on $M$ as $d$ by \cite[Theorem~20.1]{munkres:topology:2000}. Topologise the based loop space $\mathscr{L}(M,x) \subseteq \mathscr{L}M$ as a subspace of the free loop space on $M$, where $\mathscr{L}M$ is equipped with the topology of uniform convergence, that is, with the supremum metric
	\begin{equation*}
		\bar{d}_\infty(\gamma,\gamma') = \sup_{t \in \mathbb{S}^1}\bar{d}\del[1]{\gamma(t),\gamma'(t)} \qquad \forall \gamma,\gamma' \in \mathscr{L}M.
	\end{equation*}
	There is a canonical pseudometric on the universal covering manifold $\tilde{M}$ induced by $\bar{d}$ given by $\bar{d} \circ \pi$. As every pseudometric generates a topology, we topologise the based twisted loop space $\mathscr{L}_\varphi(\tilde{M},e) \subseteq \mathscr{P}\tilde{M}$ as a subspace of the free path space on $\tilde{M}$ for every $e \in \pi^{-1}(x)$ via the supremum metric $\tilde{d}_\infty$ corresponding to $\bar{d} \circ \pi$. In fact, $\tilde{d}_\infty$ is a metric as if $\tilde{d}_\infty(\gamma,\gamma') = 0$, then by definition of $\tilde{d}_\infty$ we have that $\pi(\gamma) = \pi(\gamma')$. But as $\gamma(0) = e = \gamma'(0)$, we conclude $\gamma = \gamma'$ by the unique lifting property of paths \cite[Corollary~11.14]{lee:tm:2011}. Note that the resulting topology of uniform convergence on $\mathscr{L}_\varphi(\tilde{M},e)$ coincides with the compact-open topology by \cite[Theorem~46.8]{munkres:topology:2000} or \cite[Proposition~A.13]{hatcher:at:2001}. In particular, the topology of uniform convergence does not depend on the choice of a metric (see \cite[Corollary~46.9]{munkres:topology:2000}). It follows from \cite[Theorem~11.15~(b)]{lee:tm:2011}, that $\Psi_e$ and $\Psi_{e'}$ are well-defined. Moreover, it is immediate by the fact that the projection $\pi \colon \tilde{M} \to M$ is an isometry with respect to the above metric, that $\Psi_e$ and $\Psi_{e'}$ are continuous with continuous inverse given by the composition with $\pi$. It is also immediate that $L_\psi$ is continuous with continuous inverse $L_{\psi^{-1}}$.

	Next we show that the diagram \eqref{cd:twisted} commutes. Note that
	\begin{equation*}
		\pi \circ L_\psi \circ \Psi_e = \pi \circ \Psi_e = \id_{U_\eta} = \pi \circ \Psi_{e'},
	\end{equation*}
	\noindent thus by
	\begin{equation*}
		(L_\psi \circ \Psi_e(\gamma))(0) = \psi(\tilde{\gamma}_e(0)) = \psi(e) = e' = \tilde{\gamma}_{e'}(0) = \Psi_{e'}(\gamma)(0)
	\end{equation*}
	\noindent and by uniqueness it follows that 
	\begin{equation*}
		L_\psi \circ \Psi_e = \Psi_{e'}.
	\end{equation*}
	In particular
	\begin{equation*}
		\Psi_{e'}(1) = (L_\psi \circ \Psi_e)(1) = \psi(\varphi(e)) = (\psi \circ \varphi \circ \psi^{-1})(e'),
	\end{equation*}
	\noindent and thus $\Psi_{e'}(\gamma) \in \mathscr{L}_{\psi \circ \varphi \circ \psi^{-1}}(\tilde{M},e')$. Consequently, the homeomorphism $\Psi_{e'}$ is well-defined.

	Recall, that by the Monodromy Theorem \cite[Theorem~11.15~(b)]{lee:tm:2011} 
	\begin{equation*}
		\gamma \simeq \gamma' \qquad \Leftrightarrow \qquad \Psi_e(\gamma)(1) = \Psi_e(\gamma')(1)
	\end{equation*}
	\noindent for all paths $\gamma$ and $\gamma'$ in $M$ starting at $x$ and ending at the same point. Note that the statement of the the Monodromy Theorem is an if-and-only-if statement since $\tilde{M}$ is simply connected.

	Suppose $\gamma \in \mathscr{L}(M,x)$ is contractible. Then $\gamma \simeq c_x$, implying $e \in \Fix(\varphi)$. But the only deck transformation of $\pi$ fixing any point of $\tilde{M}$ is $\id_{\tilde{M}}$ by \cite[Proposition~12.1~(a)]{lee:tm:2011}.

	Conversely, assume that $\gamma \in \mathscr{L}(M,x)$ is not contractible. Then we have that $\Psi_e(\gamma)(1) \neq e$. Indeed, if $\Psi_e(\gamma)(1) = e$, then $\gamma \simeq c_x$ and consequently, $\gamma$ would be contractible. As normal covering maps have transitive automorphism groups by \cite[Corollary~12.5]{lee:tm:2011}, there exists $\psi \in \Aut_\pi(\tilde{M}) \setminus \cbr[0]{\id_{\tilde{M}}}$ such that $\Psi_e(\gamma)(1) = \psi(e)$.  

	For proving part (b), observe that $\Phi_e$ and $\Phi_{e'}$ are isomorphisms follows from \cite[Corollary~12.9]{lee:tm:2011}. Moreover, it is also clear that $C_\psi$ is an isomorphism with inverse $C_{\psi^{-1}}$. Let $[\gamma] \in \pi_1(M,x)$. Then using part (a) we compute
\allowdisplaybreaks
	\begin{align*}
	(C_\psi \circ \Phi_e)[\gamma](e') &= (\psi \circ \Phi_e[\gamma] \circ \psi^{-1})(e')\\
	&= \psi\del[1]{\varphi_{[\gamma]}^e(e)}\\
	&= \psi(\tilde{\gamma}_e(1))\\
	&= (L_\psi \circ \Psi_e)(\gamma)(1)\\
	&= \Psi_{e'}(\gamma)(1)\\
	&= \tilde{\gamma}_{e'}(1)\\
	&= \varphi^{e'}_{[\gamma]}(e')\\
	&= \Phi_{e'}[\gamma](e').
\end{align*}
Thus by uniqueness \cite[Proposition~12.1~(a)]{lee:tm:2011}, we conclude
\begin{equation*}
	C_\psi \circ \Phi_e = \Phi_{e'}.
\end{equation*}
Finally for proving (c), define a metric $\tilde{d}_\infty$ on 
	\begin{equation*}
		E := \coprod_{\substack{\varphi \in \Aut_\pi(\tilde{M})\\e \in \pi^{-1}(x)}} \mathscr{L}_\varphi(\tilde{M},e)
	\end{equation*}
	\noindent by
	\begin{equation*}
		\tilde{d}_\infty(\gamma,\gamma') := \begin{cases}
			\bar{d}_\infty\del[1]{\pi(\gamma),\pi(\gamma')} & \gamma,\gamma' \in \mathscr{L}_\varphi(\tilde{M},e),\\
			1 & \text{else}.
		\end{cases}
	\end{equation*}
	Then the induced topology coincides with the disjoint union topology and with respect to this topology, $\tilde{\pi}_x$ is continuous. So left to show is that $\tilde{\pi}_x$ is a covering map. Surjectivity is clear. So let $\gamma \in \mathscr{L}(M,x)$. Then $\gamma \in U_\eta$ for some $[\eta] \in \pi_1(M,x)$. Now note that $U_\eta$ is open in $\mathscr{L}(M,x)$ and by part (a) we conclude
	\begin{equation}
		\label{eq:twisted_fibre}
		\tilde{\pi}_x^{-1}(U_\eta) = \coprod_{\psi \in \Aut_\pi(\tilde{M})} \mathscr{L}_{\psi \circ \varphi \circ \psi^{-1}}(\tilde{M},\psi(e))
	\end{equation}
	\noindent for some fixed $e \in \pi^{-1}(x)$ and $\varphi \in \Aut_\pi(\tilde{M})$ such that $\varphi(e) = \tilde{\eta}_e(1)$. 


	As the cardinality of the fibre $\pi^{-1}(x)$ and of $\Aut_\pi(\tilde{M})$ coincides with the cardinality of the fundamental group $\pi_1(M,x)$ by \cite[Corollary~11.31]{lee:tm:2011} and part (b), we conclude that the number of sheets is equal to the cardinality of the fundamental group $\pi_1(M,x)$.

  Equip $\Aut_\pi(\tilde{M})$ with the discrete topology. As the fundamental group of every topological manifold is countable by \cite[Theorem~7.21]{lee:tm:2011}, we have that $\Aut_\pi(\tilde{M})$ is a discrete topological Lie group. Now label the distinct path classes in $\pi_1(M,x)$ by $\beta \in B$ and for fixed $e \in \pi^{-1}(x)$ define local trivialisations
 \begin{equation*}
	 (\tilde{\pi}_x,\alpha_\beta) \colon \tilde{\pi}_x^{-1}(U_\beta)  \xrightarrow{\cong} U_\beta \times \Aut_\pi(\tilde{M}),
 \end{equation*} 
 \noindent making use of \eqref{eq:twisted_fibre} by
 \begin{equation*}
	 \alpha_\beta(\gamma) := \psi^{-1},
 \end{equation*}
 \noindent whenever $\gamma \in \mathscr{L}_{\psi \circ \varphi \circ \psi^{-1}}(\tilde{M},\psi(e))$. Consequently, $\tilde{\pi}_x$ is a fibre bundle with discrete fibre $\Aut_\pi(\tilde{M})$ and bundle atlas $(U_\beta,\alpha_\beta)_{\beta \in B}$. Define a free right action
 \begin{equation*}
	 E \times \Aut_\pi(\tilde{M}) \to E, \qquad \gamma \cdot \xi := \xi^{-1} \circ \gamma.
 \end{equation*}
  Then $\alpha_\beta$ is $\Aut_\pi(\tilde{M})$-equivariant with respect to this action for all $\beta \in B$. Indeed, using again the commutative diagram \eqref{cd:twisted} we compute
 \begin{equation*}
	 \alpha_\beta(\gamma \cdot \xi) = \alpha_\beta(\xi^{-1} \circ \gamma) = \del[1]{\xi^{-1} \circ \psi}^{-1} = \psi^{-1} \circ \xi = \alpha_\beta(\gamma) \circ \xi
 \end{equation*}
 \noindent for all $\xi \in \Aut_\pi(\tilde{M})$ and $\gamma \in \mathscr{L}_{\psi \circ \varphi \circ \psi^{-1}}(\tilde{M},\psi(e))$. Note, that here we use again the fact that $\Aut_\pi(\tilde{M})$ acts transitively on the fibre $\pi^{-1}(x)$.

 Suppose that $M$ admits a smooth structure. Then for every compact smooth manifold $N$ we have that the mapping space $C(N,M)$ admits the structure of a smooth Banach manifold by \cite{wittmann:loop_space:2019}. By \cite[Theorem~1.1~p.~24]{loop_spaces:2015}, there is a smooth fibre bundle, called the \bld{loop-loop fibre bundle},
 \begin{equation*}
		 \mathscr{L}(M,x) \hookrightarrow \mathscr{L}M \xrightarrow{\ev_0} M 
 \end{equation*}
 \noindent where
 \begin{equation*}
	 \ev_0 \colon \mathscr{L} M \to M, \qquad \ev_0(\gamma) := \gamma(0).
 \end{equation*}
 Thus the based loop space $\mathscr{L}(M,x) = \ev_0^{-1}(x)$ on $M$ is a smooth Banach manifold by the implicit function theorem \cite[Theorem~A.3.3]{mcduffsalamon:J-holomorphic_curves:2012} for all $x \in M$. Likewise, by \cite[Theorem~1.2~p.~25]{loop_spaces:2015}, there is a smooth fibre bundle, called the \bld{path-loop fibre bundle}, 
  \begin{equation*}
	  \mathscr{L}(\tilde{M},e) \hookrightarrow \mathscr{P}(\tilde{M},e) \xrightarrow{\ev_1} \tilde{M}, 
 \end{equation*}
 \noindent where
 \begin{equation*}
	 \mathscr{P}(\tilde{M},e) := \{\gamma \in C(I,\tilde{M}) : \gamma(0) = e\}
 \end{equation*}
 \noindent denotes the based path space and
 \begin{equation*}
	 \ev_1 \colon \mathscr{P}(\tilde{M},e) \to \tilde{M}, \qquad \ev_1(\gamma) := \gamma(1).
 \end{equation*}
 Therefore, the twisted loop space $\mathscr{L}_\varphi(\tilde{M},e) = \ev_1^{-1}(\varphi(e))$ is also a smooth Banach manifold for all $\varphi \in \Aut_\pi(\tilde{M})$ and $e \in \pi^{-1}(x)$ by the implicit function theorem \cite[Theorem~A.3.3]{mcduffsalamon:J-holomorphic_curves:2012}. As the fundamental group $\pi_1(M,x)$ is countable, the topological space $E$ has only countably many connected components being smooth Banach manifolds and thus the total space itself is a smooth Banach manifold. Finally, $\Aut_\pi(\tilde{M})$ is trivially a Banach manifold with $\dim \Aut_\pi(\tilde{M}) = 0$ as a discrete Lie group.
\end{proof}

\begin{corollary}
	\label{cor:cd_twisted_abelian}
	Let $(M,x)$ be a connected pointed topological manifold and denote by $\pi \colon \tilde{M} \to M$ the universal covering of $M$. Assume that $\pi_1(M,x)$ is abelian.
	\begin{enumerate}[label=\textup{(\alph*)}]
		\item Fix a path class $[\eta] \in \pi_1(M,x)$. For every $e,e' \in \pi^{-1}(x)$ and deck transformation $\varphi \in \Aut_\pi(\tilde{M})$ such that $\varphi(e) = \tilde{\eta}_e(1)$, we have a commutative diagram of homeomorphisms
		\begin{equation*}
			\qquad \begin{tikzcd}
				\mathscr{L}_\varphi(\tilde{M},e) \arrow[rr,"L_\psi"] & & \mathscr{L}_\varphi(\tilde{M},e')\\
				& U_\eta \arrow[lu,"\Psi_e"] \arrow[ru,"\Psi_{e'}"'],
			\end{tikzcd}
		\end{equation*}
		\noindent where $\psi \in \Aut_\pi(\tilde{M})$ is such that $\psi(e) = e'$.
	\item For every $\varphi \in \Aut_\pi(\tilde{M})$ we have that $\Phi_e = \Phi_{e'}$ for all $e,e' \in \pi^{-1}(x)$. 
	\end{enumerate}
\end{corollary}

Lemma \ref{lem:noncontractible} now follows from part (a) of Theorem \ref{thm:cd_twisted}. Indeed, by assumption $\varphi \in \Aut_\pi(\Sigma) \setminus \{\id_\Sigma\}$ and using the long exact sequence of homotopy groups of a fibration \cite[Theorem~4.41]{hatcher:at:2001}, there is a short exact sequence 
\begin{equation*}
	\begin{tikzcd}[column sep = scriptsize]
		0 \arrow[r] & \pi_1(\Sigma,x) \arrow[r] & \pi_1(\Sigma/\mathbb{Z}_m,\pi(x)) \arrow[r] & \pi_0(\mathbb{Z}_m) \arrow[r] & 0.
	\end{tikzcd}
\end{equation*}
In particular, by \cite[Corollary~12.9]{lee:tm:2011} we conclude
\begin{equation*}
	\Aut_\pi(\Sigma) \cong \pi_1(\Sigma/\mathbb{Z}_m,\pi(x)) \cong \mathbb{Z}_m \cong \{\id_\Sigma,\varphi,\dots,\varphi^{m - 1}\}.
\end{equation*}

Finally, we discuss a smooth structure on the continuous free twisted loop space of a smooth manifold.


\begin{lemma}
	\label{lem:free_twisted_loop_space}
	Let $M$ be a smooth manifold and $\varphi \in \Diff(M)$. Then the continuous free twisted loop space $\mathscr{L}_\varphi M$ is the pullback of
	\begin{equation*}
		(\ev_0,\ev_1) \colon \mathscr{P}M \to M \times M, \qquad \gamma \mapsto (\gamma(0),\gamma(1)),
	\end{equation*}
	\noindent where we abbreviate $\mathscr{P}M := C(I,M)$, along the graph of $\varphi$
	\begin{equation*}
		\Gamma_\varphi \colon M \to M \times M, \qquad \Gamma_\varphi(x) := (x,\varphi(x)),
	\end{equation*}
	\noindent in the category of smooth Banach manifolds. Moreover, we have that
	\begin{equation*}
		T_\gamma \mathscr{L}_\varphi M = \{X \in \Gamma^0(\gamma^*TM) : X(1) = D\varphi(X(0))\}
	\end{equation*}
	\noindent for all $\gamma \in \mathscr{L}_\varphi M$.
\end{lemma}

\begin{proof}
	Write $f := (\ev_0,\ev_1)$. Then
	\begin{equation*}
		\mathscr{L}_\varphi M = f^{-1}(\Gamma_\varphi(M)).
	\end{equation*}
	Thus in order to show that the free twisted loop space $\mathscr{L}_\varphi M$ is a smooth Banach manifold, it is enough to show that $f$ is transverse to the properly embedded smooth submanifold $\Gamma_\varphi(M) \subseteq M \times M$. By \cite[Proposition~2.4]{lang:dg:1999} we need to show that the composition
	\begin{equation*}
		\Phi_\gamma\colon T_\gamma\mathscr{P}M \xrightarrow{Df_\gamma} T_{(x,\varphi(x))}(M \times M) \to T_{(x,\varphi(x))}(M \times M)/T_{(x,\varphi(x))}\Gamma_\varphi(M)
	\end{equation*}
	\noindent is surjective and $\ker \Phi_\gamma$ is complemented for all $\gamma \in f^{-1}(\Gamma_\varphi(M))$, where we abbreviate $x := \gamma(0)$. Note that we have a canonical isomorphism
	\begin{equation*}
		T_{(x,\varphi(x))}(M \times M)/T_{(x,\varphi(x))}\Gamma_\varphi(M) \to T_{\varphi(x)}M, \quad [(v,u)] := u - D\varphi(v).
	\end{equation*}
	Under this canonical isomorphism, the linear map $\Phi_\gamma$ is given by
	\begin{equation*}
		\Phi_\gamma(X) = X(1) - D\varphi(X(0)), \qquad \forall X \in \Gamma^0(\gamma^*TM).
	\end{equation*}
	Fix a Riemannian metric on $M$ and let $X_v \in \Gamma(\gamma^*TM)$ be the unique parallel vector field with $X_v(1) = v \in T_{\varphi(x)}M$. Fix a cutoff function $\beta \in C^\infty(I)$ such that $\supp \beta \subseteq \intcc[1]{\frac{1}{2},1}$ and $\beta = 1$ in a neighbourhood of $1$. Then $\Phi_\gamma(\beta X_v) = v$ and consequently, $\Phi_\gamma$ is surjective. Moreover
	\begin{equation*}
		\ker \Phi_\gamma = \cbr[0]{X \in \Gamma^0(\gamma^*TM) : X(1) = D\varphi(X(0))}
	\end{equation*}
	\noindent is complemented by the finite-dimensional vector space 
	\begin{equation*}
		V := \cbr[0]{\beta X_v \in \Gamma(\gamma^*TM) : v \in T_{\varphi(x)}M}.
	\end{equation*}
	Indeed, any $X \in \Gamma^0(\gamma^*TM)$ can be decomposed uniquely as
	\begin{equation*}
		X = X - \beta X_v + \beta X_v, \qquad v := X(1) - D\varphi(X(0)).
	\end{equation*}
	Abbreviating $Y := X - \beta X_v \in \Gamma^0(\gamma^*TM)$, we have that
	\begin{equation*}
		Y(1) = D\varphi(X(0)) = D\varphi(Y(0)),
	\end{equation*}
	\noindent implying $Y \in \ker \Phi_\gamma$. Thus $\mathscr{L}_\varphi M$ is a smooth Banach manifold.

	Now note that $\mathscr{L}_\varphi M$ can be identified with the pullback
	\begin{equation*}
		f^*\mathscr{P}M = \{(x,\gamma) \in M \times \mathscr{P}M : (\gamma(0),\gamma(1)) = \del[0]{x,\varphi(x)}\},
	\end{equation*}
	\noindent making the diagram 
	\begin{equation*}
		\begin{tikzcd}
			f^* \mathscr{P}M \arrow[r,"\pr_2"]\arrow[d,"\pr_1"'] & \mathscr{P}M \arrow[d,"f"]\\
			M \arrow[r,"\Gamma_\varphi"'] & M \times M
		\end{tikzcd}
	\end{equation*}
	\noindent commute, via the homeomorphism
	\begin{equation*}
		\mathscr{L}_\varphi M \to f^*\mathscr{P}M, \qquad \gamma \mapsto (\gamma(0),\gamma).
	\end{equation*}
	Finally, one computes
	\begin{equation*}
		T_{(x,\gamma)}f^*\mathscr{P}M = \{(v,X) \in T_xM \times T_\gamma \mathscr{P}M : Df_\gamma X = D\Gamma_\varphi\vert_x(v)\}
	\end{equation*}
	\noindent for all $(x,\gamma) \in f^*\mathscr{P}M$.
\end{proof}

\begin{remark}
	Using Lemma \ref{lem:free_twisted_loop_space} one should be able to prove similar results as in Theorem \ref{thm:cd_twisted} in the case of free twisted loop spaces. However, in the non-abelian case the situation gets much more complicated as in general it is not true, that lifts of conjugated elements of the fundamental group lie in the same free twisted loop space by \cite[Theorem~1.6~(i)]{loop_spaces:2015}.
\end{remark}

\newpage
\section{On the Nonexistence of the Gradient of the Twisted Rabinowitz Action Functional}
Let $(M,\lambda)$ be an exact symplectic manifold and $\varphi \in \Symp(M,d\lambda)$ a symplectomorphism of finite order. For $H \in C^\infty(M)$ such that $H \circ \varphi = H$, one can define the twisted Rabinowitz action functional
\begin{equation}
	\label{eq:twisted_Rabinowitz_action_functional}
	\mathscr{A}^H_\varphi \colon \mathscr{L}_\varphi M \times \mathbb{R} \to \mathbb{R}, \qquad \mathscr{A}^H_\varphi(\gamma,\tau) := \int_0^1 \gamma^*\lambda - \tau \int_0^1 H(\gamma(t))dt.
\end{equation}
Let $J$ be a $d\lambda$-compatible almost complex structure such that $\varphi^*J = J$. Then one can consider the gradient of $\mathscr{A}^H_\varphi$ with respect to the $L^2$-metric
\begin{equation}
	\langle (X,\eta), (Y,\sigma) \rangle_J := \int_0^1 d\lambda(JX(t),Y(t))dt + \eta\sigma
	\label{eq:L^2-metric}
\end{equation}
\noindent for all $(X,\eta), (Y,\sigma) \in T_\gamma\mathscr{L}_\varphi M \times \mathbb{R}$ and $(\gamma,\tau) \in \mathscr{L}_\varphi M \times \mathbb{R}$. 

\begin{theorem}[Nonexistence Gradient]
	\label{thm:nonexistence_gradient}
	Let $(M,\lambda)$ be a connected exact symplectic manifold, $\varphi \in \Symp(M,d\lambda)$ of finite order and $H \in C^\infty(M)$ such that $H \circ \varphi = H$. If $\varphi^*\lambda \neq \lambda$, then the gradient of the twisted Rabinowitz action functional \eqref{eq:twisted_Rabinowitz_action_functional} with respect to the $L^2$-metric \eqref{eq:L^2-metric} does not exist.
\end{theorem}

\begin{proof}
	Assume that the gradient $\grad \mathscr{A}^H_\varphi \in \mathfrak{X}(\mathscr{L}_\varphi M \times \mathbb{R})$ exists. We write
\begin{equation*}
	\grad \mathscr{A}^H_\varphi = \grad \mathscr{A}^H + V
\end{equation*}
\noindent for some $V \in \mathfrak{X}(\mathscr{L}_\varphi M)$ and
\begin{equation*}
	\grad \mathscr{A}^H\vert_{(\gamma,\tau)} = \begin{pmatrix}
		\displaystyle J(\dot{\gamma} - \tau X_H(\gamma))\\
		\displaystyle -\int_0^1 H(\gamma(t))dt
	\end{pmatrix}
\end{equation*}
\noindent for all $(\gamma,\tau) \in \mathscr{L}_\varphi M \times \mathbb{R}$. Indeed, this follows from 
\begin{equation}
	d\mathscr{A}^H_\varphi\vert_{(\gamma,\tau)}(X,\eta) = d\mathscr{A}^H\vert_{(\gamma,\tau)}(X,\eta) + (\varphi^*\lambda - \lambda)(X(0))
	\label{eq:twisted_differential}
\end{equation}
\noindent for all $(X,\eta) \in T_\gamma\mathscr{L}_\varphi M \times \mathbb{R}$, where
\begin{equation*}
	d\mathscr{A}^H\vert_{(\gamma,\tau)}(X,\eta) = \int_0^1 d\lambda(X,\dot{\gamma}(t) - \tau X_H(\gamma(t))) dt - \eta \int_0^1 H(\gamma(t))dt.
\end{equation*}

\noindent By assumption, there exists $x \in M$ and $v \in T_x M$ with $(\varphi^*\lambda)_x(v) \neq \lambda_x(v)$. As by assumption $M$ is connected, there exists a smooth path $u \in C^\infty(\intcc[0]{0,1},M)$ from $x$ to $\varphi(x)$. Fix a smooth cutoff function $\beta \in C^\infty(\intcc[0]{0,1},\intcc[0]{0,1})$ such that $\beta = 0$ in a neighbourhood of $0$ and $\beta = 1$ in a neighbourhood of $1$. Then we can extend $u$ by
\begin{equation*}
	\gamma(t) := \varphi^k(u(\beta(t - k))) \qquad \forall t \in \intcc[0]{k,k+1}, k \in \mathbb{Z}.
\end{equation*}
Clearly, $\gamma \in \mathscr{L}_\varphi M$ by construction. Extend $v \in T_{\gamma(0)}M$ to $X_v \in T_\gamma \mathscr{L}_\varphi M$ by
\begin{equation*}
	X_v(t) := (1 - \beta(t - k))P_{0,\beta(t - k)}^{\varphi^k \circ u}(D\varphi^k(v)) + \beta(t - k)P_{1,\beta(t - k)}^{\varphi^k \circ u}(D\varphi^{k + 1}(v)),
\end{equation*}
\noindent for all $t \in \intcc[0]{k,k+1}$ and $k \in \mathbb{Z}$, where $P$ denotes the parallel transport system induced by the Levi--Civita connection associated with the metric $m_J$. Choose a sequence $(\beta_j) \subseteq C^\infty(\mathbb{S}^1,\intcc[0]{0,1})$ with $\beta_j = 1$ on $\intcc[1]{0,\frac{1}{2j}} \cup \intcc[1]{1 - \frac{1}{2j},1}$ and such that $\supp \beta_j \subseteq \intcc[1]{0,\frac{1}{j}} \cup \intcc[1]{1 - \frac{1}{j},1}$ for all $j \in \mathbb{N}$. Using \eqref{eq:twisted_differential} we compute
\begin{align*}
	\langle V, \beta_j X_v \rangle_J &= \langle \grad \mathscr{A}^H_\varphi,\beta_j X_v \rangle_J - \langle \grad \mathscr{A}^H,\beta_j X_v\rangle_J\\
	&= d\mathscr{A}^H_\varphi(\beta_j X_v) - d\mathscr{A}^H(\beta_j X_v)\\
	&= (\varphi^*\lambda - \lambda)(\beta_j(0)X_v(0))\\
	&= (\varphi^*\lambda - \lambda)(v)
\end{align*}
\noindent for all $j \in \mathbb{N}$, implying
\begin{align*}
	(\varphi^*\lambda - \lambda)(v) &= \lim_{j \to \infty} \langle V,\beta_j X_v \rangle_J\\
	&= \lim_{j \to \infty} \int_0^1 d\lambda(JV(t),\beta_j(t)X_v(t))dt\\
	&= 0
\end{align*}
\noindent by dominated convergence.
\end{proof}

\newpage
\section{M-Polyfolds}
\label{ch:polyfolds}
The classical approach for establishing generic transversality results in Floer theories is via a suitable version of the Sard--Smale Theorem \cite[Theorem~A.5.1]{mcduffsalamon:J-holomorphic_curves:2012}. The idea is to represent the moduli space of negative gradient flow lines as the zero set of an appropriate Fredholm section. Unfortunately, this does not work for the moduli space of unparametrised negative gradient flow lines as the reparametrisation action is not smooth. Moreover, the transversality results usually require perturbing the given metric to a generic one. There is a more abstract approach for proving transversality results via polyfold theory. This theory was and is still developed by Hofer--Wysocki--Zehnder \cite{hoferwysockizehnder:polyfolds:2021} primarily having symplectic field theory in mind. Another more algebraic approach to abstract perturbations is via Kuranishi structures developed by Fukaya--Oh--Ohta--Ono \cite{FOOO:kuranishi:2020}.

In the first section we introduce the basic terminology of polyfold theory, namely the notion of scale smoothness on scale Banach spaces.

In the second section we formulate a prototypical result for Morse--Bott homology following the brilliant lecture notes \cite{cieliebak:ga:2018}.

\subsection{Scale Calculus}
\label{sec:scale_calculus}

\begin{definition}[{Scale Structure, \cite[Definition~1.1.1]{hoferwysockizehnder:polyfolds:2021}}]
    A \bld{scale structure} on a Banach space $E$ is a decreasing sequence 
    \begin{equation*}
        E =: E_0 \supseteq E_1 \supseteq E_2 \supseteq \dots
    \end{equation*}
    \noindent of Banach spaces such that the inclusion $E_{k + 1} \hookrightarrow E_k$ is compact for every $k \in \mathbb{N}_0$ and such that $E_\infty := \cap_{k = 0}^\infty E_k$ is dense in every $E_k$. A Banach space $E$ together with a scale structure $(E_k)$ is called a \bld{scale Banach space}.
\end{definition}

\begin{example}[{Shifted Scale Banach Space, \cite[Definition~3.3]{frauenfelderweber:shift:2021}}]
	Let $(E,(E_k))$ be a scale Banach space and $m \in \mathbb{N}_0$. Then $\del[1]{E^k,(E^m_k)}$ is a scale Banach space where $E^m_k := E_{m + k}$ for all $k \in \mathbb{N}_0$.
\end{example}

\begin{example}[{Scale Direct Sum, \cite[Definition~3.4]{frauenfelderweber:shift:2021}}]
    Let $(E,(E_k))$ and $(F,(F_k))$ be scale Banach spaces. Then $(E \oplus F, (E_k \oplus F_k))$ is also a scale Banach space.
\end{example}

The following example underlies Morse and Floer theory.

\begin{example}[{Weighted Sobolev Spaces, \cite[Example~3.9]{frauenfelderweber:shift:2021}}]
    \label{ex:weighted_Sobolev_spaces}
    Fix a monotone cutoff function $\beta \in C^\infty(\mathbb{R},\intcc[0]{-1,1})$ with 
    \begin{equation*}
        \beta(s) = \begin{cases}
        1 & s \geq 1,\\
        -1 & s \leq -1,
        \end{cases}
    \end{equation*}
    \noindent and $\delta > 0$. Define
    \begin{equation*}
        \gamma_\delta \colon \mathbb{R} \to \mathbb{R}, \qquad \gamma_\delta(s) := e^{\delta \beta(s)s}.
    \end{equation*}
    For $p \in \intoo[0]{1,+\infty}$ and $k \in \mathbb{N}_0$ define the \bld{Sobolev spaces with weight $\delta$} by
    \begin{equation*}
        W^{k,p}_\delta(\mathbb{R},\mathbb{R}^n) := \{ u \in W^{k,p}(\mathbb{R},\mathbb{R}^n) : \gamma_\delta u \in W^{k,p}(\mathbb{R},\mathbb{R}^n)\}.
    \end{equation*}
    Choose a strictly increasing sequence $(\delta_k)$ with $\delta_0 = 0$. Then $(E,(E_k))$ is a scale Banach space with
    \begin{equation*}
        E_k := W^{k,p}_{\delta_k}(\mathbb{R},\mathbb{R}^n) \qquad \forall k \in \mathbb{N}_0.
    \end{equation*}
\end{example}

\begin{definition}[{Scale Continuity, \cite[Definition~4.1]{frauenfelderweber:shift:2021}}]
    Let $(E,(E_k))$ and $(F,(F_k))$ be two scale Banach spaces. A map $f \colon E \to F$ is called \bld{scale continuous}, iff $f(E_k) \subseteq F_k$ for all $k \in \mathbb{N}_0$ and the restriction $f\vert_{E_k} \colon E_k \to F_k$ is continuous.
\end{definition}

\begin{definition}[{Tangent Bundle, \cite[Definition~1.1.14]{hoferwysockizehnder:polyfolds:2021}}]
    For a scale Banach space $(E,(E_k))$ define its \bld{tangent bundle} by $TE := E^1 \oplus E$.
\end{definition}

\begin{definition}[{Scale Differentiability, \cite[Definition~4.2]{frauenfelderweber:shift:2021}}]
    A scale continuous map $f \colon E \to F$ between scale Banach spaces $(E,(E_k))$ and $(F,(F_k))$ is called \bld{scale differentiable}, iff for every $x \in E_1$ there exists a bounded linear operator
    \begin{equation*}
        Df(x) \colon E_0 \to F_0,
    \end{equation*}
    \noindent called the \bld{scale differential of $f$}, such that the restriction $f\vert_{E_1} \colon E_1 \to F_0$ is Fr\'echet differentiable with derivative $Df\vert_{E_1}$ and the \bld{tangent map}
    \begin{equation*}
        Tf\colon TE \to TF, \qquad Tf(x,h) := (f(x),Df(x)h) 
    \end{equation*}
    \noindent is scale continuous.
\end{definition}

Using the iterated notion of scale differentiability one can define higher scale regularity. We say that a scale differentiable map is \bld{scale smooth}, iff its tangent map is infinitely scale differentiable. The following example motivated the development of scale calculus.

\begin{example}[{Weighted Sobolev Spaces, \cite[Theorem~6.2]{frauenfelderweber:shift:2021}}]
	\label{ex:time_shift}
    Let $E_k$ be the scale of weighted Sobolev spaces introduced in Example \ref{ex:weighted_Sobolev_spaces} on the scale Banach space $E = L^p(\mathbb{R},\mathbb{R}^n)$. Then the shift map
    \begin{equation*}
        \mathbb{R} \oplus E \to E, \qquad (r,x) \mapsto x(\cdot + r)
    \end{equation*}
    \noindent is scale smooth.
\end{example}

One important property of scale differentiability is that the chain rule remains valid in this general setting.

\begin{proposition}[{Chain Rule, \cite[Theorem~1.3.1]{hoferwysockizehnder:polyfolds:2021}}]
    Uppose that $f \colon E \to F$ and $g \colon F \to G$ are scale differentiable maps for scale Banach spaces $E$, $F$ and $G$. Then the composition $g \circ f$ is scale differentiable with tangent map
    \begin{equation*}
        T(g \circ f) = Tg \circ Tf \colon TE \to TG.
    \end{equation*}
\end{proposition}

The proof of the chain rule heavily relies on the compactness of the embeddings of the scales. Using the chain rule one is able to define the notion of scale manifolds and scale differential geometry in analogy to the finite-dimensional case. For details see \cite[Chapter~5]{cieliebak:ga:2018}. Thus with the theory developed so far, one can make sense of the smooth moduli space of unparametrised negative gradient flow lines. However, for broken negative gradient flow lines one needs an even more general notion, including scale manifolds.

\begin{definition}[{Retraction, \cite[Definition~2.1.1]{hoferwysockizehnder:polyfolds:2021}}]
    Let $E$ be scale Banach space. A \bld{retraction on $E$} is a scale smooth map $r \colon E \to E$ such that $r^2 = r$.
\end{definition}

\begin{remark}
	If $X$ is a smooth Banach manifold, then $\Fix(r) = r(M)$ is a smooth submanifold for every smooth retraction $r \colon X \to X$ by a result of Cartan \cite[Proposition~2.1.2]{hoferwysockizehnder:polyfolds:2021}.
\end{remark}

It is in general not true that the fixed point set of a scale smooth retraction of a scale manifold is a scale submanifold. This lead Hofer--Wysocki--Zehnder to the generalised notion of an \bld{M-polyfold}, where the ``M'' stands for ``manifold flavoured''. Heuristically, an M-polyfold is locally the fixed point set of a scale smooth retraction of a scale smooth manifold. The main aspect for us is that Fredholm theory still is valid in M-polyfolds in some sense. For an extensive treatment see \cite[Part~I]{hoferwysockizehnder:polyfolds:2021}.

\subsection{M-Polyfold Setup for Morse--Bott Homology}
In this section we explain how an M-polyfold Fredholm setup for Morse--Bott homology can be defined. We follow \cite[Section~8.4]{cieliebak:ga:2018}. The essential arguments are contained in \cite[Appendix~A]{frauenfelder:arnold-givental:2004}. We assume the following setup. Let $(M,g)$ be a compact Riemannian manifold and $f \in C^\infty(M)$ a Morse--Bott function. Choose an additional Morse function $h \in C^\infty(\Crit f)$ and a Riemannian metric $g_0$ on $\Crit f$ such that $(h,g_0)$ is a Morse--Smale pair, that is, the stable and unstable manifolds intersect transversally. Pick two connected critical components $C^\pm \subseteq \Crit f$. For 
\begin{equation*}
	0 < \delta < \min \{\abs[0]{\lambda} : \lambda \in \sigma(\Hess_x f)) \setminus \{0\}, x \in \Crit f\}
\end{equation*}
\noindent and a positive strictly increasing sequence $0 < \delta_0 < \delta_1 < \dots < \delta$, consider the Banach manifold $\widetilde{\mathscr{E}}_k(C^-,C^+)$ consisting of all maps $u \in H^{k + 2}_{\delta_k}(\mathbb{R},M)$ converging exponentially to $C^\pm$, that is, there exist $x^\pm \in C^\pm$ as well as  
\begin{equation*}
	\xi^- \in H^{k + 2}_{\delta_k}(\intoc[0]{-\infty,-T},T_{x^-}M) \quad \text{and} \quad \xi^+ \in H^{k + 2}_{\delta_k}(\intco[0]{T,+\infty},T_{x^+}M)
\end{equation*}
\noindent for some $T \in \mathbb{R}$ with
\begin{equation*}
	u(s) = \exp_{x^\pm}(\xi^\pm(s)) \qquad \forall \pm s \geq T.
\end{equation*}
See Figure \ref{fig:scale_manifold_1}. Local charts on $\widetilde{\mathscr{E}}_k(C^-,C^+)$ are constructed by exponential neighbourhoods around smooth paths \cite[Appendix A]{schwarz:morse:1993}. Similar to \cite{alberswysocki:sc:2013}, one can construct scale smooth charts using those exponential neighbourhoods, equipping the Banach manifold $\widetilde{\mathscr{E}}(C_-,C_+) := \widetilde{\mathscr{E}}_0(C_-,C_+)$ with a scale smooth structure. Moreover, there are two natural scale smooth evaluation maps
\begin{equation*}
	\ev^\pm \colon \widetilde{\mathscr{E}}(C^-,C^+) \to \Crit f, \qquad \ev^\pm(u) = x^\pm.
\end{equation*}

\begin{figure}[h!tb]
	\centering
	\includegraphics[width=.7\textwidth]{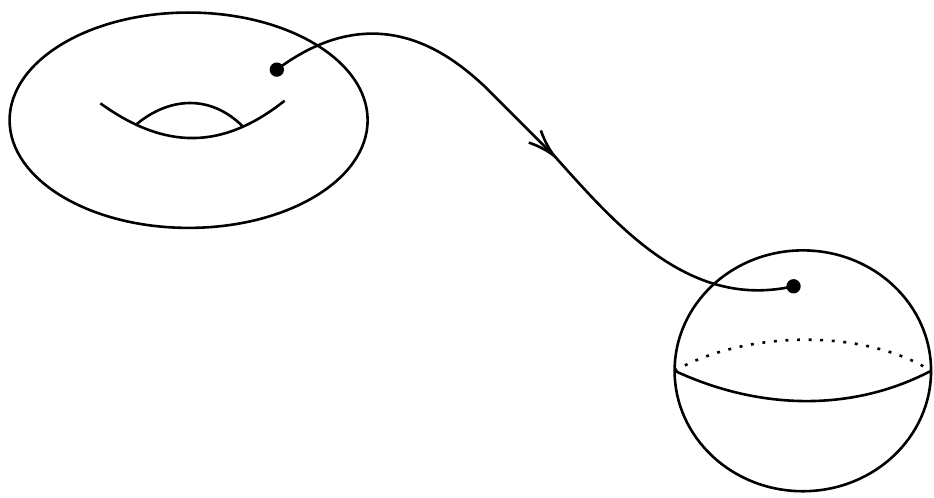}
	\caption{An asymptotically exponential Sobolev path connecting two critical components of a Morse--Bott function.}
	\label{fig:scale_manifold_1}
\end{figure}

\noindent Fix $x^\pm \in C^\pm \cap \Crit h$ and define the stable and unstable manifolds by
\begin{equation*}
	W^\pm(x^\pm) := \cbr[3]{x \in \Crit f : \lim_{s \to \pm\infty}\phi_s(x) = x^\pm},
\end{equation*}
\noindent where $\phi \colon \mathbb{R} \times \Crit f \to \Crit f$ denotes the negative gradient flow of $h$ with respect to the Morse--Smale metric $m_0$. Define
\begin{equation*}
	\widetilde{\mathscr{E}}_k^1(x^-,x^+) := \cbr[1]{u \in \widetilde{\mathscr{E}}_k(C^-,C^+) : \ev^\pm(u) \in W^\pm(x^\pm)}
\end{equation*}
\noindent for all $k \in \mathbb{N}$. See Figure \ref{fig:scale_manifold_2}. Again, $\widetilde{\mathscr{E}}^1(x_-,x_+) := \widetilde{\mathscr{E}}^1_0(x_-,x_+)$ is a scale manifold. For details, see the proof of \cite[Theorem~A.12]{frauenfelder:arnold-givental:2004} and \cite[Theorem~A.14]{frauenfelder:arnold-givental:2004}.

\begin{figure}[h!tb]
	\centering
	\includegraphics[width=.7\textwidth]{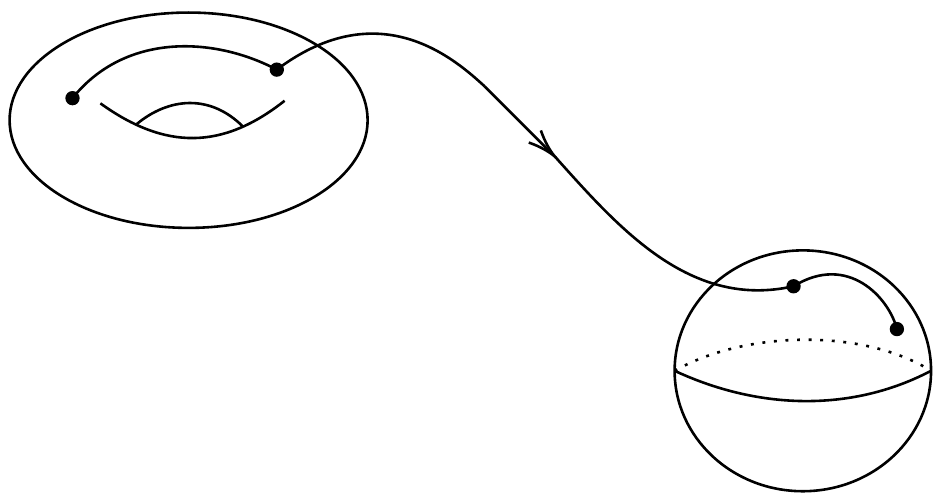}
	\caption{A weighted Sobolev path with one cascade.}
	\label{fig:scale_manifold_2}
\end{figure}

\noindent There exists a canonical scale smooth bundle
\begin{equation*}
	\widetilde{\mathscr{F}}^1(x^-,x^+) \to \widetilde{\mathscr{E}}^1(x^-,x^+),
\end{equation*}
\noindent where the fibre over $u \in \widetilde{\mathscr{E}}_k^1(C^-,C^+)$ is given by
\begin{equation*}
	\widetilde{\mathscr{F}}_k^1(x^-,x^+)\vert_u = H^{k + 1}_{\delta_k}(\mathbb{R},u^*TM)
\end{equation*}
\noindent for all $k \in \mathbb{N}$. Then 
\begin{equation*}
	\widetilde{\partial} \colon \widetilde{\mathscr{E}}^1(x^-,x^+) \to \widetilde{\mathscr{F}}^1(x^-,x^+), \qquad \widetilde{\partial}(u) := \dot{u} + \grad_m f(u)
\end{equation*}
\noindent is a scale Fredholm operator. If $x^- \neq x^+$, the scale smooth reparametrisation action 
\begin{equation*}
	\mathbb{R} \oplus \widetilde{\mathscr{E}}^1(x^-,x^+) \to \widetilde{\mathscr{E}}^1(x^-,x^+), \qquad (r,u) \mapsto u(\cdot + r),
\end{equation*}
\noindent as in Example \ref{ex:time_shift} is free, giving rise to a scale smooth bundle
\begin{equation*}
	\widetilde{\mathscr{F}}^1(x^-,x^+)/\mathbb{R} \to \widetilde{\mathscr{E}}^1(x^-,x^+)/\mathbb{R}
\end{equation*}
\noindent to which $\widetilde{\partial}$ descends to a scale Fredholm section $\partial$. This construction can be generalised to $m$ cascades and thus yields the scale smooth bundle 
\begin{equation*}
	\mathscr{F}^m(x^-,x^+) \to \mathscr{E}^m(x^-,x^+)
\end{equation*}
\noindent of unparametrised paths from $x^-$ to $x^+$ with $m$ cascades. Now we consider broken paths. For $x^0,\dots,x^l \in \Crit h$ with
\begin{equation*}
	f(x^0) \geq f(x^1) \geq \dots \geq f(x^{l - 1}) \geq f(x^l),
\end{equation*}
\noindent we define the space of broken paths from $x^0$ to $x^l$ by
\begin{equation*}
	\mathscr{E}(x^0,\dots,x^l) := \coprod_{(\alpha_0,\dots,\alpha_l) \in \mathbb{N}^{l + 1}_0} \mathscr{E}^{\alpha_0}(x^0,x^1) \times \dots \times \mathscr{E}^{\alpha_l}(x^{l - 1},x^l).
\end{equation*}
Then $\mathscr{E}(x^0,\dots,x^l)$ is a scale manifold over which we have the scale bundle
\begin{equation*}
	\mathscr{F}(x^0,\dots,x^l) := \coprod_{(\alpha_0,\dots,\alpha_l) \in \mathbb{N}^{l + 1}_0} \mathscr{F}^{\alpha_0}(x^0,x^1) \times \dots \times \mathscr{F}^{\alpha_l}(x^{l - 1},x^l)
\end{equation*}
\noindent and the scale Fredholm section
\begin{equation*}
	\partial \times \dots \times \partial \colon \mathscr{E}(x^0,\dots,x^l) \to \mathscr{F}(x^0,\dots,x^l).
\end{equation*}
For $x^\pm \in \Crit h$ with $f(x^-) \geq f(x^+)$ we define the space of broken paths from $x^-$ to $x^+$ by
\begin{equation*}
	X(x^-,x^+) := \coprod_{l \in \mathbb{N}} \coprod_{\substack{x^1,\dots,x^{l - 1} \in \Crit h\\f(x^-) \geq f(x^1) \geq \dots \geq f(x^{l - 1}) \geq f(x^+)}} \mathscr{E}(x^-,x^1,\dots,x^{l - 1},x^+).
\end{equation*}
Similarly, one defines 
\begin{equation*}
	Y(x^-,x^+) := \coprod_{l \in \mathbb{N}} \coprod_{\substack{x^1,\dots,x^{l - 1} \in \Crit h\\f(x^-) \geq f(x^1) \geq \dots \geq f(x^{l - 1}) \geq f(x^+)}} \mathscr{F}(x^-,x^1,\dots,x^{l - 1},x^+).
\end{equation*}
One can show that $X(x^-,x^+)$ carries the natural structure of an M-polyfold and 
\begin{equation*}
	Y(x^-,x^+) \to X(x^-,x^+)
\end{equation*}
\noindent is a bundle. Moreover, the scale Fredholm section above induces a scale Fredholm section 
\begin{equation*}
	\partial \colon X(x^-,x^+) \to Y(x^-,x^+)
\end{equation*}
\noindent of index
\begin{equation*}
	\ind \partial = \ind(x^-) - \ind(x^+) - 1,
\end{equation*}
\noindent where either
\begin{equation*}
	\ind = \ind_f + \ind_h \qquad \text{or} \qquad \ind = -\frac{1}{2}\sgn \Hess f- \frac{1}{2}\sgn \Hess h.
\end{equation*}
For details see \cite[Appendix~A]{cieliebakfrauenfelder:rfh:2009}. One can show that the level set $\partial^{-1}(0)$ of unparametrised broken negative gradient flow lines is compact and that there exists an abstract perturbation of $\partial$, that is, a scale smooth section $\sigma \colon X \to Y$, where
\begin{equation*}
	X := \coprod_{\substack{x^\pm \in \Crit h\\f(x^-) \geq f(x^+)}}X(x^-,x^+) \qquad \text{and} \qquad Y := \coprod_{\substack{x^\pm \in \Crit h\\f(x^-) \geq f(x^+)}}Y(x^-,x^+),
\end{equation*}
\noindent with the property that for all $x^\pm \in \Crit h$ the moduli space
\begin{equation*}
	\mathscr{M}(x^-,x^+) := \{u \in X(x^-,x^+) : (\partial + \sigma)(u) = 0\}
\end{equation*}
\noindent is compact and $\partial + \sigma$ is transverse to the zero section. Hence we arrive at a prototypical result for Morse--Bott homology, compare \cite[Corollary~8.9]{cieliebak:ga:2018}.

\begin{theorem}[M-Polyfold Setup for Morse--Bott Homology]
	\label{thm:MB-polyfold}
	Let $(M,g)$ be a compact Riemannian manifold without boundary and $f \in C^\infty(M)$ a Morse--Bott function. Choose an additional Morse function $h \in C^\infty(\Crit f)$ and an additional Riemannian metric $g_0$ on $\Crit f$ such that $(h,g_0)$ is a Morse--Smale pair. For all critical points $x^\pm \in \Crit h$, the moduli space
	\begin{equation*}
		\mathscr{M}(x^-,x^+) := \{u \in X(x^-,x^+) : (\partial + \sigma)(u) = 0\}
	\end{equation*}
	\noindent is a smooth compact manifold with corners of dimension
	\begin{equation*}
		\dim \mathscr{M}(x^-,x^+) = \ind(x^-) - \ind(x^+) - 1.
	\end{equation*}
	Moreover, there is a canoncial diffeomorphism
	\begin{equation*}
		\partial \mathscr{M}(x^-,x^+) = \mathscr{M}(x^-,x^+) \cap \partial X \cong \coprod_{x \in \Crit h} \mathscr{M}(x^-,x) \times \mathscr{M}(x,x^+).
	\end{equation*}
\end{theorem}

\newpage
\section{Bubbling Analysis}
\label{ch:bubbling_analysis}

In this section we prove the main result about the compactness of the moduli space of negative gradient flow lines of the symplectic action functional.

\begin{definition}[Symplectic Asphericity]
	A connected symplectic manifold $(M,\omega)$ is said to be \bld{symplectically aspherical}, if
	\begin{equation*}
		\int_{\mathbb{S}^2} f^*\omega = 0 \qquad \forall f \in C^\infty(\mathbb{S}^2,M).
	\end{equation*}
\end{definition}

\begin{remark}
	A symplectic manifold $(M,\omega)$ is symplectically aspherical, if and only if $[\omega]\vert_{\pi_2(M)} = 0$, where $[\omega] \in \operatorname{H}_{\dR}^2(M;\mathbb{R})$ denotes the cohomology class of the closed form $\omega$.
\end{remark}

\begin{theorem}[Bubbling]
	\label{thm:bubbling}
	Let $(M,\omega)$ be a compact symplectically aspherical symplectic manifold and let $(u_k)$ be a sequence of negative gradient flow lines of the symplectic action functional $\mathcal{A}_H$ for some $H \in C^\infty(M \times \mathbb{T})$ with uniformly bounded energy
\begin{equation*}
	E_J(u_k) := \int_{-\infty}^{+\infty}\norm[0]{\partial_s u_k}^2_J
\end{equation*}
\noindent for some, and hence every, $\omega$-compatible almost complex structure $J$. Then the derivatives of $(u_k)$ are uniformly bounded.
\end{theorem}

The main idea of the proof is to assume that the derivatives $(Du_k)$ explode and then to construct a nonconstant $J$-holomorphic sphere. Indeed, assume that there exists a sequence $(s_k,t_k)$ in  $\mathbb{R} \times \mathbb{T}$ such that
\begin{equation*}
    \lim_{k \to \infty}\norm[0]{\partial_su_k(s_k,t_k)} \to +\infty.
\end{equation*}
Then we rescale the sequence $(u_k)$, see Figure \ref{fig:magnifying_glass}. Set  
\begin{equation*}
    m_k := \norm[0]{\partial_su_k(s_k,t_k)} \qquad \text{and} \qquad v_k(\sigma,\tau) := u_k\del[3]{\frac{\sigma}{m_k} + s_k,\frac{\tau}{m_k} + t_k}
\end{equation*}
\noindent for all $(\sigma,\tau) \in \mathbb{C}$.

\noindent Then we compute
\begin{align*}
    m_k \partial_\sigma v_k(\sigma,\tau) &= \partial_s u_k\del[3]{\frac{\sigma}{m_k} + s_k,\frac{\tau}{m_k} + t_k},\\
    m_k \partial_\tau v_k(\sigma,\tau) &= \partial_t u_k\del[3]{\frac{\sigma}{m_k} + s_k,\frac{\tau}{m_k} + t_k}.
\end{align*}
In particular $\norm[0]{\partial_\sigma v_k(0,0)} = 1$ for all $k \in \mathbb{N}$. Moreover, every $v_k$ solves 
\begin{equation*}
    \partial_\sigma v_k(\sigma,\tau) + J\partial_\tau v_k(\sigma,\tau) = \frac{1}{m_k}JX_{H_{\frac{\tau}{m_k}+t_k}}(v_k(\sigma,\tau)) \qquad \forall (\sigma,\tau) \in \mathbb{C}
\end{equation*}
\noindent as $u_k$ satisfies the Floer equation \eqref{eq:Floer_equation}. If there exists  $v_\infty \in C^\infty(\mathbb{C},M)$ such that
\begin{equation*}
	v_k \xrightarrow{C^\infty_{\loc}} v_\infty, \qquad k \to \infty,
\end{equation*}
\noindent modulo subsequences, then $v_\infty$ satisfies 
\begin{equation*}
	\partial_\sigma v_\infty(\sigma,\tau) + J\partial_\tau v_\infty(\sigma,\tau) = 0 \qquad \forall (\sigma,\tau) \in \mathbb{C}.
\end{equation*}
Consequently, $v_\infty$ is a $J$-holomorphic plane. Using the assumption that the energy of the sequence $(u_k)$ is uniformly bounded, one can extend $v_\infty$ to a $J$-holomorphic sphere $v \in C^\infty(\mathbb{S}^2,M)$ such that $v\vert_\mathbb{C} = v_\infty$ via the identification $\mathbb{S}^2 \cong \mathbb{C} \cup \{\infty\}$. 

\begin{figure}[h!tb]
    \centering
    \includegraphics[width=\textwidth]{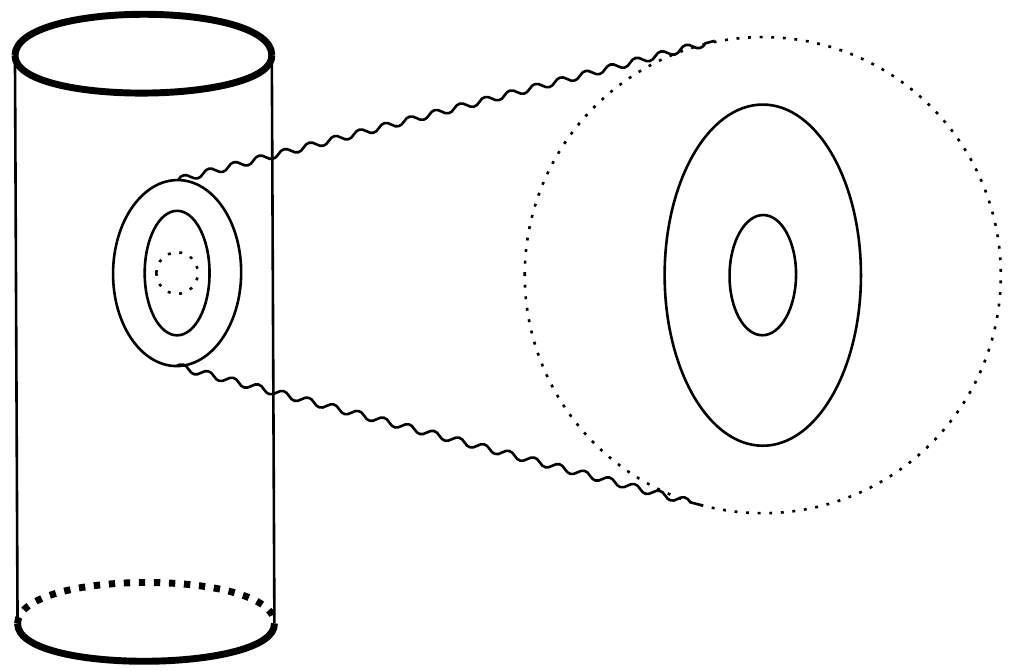}
    \caption{Looking at the sequence $(u_k)$ of negative gradient flow lines of the symplectic action functional with a magnifying glass via rescaling.}
    \label{fig:magnifying_glass}
\end{figure}

\subsection{Rescaling}

\begin{lemma}[{Hofer, \cite[Lemma~4.4.4]{abbashofer:cg:2019}}]
	\label{lem:Hofer}
	Let $(X,d)$ be a complete metric space and
	\begin{equation*}
		f \colon X \to \intco[0]{0,+\infty}
	\end{equation*}
	\noindent continuous. Given $\varepsilon > 0$ and $x \in X$, there exist $0 < \delta \leq \varepsilon$ and $y \in X$ such that
	\begin{equation*}
		 d(x,y) \leq 2\delta, \qquad \delta f(y) \geq \varepsilon f(x), \qquad \text{and} \qquad  \sup_{z \in \overline{B}_\delta(y)} f(z) \leq 2f(y).
	\end{equation*}
\end{lemma}

\begin{lemma}[Rescaling]
    \label{lem:rescaling}
	Let $(u_k)$ be a sequence of negative gradient flow lines of the symplectic action functional $\mathcal{A}_H$ for some $H \in C^\infty(M \times \mathbb{T})$ and $J$ some $\omega$-compatible almost complex structure. Then there exists a sequence $(R_k)$ in $\intoo[0]{0,+\infty}$ with $R_k \to +\infty$ and a sequence $v_k \in C^\infty(B_{R_k}(0),M)$ such that
    \begin{equation*}
        \lim_{k \to \infty}\norm[0]{\partial_\sigma v_k + J\partial_\tau v_k}_\infty = 0 \qquad \text{and} \qquad 1 \leq \norm[0]{\partial_\sigma v_k}_\infty \leq 2.
    \end{equation*}
\end{lemma}

\begin{proof}
    Abbreviate $z_k := (s_k,t_k) \in \mathbb{R} \times \mathbb{T}$. Apply Hofer's Lemma \ref{lem:Hofer} with
	\begin{equation*}
		f \colon \mathbb{R} \times \mathbb{T} \to \intco[0]{0,+\infty}, \qquad f(s,t) := \norm[0]{\partial_s u_k(s,t)},
	\end{equation*}
	\noindent as well as
    \begin{equation*}
		\varepsilon_k := \frac{1}{\sqrt{m_k}}, \qquad \text{and} \qquad x_k := z_k.
    \end{equation*}
    Thus there exists a nearby sequence $z'_k = (s'_k,t'_k) \in \mathbb{R} \times \mathbb{T}$ and $0 < \delta_k \leq \varepsilon_k$ with
    \begin{equation*}
        \delta_k\norm[0]{\partial_s u_k(z'_k)} \geq \varepsilon_k \norm[0]{\partial_s u_k(z_k)} = \sqrt{m}_k
    \end{equation*}
	\noindent and
	    \begin{equation*}
        \sup_{z \in B_{\delta_k}(z_k')} \norm[0]{\partial_s u_k(z)}\leq 2\norm[0]{\partial_s u_k(z'_k)}.
    \end{equation*}
    Rescale $m_k' := \norm[0]{\partial_su_k(z_k')}$, set $R_k := \delta_k m_k'$ and
    \begin{equation*}
        v_k(\sigma,\tau) := u_k\del[3]{\frac{\sigma}{m_k'} + s_k',\frac{\tau}{m_k'} + t_k'} \qquad \forall (\sigma,\tau) \in B_{R_k}(0).
    \end{equation*}
	As $u_k$ is a solution to the Floer equation \eqref{eq:Floer_equation}, $v_k$ satisfies the equation
    \begin{equation}
		\label{eq:rescaled_Floer_equation}
		\partial_\sigma v_k(\sigma,\tau) + J\partial_\tau v_k(\sigma,\tau) = \frac{1}{m_k'}JX_{H_{\frac{\tau}{m_k'}+t_k'}}(v_k(\sigma,\tau))
	\end{equation}   
    \noindent for all $(\sigma,\tau) \in B_{R_k}(0)$. As $m'_k \geq m_k \to +\infty$, we conclude
    \begin{equation*}
        \lim_{k \to \infty}\norm[0]{\partial_\sigma v_k + J\partial_\tau v_k}_\infty = 0.
    \end{equation*}  
    Finally, we compute
    \begin{align*}
        \norm[0]{\partial_\sigma v_k}_\infty &= \sup_{z \in B_{R_k}(0)}\norm[0]{\partial_\sigma v_k(z)}\\
        &= \frac{1}{m_k'}\sup_{z \in B_{\delta_k}(z_k')} \norm[0]{\partial_s u_k(z)}\\
        &\leq \frac{2}{m_k'}\norm[0]{\partial_s u_k(z'_k)}\\
        &= 2,
    \end{align*}
    \noindent and 
    \begin{equation*}
		m_k'\norm[0]{\partial_\sigma v_k(0,0)} = \norm[0]{\partial_s u_k(z_k')} = m_k' .
    \end{equation*}
    This concludes the proof of the lemma.
\end{proof}

\begin{lemma}
    \label{lem:continuous_holomorphic_plane}
    Let $(v_k)$ be the sequence constructed in the Rescaling Lemma \ref{lem:rescaling}. Then there exists $v_\infty \in C^0(\mathbb{C},M)$ such that
    \begin{equation*}
        v_k \xrightarrow{C^0_{\mathrm{loc}}} v_\infty, \qquad k \to \infty,
    \end{equation*}
    \noindent up to a subsequence.
\end{lemma}

\begin{proof}
    By the Rescaling Lemma \ref{lem:rescaling} and compactness of $M$ there exists $C > 0$ such that
	\begin{equation*}
		\norm[0]{\partial_\tau v_k}_\infty \leq C \qquad \forall k \in \mathbb{N}.
    \end{equation*}
    Fix $R > 0$ and choose $K_R \in \mathbb{N}$ such that $R_k \geq R$ for all $k \geq K_R$ and consider the restrictions $v_k\vert_{B_R(0)}$ for all $k \geq K_R$. As the derivatives of $v_k$ are uniformly bounded, we conclude that $v_k\vert_{B_R(0)}$ is of Sobolev class 
    \begin{equation*}
        W^{1,\infty}(B_R(0)) := \{f \in W^{1,\infty}(B_R(0),\mathbb{R}^{4n + 1}) : f(B_R(0)) \subseteq M\}
    \end{equation*}
    \noindent for all $k \geq K_k$ where we consider $M^{2n} \hookrightarrow \mathbb{R}^{4n + 1}$ via the Whitney embedding Theorem. Thus by Morrey's inequality \cite[Corollary~9.14]{brezis:fa:2011}, every $v_k\vert_{B_R(0)}$ is of H\"older class $C^{0,1}(B_R(0))$ , and hence $v_k\vert_{B_R(0)}$ is equicontinuous. As $M$ is compact, Ascolis Theorem implies the existence of $v^R \in C^0(B_R(0),M)$ with
    \begin{equation*}
        v_k\vert_{B_R(0)} \xrightarrow{C^0} v^R, \qquad k \to \infty,
    \end{equation*}
    \noindent up to a subsequence. Choose a subsequence $(k_j^1)$ with $k_j^1 \geq K_1$ for all $j \in \mathbb{N}$ and such that there exists $v^1 \in C^0(B_1(0),M)$ with
    \begin{equation*}
        v_{k_j^1}\vert_{B_1(0)} \xrightarrow{C^0} v^1, \qquad j \to \infty.
    \end{equation*}
    \noindent Inductively, choose a subsequence $(k_j^{\mu + 1})$ of $(k_j^\mu)$ for all $\mu \in \mathbb{N}$ with $k_j^{\mu + 1} \geq K_{\mu + 1}$ for all $j \in \mathbb{N}$ and such that there exists $v^{\mu + 1} \in C^0(B_{\mu + 1}(0),M)$ with
    \begin{equation*}
        v_{k^{\mu + 1}_j}\vert_{B_{\mu + 1}(0)} \xrightarrow{C^0} v^{\mu + 1}, \qquad j \to \infty.
    \end{equation*}
    Finally, taking the diagonal subsequence yields
    \begin{equation*}
		v_{k^j_j} \xrightarrow{C^0_{\loc}} v_\infty \in C^0(\mathbb{C},M), \qquad j \to \infty.
    \end{equation*}
	{}
\end{proof}

\subsection{Elliptic Bootstrapping}
\label{sec:elliptic_bootstrapping}

\begin{lemma}[Elliptic Bootstrapping]
    \label{lem:holomorphic_plane}
    Denote by $v_\infty \in C^0(\mathbb{C},M)$ the map constructed in Lemma \ref{lem:continuous_holomorphic_plane}. Then $v_\infty \in C^\infty(\mathbb{C},M)$ is a nonconstant $J$-holomorphic plane.
\end{lemma}

\begin{proof}
    Fix sequences $(z_j) \subseteq \mathbb{C}$ and $r_j \subseteq \intoo[0]{0,+\infty}$ such that
	\begin{itemize}
		\item $v_\infty\vert_{B_{4r_j}(z_j)}$ is contained in a chart $U_j$ of $M$ for all $j \in \mathbb{N}$.
		\item $\bigcup_{j \in \mathbb{N}}B_{r_j}(z_j) = \mathbb{C}$.
		\item for all $j \in \mathbb{N}$ there exists $K_j$ such that $v_k\vert_{B_{2r_j}(z_j)} \subseteq U_j$ for all $k \geq K_j$.
	\end{itemize}
	Fix smooth bump functions $\beta_j \in C^\infty(\mathbb{C},\intcc[0]{0,1})$ for $\overline{B}_{r_j}(z_j)$ supported in $B_{2r_j}(z_j)$ and define
    \begin{equation*}
        \overline{v}_k^j := \beta_j v_k \in C^\infty_c(\mathbb{C},\mathbb{R}^{2n})
    \end{equation*}
    \noindent for all $k,j \in \mathbb{N}$. We compute
    \begin{equation}
    \label{eq:derivative}
        \Delta \overline{v}^j_k = (\Delta \beta_j)v_k + 2 \partial_\sigma\beta_j\partial_\sigma v_k + 2 \partial_\tau \beta_j \partial_\tau v_k + \beta_j\Delta v_k. 
    \end{equation}
    To compute $\Delta v_k$, we differentiate \eqref{eq:rescaled_Floer_equation} in charts. Applying $\partial_\sigma$ we get that 
    \begin{equation*}
        \partial_\sigma^2 v_k(\sigma,\tau) + DJ\partial_\sigma v_k(\sigma,\tau)\partial_\tau v_k(\sigma,\tau) + J\partial_\sigma\partial_\tau v_k(\sigma,\tau)
    \end{equation*} 
    \noindent is equal to
    \begin{equation*}
        \frac{1}{m_k'} D\grad_J H_{\frac{\tau}{m_k'}+t_k'}(v_k(\sigma,\tau)) \partial_\sigma v_k(\sigma,\tau),
    \end{equation*}
	\noindent where $\grad_J H_t$ denotes the gradient of $H_t$ with respect to the Riemannian metric induced by $J$, that is, the Riemannian metric $\omega(J\cdot,\cdot)$. Applying $\partial_\tau$ to \eqref{eq:rescaled_Floer_equation} we get that 
    \begin{equation*}
        \partial_\tau\partial_\sigma v_k(\sigma,\tau) + J\partial_\tau^2 v_k(\sigma,\tau) + DJ(\partial_\tau v_k(\sigma,\tau))^2
    \end{equation*} 
    \noindent is equal to
    \begin{equation*}
        \frac{1}{m_k'} D\grad_J H_{\frac{\tau}{m_k'}+t_k'}(v_k(\sigma,\tau)) \partial_\tau v_k(\sigma,\tau) + \frac{1}{m_k'} \partial_\tau\grad_J H_{\frac{\tau}{m_k'}+t_k'}(v_k(\sigma,\tau)).
    \end{equation*}
    Hence 
    \begin{equation*}
        \Delta v_k = P(\partial_\sigma v_k,\partial_\tau v_k),
    \end{equation*}
    \noindent where $P$ is a polynomial of degree $2$ with $C^\infty$-coefficients and so
    \begin{equation*}
        \Delta \overline{v}_k^j = P(\partial_\sigma \overline{v}_k^j,\partial \overline{v}_k^j)
    \end{equation*}
    \noindent by \eqref{eq:derivative}. As the derivatives $\partial_\sigma v_k$ and $\partial_\tau v_k$ are uniformly bounded, we conclude that $\Delta \overline{v}_k^j$ is also uniformly bounded. Thus $\norm[0]{\Delta \overline{v}_k^j}_{L^p}$ is uniformly bounded for all exponents $p \in \intcc[0]{1,+\infty}$ as $\Delta \overline{v}_k^j$ is compactly supported by construction. Therefore, the Calderon--Zygmund inequality \cite[Corollary~B.2.7]{mcduffsalamon:J-holomorphic_curves:2012} implies that $\norm[0]{\overline{v}_k^j}_{W^{2,p}}$ is uniformly bounded for all $1 < p < +\infty$. In particular,
    \begin{equation*}
        \norm[0]{\partial_\sigma\overline{v}_k^j}_{W^{1,p}} \qquad \text{and} \qquad \norm[0]{\partial_\tau\overline{v}_k^j}_{W^{1,p}}
    \end{equation*}
    \noindent are uniformly bounded. Thus again by Morrey's inequality, $\partial_\sigma\overline{v}_k^j$ and $\partial_\tau\overline{v}_k^j$ belong to the H\"older class $C^{0,\alpha}$ for all $0 < \alpha < 1$. By Ascolis Theorem, $\partial_\sigma\overline{v}_k^j$ and $\partial_\tau\overline{v}_k^j$ admit convergent subsequences and 
    \begin{equation*}
        \overline{v}_\infty^j := \beta_j v_\infty \in C^1_c(\mathbb{C},\mathbb{R}^{2n})
    \end{equation*}
    \noindent satisfies
    \begin{equation*}
        v_k\vert_{B_{r_j}(z_j)} \xrightarrow{C^1} \overline{v}_\infty^j
    \end{equation*}
	\noindent for all $j \in \mathbb{N}$.
    Hence we have showed that $v_\infty \in C^1(\mathbb{C},M)$ and
    \begin{equation*}
        v_k \xrightarrow{C^1_{\loc}} v_\infty 
    \end{equation*}
    \noindent up to subsequences. This procedure can be generalised to higher derivatives and is referred to as \bld{elliptic bootstrapping}. Note that taking higher order derivatives makes only sense locally in a chart if we do not refer to a particular connection. Thus we get $v_\infty \in C^\infty(\mathbb{C},M)$ and
    \begin{equation*}
        v_k \xrightarrow{C^\infty_{\loc}} v_\infty, \qquad k \to \infty,
    \end{equation*}
    \noindent up to subsequences. Finally, $v_\infty$ is nonconstant as
	    \begin{equation*}
        \norm[0]{\partial_\sigma v_\infty}_\infty = \lim_{k \to \infty}\norm[0]{\partial_\sigma v_k}_\infty \geq 1 
    \end{equation*}
	\noindent and satisfies
    \begin{equation*}
        \partial_\sigma v_\infty + J\partial_\tau v_\infty = 0
    \end{equation*}
	\noindent by the Rescaling Lemma \ref{lem:rescaling}.
\end{proof}

\subsection{Removal of Singularities}
Consider the conformal diffeomorphism
\begin{equation*}
    \varphi \colon \mathbb{C} \setminus \{0\} \to \mathbb{R} \times \mathbb{T}, \qquad \varphi(re^{2\pi i\theta}) := \Log(re^{2\pi i\theta}),
\end{equation*}
\noindent and define
\begin{equation*}
    w \colon \mathbb{R} \times \mathbb{T} \to M, \qquad w := v_\infty \circ \varphi^{-1},
\end{equation*}
\noindent where $v_\infty \in C^\infty(\mathbb{C},M)$ denotes the nonconstant $J$-holomorphic plane constructed in Lemma \ref{lem:continuous_holomorphic_plane}. Then 
\begin{equation*}
    \lim_{s \to -\infty} w(s,\cdot) = v_\infty(0)
\end{equation*}
\noindent and $w$ is of finite energy. Indeed, we compute
\begin{equation*}
	E_J(w) = \int_{\mathbb{R} \times \mathbb{T}} w^* \omega = \int_{-\infty}^{+\infty} \int_0^1\omega(J\partial_s w(s,t),\partial_s w(s,t))dtds \leq \sup_{k \in \mathbb{N}} E_J(u_k)
\end{equation*}
\noindent by \cite[Lemma~2.2.1]{mcduffsalamon:J-holomorphic_curves:2012}. Moreover, $w$ is a negative gradient flow line of the symplectic area functional
\begin{equation*}
	\mathscr{A} \colon \Lambda M \to \mathbb{R}, \qquad \mathscr{A}(\gamma) := \int_{\mathbb{D}} \overline{\gamma}^* \omega.
\end{equation*}

\begin{lemma}
	\label{lem:constant_asymptotic}
	For every sequence $(r_k) \subseteq \intoo[0]{0,+\infty}$ with $r_k \to +\infty$ as $k \to \infty$ define a sequence $(w_k)$ by
	\begin{equation*}
		w_k(s,t) := w(s + r_k,t) \qquad \forall (s,t) \in \mathbb{R} \times \mathbb{T}.
	\end{equation*}
	Then there exists a point $w_\infty \in M$ such that
	\begin{equation*}
		w_k \xrightarrow{C^\infty_{\loc}} w_\infty, \qquad k \to \infty,
	\end{equation*}
	\noindent up to a subsequence.
\end{lemma}

\begin{proof}
	By \cite[p.~89--90]{mcduffsalamon:J-holomorphic_curves:2012} there exists a constant $a \geq 0$ such that 
	\begin{equation*}
		\Delta e \geq - ae^2,
	\end{equation*}
	\noindent where 
	\begin{equation*}
		e \colon \mathbb{R} \times \mathbb{T} \to \intco[0]{0,+\infty}, \qquad e(s,t) := \norm[0]{\partial_s w(s,t)}^2
	\end{equation*}
	\noindent denotes the energy density. As $E_J(w) < +\infty$, there exists $R > 0$ such that
\begin{equation*}
    \int_R^{+\infty}\int_0^1 e(s,t)dt ds \leq \min \cbr[3]{\frac{\pi}{8a},1}
\end{equation*}
\noindent and
\begin{equation*}
    \int_{-\infty}^{-R}\int_0^1 e(s,t)dt ds \leq \min \cbr[3]{\frac{\pi}{8a},1}.
\end{equation*}
Let $z = (s,t) \in \mathbb{R} \times \mathbb{T}$ such that $\abs[0]{s}\geq R + 1$. Then 
\begin{equation*}
    \int_{B_1(z)}e < \frac{\pi}{8a}.
\end{equation*}
By \cite[Lemma~4.3.3]{mcduffsalamon:J-holomorphic_curves:2012} we conclude
\begin{equation*}
    e(z) \leq \frac{8}{\pi}\int_{B_1(z)} e \leq \frac{8}{\pi}.
\end{equation*}
Consequently, $e$ is uniformly bounded, as $e\vert_{\intcc[0]{-R-1,R+1}\times \mathbb{T}}$ is uniformly bounded by continuity of $e$. Hence $\norm[0]{\partial_s w}$ is uniformly bounded by definition of $e$ and $\norm[0]{\partial_t w}$ is uniformly bounded as $w$ is a $J$-holomorphic curve. By an eliptic bootrstrapping argument as in Lemma \ref{lem:holomorphic_plane}, we conclude
\begin{equation*}
        w_k \xrightarrow{C^\infty_{\loc}} w_\infty, \qquad k \to \infty,
\end{equation*}
\noindent up to a subsequence, where $w_\infty$ is a negative gradient flow line of the symplectic  area functional $\mathscr{A}$. Then $w_\infty$ is constant. Indeed, assume that $w_\infty$ is not constant. Then there exists $s < s'$ such that
\begin{equation*}
    \varepsilon := \mathscr{A}(w_\infty(s)) - \mathscr{A}(w_\infty(s')) > 0.
\end{equation*}
Moreover, there exists $K \in \mathbb{N}$ such that 
\begin{equation*}
    \mathscr{A}(w_k(s)) - \mathscr{A}(w_k(s')) = \mathscr{A}(w(s + r_k)) - \mathscr{A}(w(s' + r_k)) \geq \frac{\varepsilon}{2}
\end{equation*}
\noindent for all $k \geq K$. Define a subsequence $(r_{k_j})$ of $(r_k)$ recursively by
\begin{equation*}
    k_0 := K \qquad \text{and} \qquad k_j := \min\{l \in \mathbb{N} : s' + r_{k_{j - 1}} \leq s + r_l\}.
\end{equation*}
This works as $r_k \to +\infty$ as $k \to \infty$. Fix $l \in \mathbb{N}$. Then we compute
\allowdisplaybreaks
\begin{align*}
    E(w) &= \sup_{s \in \mathbb{R}} \mathscr{A}(w(s)) - \inf_{s \in \mathbb{R}}\mathscr{A}(w(s))\\
    &\geq \mathscr{A}(w(s + r_{k_0})) - \mathscr{A}(w(s + r_{k_l}))\\
    &= \sum_{\nu = 1}^l \del[1]{\mathscr{A}(w(s + r_{k_{\nu - 1}})) - \mathscr{A}(w(s + r_{k_{\nu}}))}\\
    &\geq \sum_{\nu = 1}^l \del[1]{\mathscr{A}(w(s + r_{k_{\nu - 1}})) - \mathscr{A}(w(s' + r_{k_{\nu - 1}}))}\\
    &\geq \frac{\varepsilon l}{2}.
\end{align*}
As $l \in \mathbb{N}$ was arbitrary, we conlcude that $E(w) = +\infty$.
\end{proof}

\begin{proof}[Proof of Theorem \ref{thm:bubbling}]
	As $M$ is compact, there exists a finite open cover $U_1,\dots,U_m$ of $M$ such that $\overline{U}_j$ is contained in a Darboux chart. Define
	\begin{equation*}
		\mathscr{U} := \{\gamma \in \Lambda M : \gamma(\mathbb{T}) \subseteq U_j \text{ for some $j = 1,\dots,m$}\}.
	\end{equation*}

	\emph{Step 1: There exists $s_0 \in \mathbb{R}$ such that $w(s) \in \mathscr{U}$ for all $s \geq s_0$.} Assume that there exists a sequence $(r_k) \subseteq \intoo[0]{0,+\infty}$ with $r_k \to +\infty$ as $k \to \infty$ and $w(r_k) \notin \mathscr{U}$ for all $k \in \mathbb{N}$. By Lemma \ref{lem:constant_asymptotic}, there exists $w_\infty \in M$ such that
	\begin{equation*}
		w_k \xrightarrow{C^\infty_{\loc}} w_\infty, \qquad k \to \infty,
	\end{equation*}
	\noindent up to a subsequence. But $w_\infty \in U_j$ for some $j$ and thus there exists $K \in \mathbb{N}$ such that $w(r_k,t) \in U_j$ for all $t \in \mathbb{T}$ and $k \geq K$. Consequently, $w(r_k) \in \mathscr{U}$ for all $k \geq K$.

	\emph{Step 2: There exist constants $C > 0$ and $\kappa > 0$ such that}
	\begin{equation*}
		d_{L^2}(w(s),w_\infty) \leq Ce^{-\kappa s} \qquad \forall s \geq s_0,
	\end{equation*}
	\noindent \emph{where $w_\infty$ is the limit of $w(k) = w_k(0)$ as $k \to \infty$ up to a subsequence $(k_j)$.} By \cite[Proposition~6.4]{frauenfeldernicholls:morse:2020} and \cite[Lemma~6.3]{frauenfeldernicholls:morse:2020}, we have the action-energy inequality
	\begin{equation*}
		\abs[0]{\mathscr{A}(\gamma)} \leq C_0 \norm[0]{\grad_J\mathscr{A}(\gamma)}_J^2 \qquad \forall \gamma \in \mathscr{U}
	\end{equation*}
	\noindent for some constant $C_0 > 0$ as the symplectic area functional $\mathscr{A}$ is Morse--Bott. See Remark \ref{rem:MB-lemma}. Let $s \geq  s_0$. Choose $j_0$ such that $k_{j_0} > s$. We estimate
	\begin{equation*}
		\mathscr{A}(w(s)) > \mathscr{A}(w(k_{j_0})) \geq \lim_{j \to \infty} \mathscr{A}(w(k_j)) = \lim_{j \to \infty} \mathscr{A}(w_{k_j}(0)) = \mathscr{A}(w_\infty) = 0.
	\end{equation*}
	Let $s_2 > s_1 > s_0$. Using \cite[Lemma~6.5]{frauenfeldernicholls:morse:2020} and \cite[Lemma~6.6]{frauenfeldernicholls:morse:2020} we compute
	\allowdisplaybreaks
	\begin{align*}
		d(w(s_1),w(s_2)) &\leq \frac{2}{\sqrt{C_0}}\del[1]{\sqrt{\mathscr{A}(w(s_1))} - \sqrt{\mathscr{A}(w(s_2))}}\\
		&\leq \frac{2}{\sqrt{C_0}}\sqrt{\mathscr{A}(w(s_1))}\\
		&\leq \frac{2}{\sqrt{C_0}}\sqrt{\mathscr{A}(w(s_0))}e^{\frac{1}{2C_0}(s_0 - s_1)}
	\end{align*}
	Choose $j_1$ such that $k_{j_1} \geq s_1$. Then for all $j \geq j_1$, we have that
	\begin{equation*}
		d_{L^2}(w(s_1),w(k_j)) \leq Ce^{-\kappa s_1},
	\end{equation*}
	\noindent where
	\begin{equation*}
		C := \frac{2}{\sqrt{C_0}}\sqrt{\mathscr{A}(w(s_0))}e^{\frac{1}{2C_0}s_0} \qquad \text{and} \qquad \kappa := \frac{1}{2C_0}. 
	\end{equation*}
	Thus
	\begin{equation*}
		d_{L^2}(w(s_1),w_\infty) = \lim_{j \to \infty}d_{L^2}(w(s_1),w(k_j)) \leq Ce^{-\kappa s_1}.
	\end{equation*}

	\emph{Step 3: There exists a unique point $w_\infty \in M$ such that}
	\begin{equation*}
		w(s) \xrightarrow{C^\infty} w_\infty, \qquad s \to +\infty.
	\end{equation*}
	Uniqueness follows immediately from the exponential decay established in Step 2. Indeed, a priori, $w_\infty$ does depend on the choice of subsequence $(k_j)$. Let $w_\infty' \in M$ be the limit of a different subsequence. Then we compute
	\begin{equation*}
		d_{L^2}(w_\infty,w_\infty') \leq d_{L^2}(w(s),w_\infty) + d_{L^2}(w(s),w_\infty') \leq Ce^{-\kappa s} + C'e^{-\kappa's} \to 0
	\end{equation*}
	\noindent as $s \to +\infty$. That the limit can also be taken with respect to the $C^\infty$-topology follows from \cite[Proposition~6.5.3]{audindamian:fh:2014}.
\end{proof}

\begin{corollary}
	\label{cor:bubbling_rfh}
	Let $(M,\omega)$ be a symplectically aspherical symplectic manifold and let $(u_k,\tau_k)$ be a sequence of solutions $(u_k,\tau_k) \in C^\infty(\mathbb{R} \times \mathbb{T}, M) \times C^\infty(\mathbb{R},\mathbb{R})$ of
	\begin{equation*}
		\begin{cases}\partial_s u_k(s,t) + J(\partial_t u_k(s,t) - \tau_k(s)X_H(u_k(s,t))) = 0,\\
		\displaystyle \partial_s\tau_k(s) = \int_0^1 H(u_k(s,t))dt,
	\end{cases}
	\end{equation*}
	\noindent for all $s \in \mathbb{R}$ and $k \in \mathbb{N}$ for some $H \in C^\infty(M)$, with uniformly bounded energy. If there exists a compact subset $K \subseteq M \times \mathbb{R}$ such that
	\begin{equation*}
		\im(u_k,\tau_k) \subseteq K \qquad \forall k \in \mathbb{N},
	\end{equation*}
	\noindent then the derivatives of $(u_k)$ are uniformly bounded.
\end{corollary}

\begin{proof}
	Crucial is the observation that under the above assumptions, the limit of \eqref{eq:rescaled_Floer_equation} is still a $J$-holomorphic curve.  
\end{proof}

\printbibliography

\end{document}